\documentclass[11pt,a4paper]{amsart}

\usepackage{amsmath,amscd}
\usepackage{amssymb}
\usepackage{amsthm}

\usepackage{graphicx}
\usepackage{todonotes}

\usepackage{mathrsfs}
\usepackage[ocgcolorlinks, linkcolor=blue]{hyperref}

\usepackage[alphabetic,abbrev,nobysame]{amsrefs} \AtBeginDocument{\def\MR#1{}} 

\usepackage{calc}
             {\begin{list}{\arabic{enumi}.}{\usecounter{enumi}%
              \setlength{\labelsep}{0.5em}%
              \settowidth{\labelwidth}{(\arabic{enumi})}%
              \setlength{\leftmargin}{\labelwidth+\labelsep}}}%
             {\end{list}}

\usepackage{tikz-cd}

\usepackage{comment}

\DeclareRobustCommand{\SkipTocEntry}[5]{}

\usepackage{xcolor}
\definecolor{LOcolor}{RGB}{150,100,0}

\DeclareMathOperator{\WF}{WF}

\newtheorem{Theorem}{Theorem}[section]

\newtheorem{Lemma}[Theorem]{Lemma}
\newtheorem{Corollary}[Theorem]{Corollary}

\newtheorem{Proposition}[Theorem]{Proposition}

\theoremstyle{definition}
\newtheorem{Definition}{Definition}[section]

\newtheorem{Remark}[Theorem]{Remark}

\numberwithin{equation}{section}

\newcommand{\mR}{\mathbb{R}}                    
\newcommand{\mC}{\mathbb{C}}                    
\newcommand{\abs}[1]{\lvert #1 \rvert}          

\newcommand{\ol}[1]{\overline{#1}}

\newcommand{\re}{\mathrm{Re}}
\newcommand{\im}{\mathrm{Im}}

\newcommand{\mG}{\mathcal{G}}

\newcommand{\supp}{\mathrm{supp}}

\newcommand{\eps}{\varepsilon}

\newcommand{\mi}{M^{\mathrm{int}}}
\newcommand{\mx}{X}

\newcommand{\p}{\partial}

\definecolor{darkred}{rgb}{1,0,0} 
\definecolor{darkgreen}{rgb}{0,0.6,0}
\definecolor{darkblue}{rgb}{0,0,0.8}

\newcounter{sidenote}

\begin{document}

\title{A general support theorem for analytic double fibration transforms}

\author[M. Mazzucchelli]{Marco Mazzucchelli}
\address{Marco Mazzucchelli\newline\indent CNRS and {\'E}cole Normale Sup{\'e}rieure de Lyon, France}
\email{marco.mazzucchelli@ens-lyon.fr}

\author[M. Salo]{Mikko Salo}
\address{Mikko Salo\newline\indent Department of Mathematics and Statistics, University of Jyv\"askyl\"a, Finland}
\email{mikko.j.salo@jyu.fi}

\author[L. Tzou]{Leo Tzou}
\address{Leo Tzou\newline\indent Korteweg-de Vries Institute, University of Amsterdam, Netherlands}
\email{leo.tzou@gmail.com}




\begin{abstract}
We develop a systematic approach for resolving the analytic wave front set for a class of integral geometry transforms appearing in various tomography problems. Combined with microlocal analytic continuation, this leads to uniqueness and support theorems for analytic integral transforms which are in the microlocal double fibration framework introduced by Guillemin.

For the case of ray transforms, we show that the double fibration setup has a concrete interpretation in terms of curve families obtained by projecting integral curves of a fixed vector field on some fiber bundle down to the base. This setup includes geodesic X-ray type transforms, null bicharacteristic ray transforms and transforms related to real principal type systems. We also study transforms integrating over submanifolds of any codimension, and give geometric characterizations for the Bolker condition required for recovering singularities.

Our approach is based on a general result related to recovering the analytic wave front set of a function from its transform given by a suitable analytic elliptic Fourier integral operator. This approach extends and unifies a number of previous works. We use wave packet transforms to extrapolate the geometric features of wave front set propagation for such operators when their canonical relation satisfies the Bolker condition.
\end{abstract}

\maketitle

\section{Introduction} \label{sec_intro}

This article was originally motivated by the recent work \cite{OSSU} that studied inverse problems for real principal type differential operators. It was shown there that certain null bicharacteristic ray transforms of the coefficient functions can be determined from boundary measurements. A special case is the light ray transform studied e.g.\ in \cites{Plamen_lightray, LOSU, FIO}, but in the general case the invertibility of these transforms was left as an open question in \cite{OSSU}.

In this work we observe that the transforms mentioned above fall under the double fibration approach to integral geometry as in \cites{Helgason1966, GGS, GuilleminSternberg, Guillemin1985}. We will thus study the much more general class of transforms arising from a double fibration satisfying a Bolker condition. In particular, this includes ray transforms arising from the flow of a vector field on some fiber bundle over the base manifold, and generalized Radon transforms where one has access to integrals over a suitable family of $k$-dimensional submanifolds (the case of ray transforms corresponds to $k=1$). 

The recovery of $C^{\infty}$ singularities for such transforms $R$ follows from the standard clean intersection calculus of Fourier integral operators (FIOs) applied to the normal operator $R^* R$ as discussed in \cites{GuilleminSternberg, Guillemin1985}. As the main contribution of this work we show that if the underlying structures are analytic, then certain analytic singularities of a function can be recovered from the analytic singularities of its transform. The strength of this approach is that unlike in the $C^{\infty}$ case, recovery of analytic singularities combined with microlocal analytic continuation (also known as the microlocal Holmgren or Sato-Kawai-Kashiwara theorem, \cite[Theorem 8.5.6']{Hormander}) leads to local uniqueness results, global uniqueness results under foliation conditions, and Helgason type support theorems for our transforms. Such results were known for the weighted Radon transform \cite{BomanQuinto1987}, geodesic X-ray transform \cites{StefanovUhlmann, FrigyikStefanovUhlmann, Krishnan2009}, light ray transform \cite{Plamen_lightray}, and certain generalized Radon transforms \cite{HomanZhou2016}. Our approach hinges on a general result that allows one to recover the analytic wave front set of a function $f$ from the knowledge of $Rf$, where $R$ is a special type of analytic elliptic FIO. We hope that this result could be useful in other contexts as well.

\subsection{Ray transforms} \label{subsec_ray}

We first discuss our results for ray transforms. Let $M$ be a manifold with smooth boundary, let $\Xi$ be a fiber bundle over $M$ whose fibers are manifolds without boundary, and let $\pi: \Xi \to M$ be the projection (then $\p \Xi = \pi^{-1}(\p M)$). Let $Y$ be a vector field in $\Xi$ with flow $\Phi_t$, and consider its ``horizontal projection'' 
\[
Y^h = d\pi \circ Y: \Xi \to TM.
\]
Given any $z \in \Xi$ we consider the maximally extended integral curve $\gamma_z: [-\tau_-(z), \tau_+(z)] \to \Xi$, $\gamma_{z}(t) = \Phi_t(z)$, so that $\dot{\gamma}_{z}(t) = Y(\gamma_{z}(t))$ and $\gamma_{z}(0) = z$. In general one may have $\tau_{\pm}(z) = +\infty$, but we will only consider curves that are not trapped and satisfy $\tau_{\pm}(z) < \infty$. There are corresponding base space curves $x_{z}(t) = \pi(\gamma_{z}(t))$ with tangent vector 
\[
\dot{x}_{z}(t) = Y^h(\Phi_t(z))).
\]

We consider the integral geometry problem where we know the integrals of a function $f$ over all curves $x_z(t)$ for $z \in \mG$, where $\mG$ is a submanifold of $\Xi$ describing the set of admissible curves. We can also include a weight $\kappa \in C^{\infty}(\mG \times M)$ that is nowhere vanishing. The related integral transform is 
\[
R: C^{\infty}_c(M^{\mathrm{int}}) \to C^{\infty}(\mG), \ \ R f(z) = \int_{-\tau_-(z)}^{\tau_+(z)} \kappa(z, x_z(t)) f(x_{z}(t)) \,dt.
\]
The inverse problem is to determine the function $f$, or some properties of $f$, from the knowledge of $R f$.

We require that $\mG$ is chosen so that $Y$ is never tangent to $\mG$ (to remove redundancy in parametrizing the curves) and that $\dim(\mG) \geq \dim(M)$ (to ensure that the inverse problem is not formally underdetermined). These conditions hold e.g.\ if $\mG$ is an open subset of the inward pointing boundary $\p_+ \Xi$ of $\Xi$, defined by 
\[
\p_+ \Xi = \{ z \in \Xi \mid \pi(z) \in \p M, \ Y^h(z) \text{ is transverse to $\partial M$, and $Y^h(z)$ points inside $M$} \}.
\]
In general $\mG$ could also be lower dimensional submanifold of $\Xi$, which corresponds to a ray transform with restricted data.

Examples of operators that arise as $R$ include the following:

\begin{itemize}
\item 
The geodesic X-ray transform, with $\Xi = SM$ and $Y = X_g$ where $SM$ is the unit sphere bundle and $X_g$ is the geodesic vector field on $SM$ (see e.g.\ \cite{PSU_book}). One can take $\mG = \p_+ SM$ (full data case) or $\mG$ could be a strict open subset of $\p_+ SM$ (partial data case). One could also take $\mG$ to be a lower dimensional submanifold of $\p_+ SM$, one example being the X-ray transform with sources on a curve \cite{FinchLanUhlmann}.
\item 
The transform in \cite{FrigyikStefanovUhlmann} involving a family of curves $\gamma$ satisfying an equation $\ddot{\gamma} = G(\gamma, \dot{\gamma})$, with the choice $\Xi = TM$ and $Y(x,y) = y^j \p_{x_j} + G^j(x,y) \p_{y_j}$. These curve families include e.g.\ Riemannian geodesics, magnetic geodesics \cite{DPSU}, and thermostats \cite{AZ17}.
\item 
The null bicharacteristic ray transform for a differential operator $P$ with real principal symbol $p$, where $\Xi = p^{-1}(0) \subset T^* M$ and $Y$ is the Hamilton vector field $H_p$ \cite{OSSU}. If $P$ is the Lorentzian wave operator, this is just the light ray transform. This case also includes ray transforms for real principal type systems such as fully anisotropic elasticity, where $p$ is a simple eigenvalue of the Christoffel matrix \cite{dehoop2021determination}.
\end{itemize}

The first question is to determine whether $R$ above is covered by the double fibration framework. This will be true under certain minimal conditions, one of which is phrased in terms of \emph{variation fields} that reduce to Jacobi fields in the case of geodesic curves.

\begin{Definition} \label{def_variations}
Let $z_s$ be a smooth curve through $z_0$ in $\mG$ with $\partial_s z_s|_{s=0} = w$. The family $(x_{z_s})$ is called a \emph{variation} of the curve $x_{z_0}$. The associated \emph{variation field} $J_w: (-\tau_-(z_0), \tau_+(z_0)) \to T M$ along $x_{z_0}$ is given by 
\[
J_w(t) = \p_s x_{z_s}(t)|_{s=0} = d\pi(d \Phi_t(w)).
\]
The space of all possible directions of variation at $x_{z_0}(t)$ is 
\[
V_{z_0}(t) = \{ J_w(t) \mid w \in T_{z_0} \mG \}.
\]
\end{Definition}

\begin{Definition} \label{def_dfrt}
We say that $R$ is a \emph{double fibration ray transform} if the curves $x_z(t)$ satisfy the following:
\begin{enumerate}
\item[(i)]
(No tangential intersections) $x_z(t) \in \mi$ for $z \in \mG$ and $t \in (-\tau_-(z),\tau_+(z))$;
\item[(ii)] 
(No self-intersections) $t \mapsto x_z(t)$ is injective for all $z \in \mG$;
\item[(iii)] 
(No singular points) $\dot{x}_z(t) \neq 0$ for all $z \in \mG$ and all $t \in (-\tau_-(z),\tau_+(z))$;
\item[(iv)]
(Nontrapping) $\tau_{\pm}(z) < \infty$ for all $z \in \mG$.
\end{enumerate}
If $\dim(\mG) \leq \dim(\Xi)-2$, we also require that 
\begin{enumerate}
\item[(v)]
(Enough variations) $V_z(t) + \mR Y^h = T_{x_z(t)} M$ for all $z \in \mG$ and all $t \in (-\tau_-(z),\tau_+(z))$.
\end{enumerate}
\end{Definition}

Conditions (i)--(iv) are easy to visualize. Condition (v) ensures that the curve family is so large that one can vary $x_z(t)$ within this family in such a way that the infinitesimal directions of variation cover all possible directions at any point. By Lemma \ref{lem: double fibration condition}, (v) is automatically satisfied if $\mG$ is an open subset of $\p_+ \Xi$, or if $\dim(\mG) = \dim(\Xi)-1$, so this condition is only relevant if $\mG$ is a lower dimensional manifold. It is proved in Section \ref{sec_ray} that assuming (i)--(v), $R$ is indeed a transform arising from a double fibration in the sense of Definition \ref{def_doublefibration_general}. Conversely, under certain orientability assumptions, any transform arising from a double fibration with one dimensional left fibers is related to a vector field $Y$ as above though the integral curves might be periodic. We remark that conditions (i) and (ii) are not always necessary and could be removed in certain cases by extension procedures (see e.g.\ \cites{Dairbekov, StefanovUhlmann}), but we assume them for simplicity.

As observed in \cite{GuilleminSternberg}, and also stated in Theorem \ref{thm_r_fio} below, a double fibration ray transform $R$ is an FIO and its canonical relation is given by 
\[
C = \{ (z, -A(z,t) \eta, x_z(t), \eta) \mid z \in \mG, \ t \in (-\tau_-(z),\tau_+(z)), \ \eta\in T_{x_z(t)}^*M \setminus 0, \ \eta \perp \dot{x}_z(t) \},
\]
where $A(z,t)\eta = (\pi \circ \Phi_t|_{\mG})^* \eta$. Here the suggestive notation $\eta \perp \dot{x}_z(t)$ means that $\eta(\dot{x}_z(t)) = 0$.  It follows from basic properties of FIOs \cite[vol.\ IV]{Hormander} that $R$ is well defined on $\mathcal{E}'(\mi)$, has certain Sobolev mapping properties, and that one has the wave front set relation 
\[
\WF(Rf) \subset C \circ \WF(f).
\]
Thus any singularity of $Rf$ at $(z,\zeta)$ must come from some singularity of $f$ at $(x, \eta)$, where $x = x_z(t)$ and $\zeta = -A(z,t)\eta$ for some $t$.

In order to recover singularities of $f$ from those of $Rf$ we need a converse statement, and this is where the Bolker condition \cites{GuilleminSternberg, Guillemin1985} comes in. The Bolker condition states that the left projection $\pi_L: C \to T^* \mG \setminus 0$ is an injective immersion. For $R$ as above we can give a geometric characterization. This requires two notations. Below, we say that $\eta$ annihilates $V$ if $\eta(v) = 0$ for all $v \in V$.

\begin{Definition} \label{def_hx_vts_ray}
Given $x \in M$ let $H_{x}$ be the submanifold of $\mG$ defined by 
\[
H_{x} = \{ z \in \mG \mid x_z(t) = x \text{ for some $t$} \}.
\]
Given $z \in \mG$ and $t, s \in (-\tau_-(z), \tau_+(z))$ let $V_{z}(t, s)$ be the space of all possible directions of variation at $x_z(t)$ among variations that keep $x_z(s)$ fixed in the normal direction, i.e.\ 
\[
V_{z}(t, s) = \{ J_w(t) \mid w \in T_{z} \mG, \ J_w(s) \parallel \dot{x}_z(s) \}.
\]
\end{Definition}

\begin{Definition} \label{def_bolker_ray}
We say that the Bolker condition holds at $(z,\zeta,x,\eta) \in C$ with $x = x_{z}(t)$ if 
\begin{enumerate}
\item[(vi)] 
$V_z(t,s)$ is not annihilated by $\eta$ for any $s \neq t$; and 
\item[(vii)] 
$d(Y^h(x,\xi(\,\cdot\,)))(T_{z} H_{x})$ is not annihilated by $\eta$, where $(x,\xi(\tilde{z}))$ is the unique point over $x$ on the integral curve $\gamma_{\tilde{z}}$ through $\tilde{z} \in H_{x}$. 
\end{enumerate}
\end{Definition}

Condition (vi) means that for any $s \neq t$, there is a variation of $x_z$ that keeps $x_z(s)$ fixed in the normal direction, and the infinitesimal variation of $x_z(t)$ has a nonzero component in the $\eta$ direction. A sufficient condition for (vi) to hold in the case of the geodesic X-ray or null bicharacteristic ray transform is that the curve $x_z$ has no conjugate points (for geodesics this holds for arbitrary $\eta$, but for null bicharacteristic rays one needs a structural condition $\eta \nparallel \xi$). Condition (vii) means that the infinitesimal variations of $\dot{x}_{z}(t)$, among all curves that go through $x_{z}(t)$, should include a vector with nonzero component in the $\eta$ direction. When $\mG$ is an open subset of $\p_+ \Xi$, condition (vii) always holds for the geodesic X-ray transform, and it also holds for the null bicharacteristic ray transform provided that the Hessian $\nabla_{\xi}^2 p$ is nondegenerate and $\eta \nparallel \xi$. See Section \ref{sec_ray} for detailed statements.

We say that a double fibration transform $R$ is \emph{analytic} if all related objects ($M$, $\Xi$, $\mG$, $Y$ and $\kappa$) are analytic (i.e.\ real-analytic). Let $\WF_a$ denote the analytic wave front set \cite{sjostrand}. We can now state our first main theorem.

\begin{Theorem}
\label{global elliptic regularity intro ray}
Let $R$ be an analytic double fibration ray transform, and assume that the Bolker condition holds at $(z,\zeta,x,\eta) \in C$. Then for any $f\in {\mathcal E}'(M^{\mathrm{int}})$, we have 
\[
(  z,  \zeta) \notin \WF_a(Rf) \implies (  x,  \eta) \notin \WF_a(f).
\]
\end{Theorem}

\subsection{General double fibration transforms}

We will now consider the case of transforms that integrate over $k$-dimensional submanifolds where $1 \leq k \leq \dim(X)-1$. Here and below we write $\mx$ instead of $\mi$, to emphasize that $\mx$ can be any manifold without boundary. The results will be stated in the abstract double fibration setting.

\begin{Definition} \label{def_doublefibration_general}
Let $\mx$ and $\mG$ be smooth oriented manifolds with no boundary, and let $Z$ be an oriented embedded submanifold of $ \mG \times \mx$ with projections $\pi_{\mG}: Z \to \mG$, $\pi_{\mx}: Z \to \mx$ and with 
\begin{equation*} 
\dim(\mG)+\dim(\mx) > {\rm dim}(Z) > {\rm dim}(\mG) \geq {\rm dim}(\mx).
\end{equation*}

\begin{enumerate}
\item[(a)]
We say that $Z$ is a \emph{double fibration} if $\pi_{\mG}$ and $\pi_{\mx}$ are submersions. Then the sets 
\[
G_z := \pi_{\mx}(\pi_{\mG}^{-1}(z)), \qquad H_x := \pi_{\mG}(\pi_{\mx}^{-1}(x))
\]
are embedded submanifolds, with orientation forms $\omega_{G_z}$ and $\omega_{H_x}$ induced by some fixed orientation forms on $\mG, \mx, Z$ and by the submersions $\pi_{\mG}$ and $\pi_{\mx}$ (see Lemma \ref{lemma_orientation}). One has $\dim(G_z) = \dim(Z)-\dim(\mG)$.
\item[(b)] 
Let $\kappa \in C^{\infty}(\mG \times \mx)$ be nowhere vanishing. The linear map $R: C^{\infty}_c(\mx) \to C^{\infty}(\mG)$ given by 
\[
Rf(z) = \int_{G_z} \kappa(z,x) f(x) \,d\omega_{G_z}(x), \qquad z \in \mG,
\]
is called a \emph{double fibration transform}.
\item[(c)]
We say that $R$ is an \emph{analytic double fibration transform} if $R$ is as in (b) and all the related objects (i.e.\ $\mG$, $\mx$, $Z$, $\omega_{G_z}$, $\kappa$) are analytic.
\item[(d)]
If $R$ is a double fibration transform, we consider its canonical relation 
\[
C := (N^* Z \setminus 0)' = \{ (z,\zeta,x,\eta) \mid (z,\zeta,x,-\eta) \in N^* Z \setminus 0 \}
\]
and define the left and right projections $\pi_L: C \to T^* \mG \setminus 0 $ and $\pi_R: C \to T^* \mx \setminus 0$.
\item[(e)] 
We say that the \emph{Bolker condition} holds at $(z,\zeta,x,\eta) \in C$ if $\pi_L^{-1}(z,\zeta) = \{ (z,\zeta,x,\eta) \}$ and if $d\pi_L|_{(z,\zeta,x,\eta)}$ is injective.
\end{enumerate}
\end{Definition}

One has the related diagrams 
\begin{equation} \label{df_diagrams}
\begin{tabular}{cc}
\begin{tikzcd}
& Z \arrow[dl,swap,"\pi_{\mG}"] \arrow[dr,"\pi_{\mx}"] \\
\mG && \mx
\end{tikzcd} \qquad & \qquad
\begin{tikzcd}
& C \arrow[dl,swap,"\pi_L"] \arrow[dr,"\pi_R"] \\
T^* \mG \setminus 0 && T^* \mx \setminus 0
\end{tikzcd}
\end{tabular}
\end{equation}

As discussed in Section \ref{subsec_ray}, any double fibration ray transform according to Definition \ref{def_dfrt} is a double fibration transform, with $Z$ given by 
\[
Z = \{ (z, x_z(t)) \mid z \in \mG,\ t \in (-\tau_-(z), \tau_+(z)) \}.
\]
The $k$-plane transform that encodes the integrals of $f \in C^{\infty}_c(\mR^n)$ over all $k$-planes in $\mR^n$, where $1 \leq k \leq n-1$, is another example. Generalized Radon transforms studied in \cites{Beylkin, HomanZhou2016} where one integrates over hypersurfaces $\{ b(x,\theta) = s \}$, for a defining function $b$ satisfying suitable conditions, are also in this class. We extend this setup to any codimension in Section \ref{subsubsec_k}.

For double fibration ray transforms, the Bolker condition holds in the sense of Definition \ref{def_doublefibration_general} iff conditions (vi) and (vii) in Definition \ref{def_bolker_ray} hold. Similar geometric conditions that characterize the validity of the Bolker condition in the general case are proved in Section \ref{injectivity of piL}.

We then have the following analytic regularity result that generalizes Theorem \ref{global elliptic regularity intro ray}.

\begin{Theorem}
\label{global elliptic regularity intro}
Let $R$ be an analytic double fibration transform, and assume that the Bolker condition holds at $(z, \zeta,  x, \eta)\in C$. Then for any $f\in {\mathcal E}'(\mx)$, we have 
\[
(  z,  \zeta) \notin \WF_a(Rf) \implies (  x,   \eta) \notin \WF_a(f).
\]
\end{Theorem}

\subsection{Applications}

As the first application of the above results, we state a local uniqueness result for analytic double fibration transforms. This is an immediate consequence of Theorem \ref{global elliptic regularity intro} and microlocal analytic continuation.

\begin{Theorem} \label{thm_local_uniqueness}
Suppose that $\Sigma$ is a $C^2$ hypersurface in $\mx$ such that $(x_0,\eta_0) \in N^* \Sigma$ and that $f \in \mathcal{E}'(\mx)$ vanishes on one side of $\Sigma$ near $x_0$. Let $R$ be an analytic double fibration transform, and assume that the Bolker condition holds at $(z_0, \zeta_0,  x_0, \eta_0)\in C$. If $Rf(z) = 0$ for $z$ near $z_0$, then $f = 0$ near $x_0$.
\end{Theorem}

Let us give a few applications to specific transforms.

\subsubsection{Geodesic X-ray transform}

We begin with the geodesic X-ray transform, defined on a compact oriented Riemannian manifold $(M,g)$ with smooth boundary. We say that $(M,g)$ is strictly convex if the second fundamental form of $\p M$ is positive definite, and that $(M,g)$ is nontrapping if any geodesic reaches the boundary in finite time. The geodesic X-ray transform encodes the integrals of a function $f \in C(M)$ over all maximal geodesics. For an account of various properties of this transform see \cite{PSU_book}.

It was conjectured in \cite{PSU_tensor} that the geodesic X-ray transform of a compact strictly convex nontrapping manifold is injective (and that the corresponding transform on tensor fields is solenoidal injective). This is known if additionally there are no conjugate points, i.e.\ $(M,g)$ is simple \cite{Mukhometov}, or if $\dim(M) \geq 3$ and $(M,g)$ can be foliated by strictly convex hypersurfaces \cite{UhlmannVasy}. There are also results in the case of hyperbolic trapping \cite{Guillarmou_trapping} and microlocal results in the presence of conjugate points \cites{MSU, HolmanUhlmann}. When $\dim(M)=2$ we can verify the conjecture provided that $M$ and $g$ are real-analytic.

\begin{Theorem} \label{thm_xray_twodim}
Let $(M,g)$ be a compact, strictly convex, nontrapping two-dimensional Riemannian manifold. If $M$ and $g$ are analytic, then the geodesic X-ray transform, possibly with an analytic nowhere vanishing weight, acting on functions on $(M,g)$ is injective.
\end{Theorem}

This theorem follows by combining the strictly convex foliation given by \cite{BeteluGulliverLittman} with a layer stripping argument based on the local uniqueness result in Theorem \ref{thm_local_uniqueness}. In the present case of geodesics the local uniqueness result is already contained in \cite{StefanovUhlmann}.

\subsubsection{Null bicharacteristic ray transform}

Let us next consider the null bicharacteristic ray transform. Let $M$ be a manifold with smooth boundary, let $p \in C^{\infty}(T^* M \setminus 0)$ be homogeneous in $\xi$, and let $\Xi = p^{-1}(0) \subset T^* M \setminus 0$. We assume that 
\[
\text{$\nabla_{\xi} p \neq 0$ and $\nabla_{\xi}^2 p$ is nondegenerate everywhere on $\Xi$.}
\]
The basic example is the light ray transform where $p(x,\xi) = g_x(\xi,\xi)$ with $g$ a Lorentzian metric on $M$. More generally, if $g$ is a pseudo-Riemannian metric, one obtains the pseudo-Riemannian geodesic X-ray transform studied in \cite{Ilmavirta} for product manifolds. The geodesic X-ray transform of a Riemannian metric $g$, for geodesics on the cosphere bundle, also arises in this way by taking $p(x,\xi) = \abs{\xi}_g^2-1$ (this $p$ is not homogeneous but it can still be included, see Section \ref{subsect: null bich transform}).

Let $Y = H_p$ be the Hamilton vector field of $p$ on $\Xi$, and for any $z \in \Xi$ let $x_z(t) = \pi(\Phi_t(z))$ be the projection to $M$ of the null bicharacteristic through $z$. We assume that the admissible curves are parametrized by an open subset $\mG$ of $\p_+ \Xi$, and that the following conditions hold for all $z \in \mG$:
\begin{itemize}
\item 
(No tangential intersections) $x_z(t) \in \mi$ for $t \in (0, \tau_+(z))$;
\item 
(No self-intersections) $t \mapsto x_z(t)$ is injective; and
\item 
(Nontrapping) $\tau_{+}(z) < \infty$.
\end{itemize}
If $\kappa \in C^{\infty}(\mG \times M)$ is nowhere vanishing, the weighted null bicharacteristic ray transform is given by 
\[
Rf(z) = \int_{0}^{\tau_+(z)} \kappa(z, x_z(t)) f(x_z(t)) \,dt.
\]
Under the above conditions, $R$ is a double fibration ray transform. Hence, if all the related quantities are analytic, one has a local uniqueness theorem as in Theorem \ref{thm_local_uniqueness}.

\begin{Theorem} \label{thm_intro_support_nbrt}
Assume that $M$, $p$, and $\kappa$ are analytic. Suppose $x_0 = x_{z_0}(t_0)$ where the curve $x_{z_0}$ has no conjugate points, and $\eta_0\in T^*_{x_0} M$ satisfies $\eta_0 \perp \dot{x}_{z_0}(t_0)$ and $\eta_0 \nparallel \Phi_{t_0}(z_0)$. If $f\in \mathcal{E}'(M^{\mathrm{int}})$ satisfies $Rf(z) = 0$ for $z \in \mG$ near $z_0$, then 
\[
(x_0,\eta_0) \notin \WF_a(f).
\]
Moreover, if $\Sigma$ is a hypersurface in $\mi$ such that $x_{z_0}$ is tangent to $\Sigma$ at $x_0$, and if $f = 0$ on one side of $\Sigma$ near $x_0$, then $f$ vanishes near $x_0$.
\end{Theorem}

For the light ray transform, it was proved in \cite{LOSU} that one can always recover $C^{\infty}$ spacelike singularities in the absence of conjugate points and that one can never recover any timelike singularities. In this case requiring that $\eta_0 \perp \dot{x}_{z_0}(t_0)$ and $\eta_0$ is not null (which implies $\eta_0 \nparallel \Phi_{t_0}(z_0)$) is equivalent with saying that $\eta_0$ is spacelike, so Theorem \ref{thm_intro_support_nbrt} implies the recovery of spacelike analytic singularities for the light ray transform. It does not give information about lightlike singularities since for lightlike $\eta_0$ the conditions $\eta_0 \perp \dot{x}_{z_0}(t_0)$ and $\eta_0 \nparallel \Phi_{t_0}(z_0)$ are incompatible. Results on lightlike $C^{\infty}$ singularities may be found in \cites{Wang_lightray1, Wang_lightray2}, and the light ray transform of wave equation solutions is studied in \cite{VasyWang}. Support theorems for analytic light ray transforms are given in \cite{Plamen_lightray}.

We will next formulate a more global uniqueness theorem, based on a layer stripping argument that iterates Theorem \ref{thm_intro_support_nbrt}. In words, the result is as follows: if smooth hypersurfaces $\Gamma_s$ for $s \in (0,1]$ foliate a subset $U$ of $\mi$, if $\Gamma_1 \cap \supp(f) = \emptyset$, and if for any $y \in \Gamma_s$ with conormal $\nu_y$ of $\Gamma_s$ at $y$ there is $z = z_y$ and a nontrapped curve $x_{z}$ with $x_z(t_y) = y$ and $\dot{x}_z(t_y) \perp \nu_y$ such that $x_z$ has  no conjugate points, self-intersections or tangential intersections and $\Phi_{t_y}(z) \nparallel \nu_y$, then $f = 0$ in $U$ provided that $Rf = 0$ in some open set containing all $z_y$ arising in this way.

A simple way to rule out conjugate points, self-intersections and tangential intersections is to assume that the hypersurfaces $\Gamma_s$ are strictly convex in a suitable sense. Then it is possible to only use short curves $x_z$ in the layer stripping argument. Short curves do not have conjugate points under the nondegeneracy assumption for $\nabla_{\xi}^2 p$ (see Lemma \ref{lem_ncp_short}). We will formulate our uniqueness theorem in a way similar to \cite{Plamen_lightray} for the light ray transform. We first need two definitions.

\begin{Definition}
For any $x \in \mi$, the set of \emph{potentially visible singularities} is 
\[
{\rm PVS}(x) = \bigcup_{\xi \in \Xi_x} \left[ Y^h(x,\xi)^{\perp} \setminus (\mR \xi) \right] \subset T_x^* M \setminus 0.
\]
If $U \subset \mi$, a smooth function $F: U \to \mR$ is said to be \emph{strictly pseudoconvex} with respect to $p$ in $U$ if 
\[
\{p, \{p, F\} \} > 0 \text{ whenever } p = \{p, F \} = 0 \text{ in $T^* U \setminus 0$}.
\]
\end{Definition}

The condition $\eta \in {\rm PVS}(x)$ means that there is some curve $x_z$ with $x_z(t_0) = x$, $\dot{x}_z(t_0) \perp \eta$ and $\Phi_{t_0}(z) \nparallel \eta$. Thus for $\eta \in {\rm PVS}(x)$ it is possible to recover a singularity of $f$ at $(x,\eta)$ from the knowledge of $Rf$ if additionally a no conjugate points condition holds, whereas for $\eta \notin  \cup_{\xi \in \Xi_x} Y^h(x,\xi)^{\perp}$ the singularity at $(x,\eta)$ can never be recovered. The notion of strict pseudoconvexity appears in connection with pseudoconvexity conditions in unique continuation problems \cite[Chapter 28]{Hormander}. Geometrically it means that curves $x_z$ tangent to a level set $F^{-1}(s)$ stay strictly on one side of the level set near the point of tangency.

The uniqueness theorem can now be stated as follows.

\begin{Theorem} \label{thm_bichar_foliation}
Assume that $M$, $p$, and $\kappa$ are analytic and let $f \in \mathcal{E}'(\mi)$. Let $F$ be a smooth real-valued function near $\supp(f)$, write $\Gamma_s = F^{-1}(s)$ and $U = \cup_{s \in (0,1]} \Gamma_s$. Suppose that $F$ satisfies the following conditions:
\begin{enumerate}
\item[(a)]
$\supp(f) \cap \Gamma_1 = \emptyset$.
\item[(b)]
$dF(x) \in {\rm PVS}(x)$ for all $x \in U$.
\item[(c)]
$F$ is strictly pseudoconvex with respect to $p$ in $U$.
\end{enumerate}
If $Rf(z) = 0$ in some open set containing all $z$ corresponding to curves that are tangent to some $\Gamma_s$ for $s \in (0,1]$, then $f=0$ in $U$.
\end{Theorem}

For the light ray transform this result was proved in \cite{Plamen_lightray}. In that case, the condition $dF(x) \in {\rm PVS}(x)$ translates to the level sets $\Gamma_s$ being timelike (i.e.\ their conormals are spacelike), and some examples of pseudoconvex foliations are given in \cite{Plamen_lightray}. In other cases, such as $p(x,\xi) = \abs{\xi}_g^2-1$ (Riemannian geodesic X-ray transform) or $p(x,\xi) = g_x(\xi,\xi)$ for a pseudo-Riemannian metric of signature $(n_1, n_2)$ with $n_1, n_2 \geq 2$, any nonzero covector is in ${\rm PVS}(x)$ and condition (b) always holds. We also stress that the pseudoconvexity condition in (c) could be weakened to a condition that some (possibly long) curves tangent to $\Gamma_s$ have no conjugate points.

When combined with \cite[Theorem 1.2]{OSSU}, the above theorem gives a uniqueness result for Calder\'on type problems for real principal type operators.

\begin{Theorem}
\label{calderon problem intro}
Let $M$ be a compact analytic manifold with smooth boundary, let $P$ be a real principal type differential operator of order $m \geq 2$ on $M$ with analytic coefficients, and let $V_1, V_2 \in C^\infty(M)$ with $V_1 = V_2$ to infinite order on $\partial M$. If the Cauchy data sets for $P+V_1$ and $P+ V_2$ agree, and if the geometric assumptions of Theorem \ref{thm_bichar_foliation} hold with $f = V_1-V_2$, then $V_1 = V_2$ in $U$.
\end{Theorem}

\subsubsection{Transforms over codimension $k$ submanifolds} \label{subsubsec_k}

The previous examples were in the setting of ray transforms. Next we state a result for generalized Radon transforms that integrate over codimension $k$ submanifolds. This extends earlier results for the Euclidean Radon transform with analytic weights \cite{BomanQuinto1987} and for Radon type transforms over hypersurfaces $\{ b(x,\theta) = s \}$ where $b$ is a suitable defining function \cite{HomanZhou2016}.

Let $\mx$ be a manifold without boundary. We will study a transform that integrates a function $f \in \mathcal{E}'(\mx)$ over codimension $k$ submanifolds $G_{\theta,s}$ in $\mx$, given explicitly as 
 \[
 G_{\theta,s} = \{ x \in \mx \mid b(x,\theta) = s \},
 \]
where $b: \mx \times V' \to \mR^k$ is a smooth function, $\theta \in V'$ and $s \in V''$, where $V'$ and $V''$ are open subsets of $\mR^m$ and $\mR^k$, respectively. Here $m \geq 1$ and $1 \leq k \leq n-1$ where $n = \dim(\mx)$. We assume that we are considering integrals over $G_{\theta,s}$ with $(\theta,s)$ close to a fixed $(\theta_0,s_0)$, so we can work in local coordinates and take $\mG = V' \times V'' \subset \mR^{m+k}$. Moreover, if we let $U \subset \mx$ be a neighborhood of $G_{\theta_0,s_0}$ such that $G_{\theta,s} \subset U$ for $(\theta,s) \in \mG$, it will be enough to have everything defined for $x \in U$.
 
We assume that the defining function $b$ is such that the Jacobi matrix $b_x$ satisfies 
\begin{equation} \label{dxb_condition}
\text{$b_x(x,\theta)$ is surjective for $x \in U$ and $\theta \in V'$.}
\end{equation}
Then $Z = \{ (\theta,s,x) \mid b(x,\theta)=s \}$ will be a double fibration. If we fix orientation forms on $Z$ and $\mG$ (we can use the Lebesgue measure on $\mG \subset \mR^{m+k}$), we obtain orientation forms on each $G_{\theta,s}$ and consider the weighted double fibration transform 
\[
Rf(\theta,s) = \int_{G_{\theta,s}} \kappa(\theta,s,x) f(x) \,d\omega_{G_{\theta,s}}(x), \qquad (\theta,s) \in \mG,
\]
where $\kappa$ is smooth and nowhere vanishing.

The Bolker condition can be described as follows. The map $d\pi_L$ is injective if for any $x \in U$, $\theta \in V'$ and $\zeta \in \mR^{k} \setminus 0$, the linear map  
\begin{equation} \label{cdk_bolker1}
\text{$\left( b_x(x,\theta)^T, \ \  \p_{\theta} (b_x(x,\theta)^T \zeta) \right)$ is surjective.}
\end{equation}
The map $\pi_L$ is injective if for any $\theta \in V'$ and $\zeta \in \mR^{k} \setminus 0$, the map 
\begin{equation} \label{cdk_bolker2}
x \mapsto (b(x,\theta), b_{\theta}(x,\theta)^T \zeta) \text{ is injective on $U$.}
\end{equation}
We remark that \eqref{cdk_bolker1} can only be valid when $m \geq n-k$, which is precisely the condition that the inverse problem for $R$ is not formally underdetermined. We also remark that the weighted Euclidean Radon transform is obtained by taking $b: \mR^n \times \mR^n \to \mR, \ b(x,\theta) = x \cdot \theta$, which satisfies \eqref{dxb_condition}--\eqref{cdk_bolker2}.

Theorem \ref{global elliptic regularity intro} and microlocal unique continuation yield the following uniqueness theorem.

\begin{Theorem} \label{thm_b_k}
Let $X, b, \kappa$ and $\omega_{G_{\theta,s}}$ above be analytic. Suppose that $\Sigma$ is a hypersurface in $\mx$ such that $(x_0,\eta_0) \in N^* \Sigma \cap N^* G_{\theta_0,s_0}$ and that $f \in \mathcal{E}'(\mx)$ vanishes on one side of $\Sigma$ near $x_0$. Assume that \eqref{dxb_condition}--\eqref{cdk_bolker2} hold. If $Rf(\theta,s) = 0$ for $(\theta,s)$ near $(\theta_0,s_0)$, then $f = 0$ near $x_0$.
\end{Theorem}

Let us compare the above setup to that in \cite{HomanZhou2016}. In \cite{HomanZhou2016}, one considers a function $b: \mx \times (\mR^n \setminus 0) \to \mR$ such that $b(x,\theta)$ is positively homogeneous of degree $1$ in $\theta$, $b_x$ is never zero, and the matrix $(\p_{x_j \theta_k} b)_{j,k=1}^n$ is positive definite. One also assumes the following form of a global Bolker condition: for any fixed $x$ and $\theta$, 
\[
\text{the map $b_{\theta}(\,\cdot\,,\theta)$ is injective, and the map $b_x(x, \,\cdot\,)$ is surjective.}
\]
We see that our setup generalizes this in several ways:
\begin{itemize}
\item 
the submanifolds $G_z$ can have any codimension $k$ with $1 \leq k \leq n-1$;
\item 
one can have any dimension $m \geq n-k$ for space of parameters $\theta$;
\item 
the homogeneity of $b$ is not needed;
\item 
the positivity of  $(\p_{x_j \theta_k} b)_{j,k=1}^n$ is replaced by the weaker condition \eqref{cdk_bolker1}; and 
\item 
the injectivity of $x \mapsto b_{\theta}(x,\theta)$ is replaced by the weaker condition \eqref{cdk_bolker2}.
\end{itemize}

\subsection{Methods} \label{sec_methods}

In our proof of the main theorem, instead of going through the normal operator $R^* R$ as in the $C^{\infty}$ case \cite{Guillemin1985}, we will work directly with the FIO $R$. In fact the adjoint $R^*$ is never used in this article. This avoids the issue of having to consider compositions of analytic FIOs. Instead we will employ the FBI transform characterization of the analytic wave front set in \cite{sjostrand}, both in the setting of Gaussian wave packets and generalized FBI transforms. Our argument also avoids the use of special cutoffs that are typical in analytic microlocal analysis and were used e.g.\ in the related works \cites{StefanovUhlmann, Plamen_lightray}.

Here is a more detailed outline of the proof of Theorem \ref{global elliptic regularity intro}. For clarity we will write $(  \hat z,   \hat \zeta,   \hat x,  \hat \eta)$ for the point of interest below.

\begin{enumerate}
\item[1.]
We first localize matters near $\hat x$ and choose $\chi \in C^{\infty}_c(X)$ with $0 \leq \chi \leq 1$ and $\chi = 1$ near $\hat x$ (in fact $\chi$ could even be a hard cutoff function). Writing $f = \chi f + (1-\chi) f$, we have 
\[
Rf = R(\chi f) + R((1-\chi)f).
\]
The Schwartz kernel of an analytic double fibration transform satisfies $\WF_a(R)' = C$ where $C$ is as in Definition \ref{def_doublefibration_general}. Now $\WF_a(R((1-\chi)f)) \subset C(\WF_a((1-\chi)f))$ by properties of general linear operators \cite[Theorem 8.5.5]{Hormander}. Since $\pi_L^{-1}(\hat{z},\hat{\zeta}) = \{(\hat z,\hat \zeta,\hat x,\hat \eta) \}$ by the Bolker condition, we always have $(\hat z,\hat \zeta) \notin \WF_a(R((1-\chi)f))$. It is thus enough to show that $(\hat z,\hat \zeta) \notin \WF_a(R(\chi f))$ implies $(\hat x,\hat \eta) \notin \WF_a(f)$.
\item[2.]
After localizing near $\hat x$, we reduce the problem to a model FIO 
\begin{equation} \label{b_op_expression}
Tf(z) = \int_{\mR^n} \int_{\mR^{n''}} e^{i(\phi(z,x') - x'') \cdot \eta} a(z,x) f(x) \,d\eta \,dx
\end{equation}
where the double fibration $Z$ is locally given by $\{ x'' = \phi(z,x') \}$ for suitable coordinates $x = (x', x'') \in \mR^{n'} \times \mR^{n''}$, and $f$ is supported near $\hat x$. From now on we only require the local Bolker condition that $d\pi_L|_{(\hat z,\hat \zeta,\hat x,\hat \eta)}$ is injective, or equivalently that the matrix 
\[
(\phi_z(\hat z, \hat x')^T,\ \partial_{x'}( \phi_z(\hat z,  x')^T\hat \eta)\mid_{x'= \hat x'})\ \ \text{ is injective}.
\]
This condition implies that $\pi_L$ is injective in a neighborhood of $(\hat z,\hat \zeta,\hat x,\hat \eta)$ in $C$. This, together with structural properties of $C$, implies that microlocalization on the $(z,\zeta)$ side combined with localization on the $x$ side effectively implies microlocalization on the $(x,\eta)$ side, without the need to introduce additional cutoffs in the $\eta$ variable (see Lemma \ref{lem_bolker_open}).
\item[3.]
We need to prove that $(  \hat z,  \hat \zeta) \notin \WF_a(Tf) \implies (  \hat x,   \hat \eta) \notin \WF_a(f)$.
We apply the FBI transform characterization of the analytic wave front set \cite[Definition 6.1]{sjostrand}. If $L_{\lambda}$ denotes the FBI transform with Gaussian wave packets, we have 
\[
(  \hat z,  \hat \zeta) \notin \WF_a(Tf) \ \, \Longleftrightarrow \, \ L_{\lambda} Tf(z,\zeta) = O(e^{-c\lambda}) \text{ uniformly for $(z,\zeta)$ near $(\hat z, \hat \zeta)$}.
\]
We write $L_{\lambda} T$ in terms of its Schwartz kernel $K_{\lambda}(z,\zeta,x)$, which can be expressed as an oscillatory integral. Then we have 
\begin{equation} \label{klambda_exp}
\int K_{\lambda}(z,\zeta,x) f(x) \,dx = O(e^{-c\lambda}) \text{ for $(z,\zeta)$ near $(\hat z, \hat \zeta)$}.
\end{equation}
\item[4.]
Next we simplify the oscillatory integral expression $K_{\lambda}$ as 
\[
K_{\lambda}(z,\zeta,x) = c \lambda^{\frac{3N}{4}} \int_{\mathcal U} e^{i\lambda \Psi(\zeta',x; z,\zeta)} \tilde{a}(\zeta',x) \,d\zeta'
\]
for some phase function $\Psi$ and amplitude $\tilde{a}$. Here we use that $a(z,x)$ is independent of $\eta$, so we can integrate out the $\eta$ variable to obtain a delta function of $Z$. We show that the phase has a unique nondegenerate critical point, apply analytic stationary phase \cite[Theorem 2.8]{sjostrand} to the formula for $K_{\lambda}$, and substitute the result in \eqref{klambda_exp} to obtain 
\begin{equation} \label{klambda_exp_two}
\int e^{i\lambda \psi(x,z,\zeta)} \tilde{a}_1(x,z,\zeta;\lambda) f(x) \,dx = O(e^{-c\lambda}).
\end{equation}
for certain $\psi$ and $\tilde{a}_1$.
\item[5.]
The final step is to verify that the left hand side of \eqref{klambda_exp_two} is a generalized FBI transform applied to $f$. This involves proving that the phase $\psi$ satisfies the estimates in \cite[Definition 6.1]{sjostrand}, including quadratic growth for $\mathrm{Im}(\psi)$. Establishing these estimates is the most technical part of the argument. Once this has been done, applying \cite[Definition 6.1]{sjostrand} in \eqref{klambda_exp_two} shows that $(\hat{x}, \hat{\eta}) \notin \WF_a(f)$ as required.
\end{enumerate}

\begin{Remark}
From the FIO point of view, it would be natural to consider operators of the form \eqref{b_op_expression} where the analytic function $a(z,x)$ is replaced by a classical analytic symbol $a(z,x,\eta)$ (see e.g.\ \cite{Treves22}). Such operators can also be handled by the method above. However, the analytic stationary phase argument will include an additional integral over $\eta \in \mR^n$ which is not directly covered by \cite[Theorem 2.8]{sjostrand}. Since operators of the form \eqref{b_op_expression} are already sufficient for our applications to integral geometry transforms, we will only give a sketch of this extension in Remark \ref{general amplitude}.
\end{Remark}

We discuss some previous work related to recovering analytic singularities. The possibility that the $C^{\infty}$ argument in \cite{Guillemin1985} could be extended to the analytic case was raised in \cites{BomanQuinto1993, Quinto1993}. However, since a detailed proof of the required general FIO composition result in the analytic category was not available, results for particular transforms were proved by alternative means in \cites{StefanovUhlmann, FrigyikStefanovUhlmann, Krishnan2009, KrishnanStefanov, Plamen_lightray, HomanZhou2016} building on the FBI transform method in  \cite{KenigSjostrandUhlmann}. Results based on realizing $R^* R$ explicitly as an analytic pseudodifferential operator are given in \cites{BomanQuinto1987, StefanovUhlmann_generic}. On the FIO theory side, Theorem 4.5 of \cite{sjostrand} gives a construction of a microlocal parametrix for an elliptic analytic FIO associated to a canonical transformation. Recent work of \cite{RoubySjostrandNgoc} gives a composition calculus under a transversal intersection condition. Neither of these works can be directly applied in our setting. Other related references on analytic microlocal analysis are \cites{KashiwaraKawai, Treves2, BonthonneauJezequel}.

There have also been several studies related to $C^{\infty}$ singularities in cases where the Bolker condition is violated. Some of these are related to seismic imaging applications, where the Bolker condition corresponds to a travel time injectivity condition \cites{NolanSymes, KSV}. In these results the composition $R^* R$ may not be covered by the clean intersection calculus, but one typically interprets $R^* R$ as a generalized Fourier integral operator whose wave front relation can be used to describe artifacts appearing in the imaging process. See e.g.\ \cites{GreenleafUhlmann, GreenleafUhlmann1990, Nolan2000, FinchLanUhlmann, StefanovUhlmann_SAR, FGGN} and references therein for various results of this type.

We would also like to mention that there is a large literature on double fibration transforms in various symmetric and homogeneous settings \cites{GGG_book, Helgason2011} and in the complex setting in connection with Penrose transforms \cite{BastonEastwood}.

\subsection*{Organization}

This paper is organized as follows. Section \ref{sec_basic_doublefibration} collects basic properties of double fibrations needed later, and Section \ref{injectivity of piL} gives several equivalent characterizations of the Bolker condition. In Section \ref{sec_ray} we give an alternative approach to double fibration ray transforms based on vector fields on fiber bundles, and discuss geodesic and null bicharacteristic ray transforms in more detail. Section \ref{sec_wf} gives the proof of Theorem \ref{global elliptic regularity intro} based on analytic microlocal analysis. Finally, in Section \ref{sec_appl} we prove Theorems \ref{thm_local_uniqueness}--\ref{thm_b_k} as consequences of Theorem \ref{global elliptic regularity intro}.

\subsection*{Acknowledgements}

The authors would like to thank Michael Hitrik, Todd Quinto, Plamen Stefanov and Gunther Uhlmann for helpful discussions and references. M.S.\ would also like to thank the Isaac Newton Institute for support and hospitality during the programme Rich and nonlinear tomography (EPSRC grant EP/R014604/1) when part of this work was undertaken. M.M.\ is partially supported by the ANR grants CoSyDy (ANRCE40-0014) and
COSY (ANR-21-CE40-0002). M.S.\ is partly supported by the Academy of Finland (Centre of Excellence in Inverse Modelling and Imaging, grants 284715 and 353091) and by the European Research Council under Horizon 2020 (ERC CoG 770924). L.T.\ is partly supported by Australian Research Council DP220101808.

\section{Basic properties of double fibration transforms} \label{sec_basic_doublefibration}

We begin by collecting certain properties of double fibration transforms that will be used later. We will state these properties in the smooth category, but the same properties continue to hold if one replaces smooth by analytic throughout.

Let $\mG$ and $\mx$ be oriented smooth manifolds without boundary so that $\dim(\mG) = N$ and $\dim(\mx) = n$. Assume that $Z$ is a smooth oriented embedded submanifold of $\mG \times \mx$ and consider the projections $\pi_{\mG}: Z \to \mG$ and $\pi_{\mx}: Z \to \mx$, as illustrated in the diagram \eqref{df_diagrams}.

We will assume that 
\[
\pi_{\mG}: Z \to \mG\ \text{is a submersion}.
\]
This forces $\dim(Z) \geq N$. If $\dim(Z) = N$ (resp.\ $\dim(Z) = N+n$) then the sets $G_z = \pi_{\mx}(\pi_{\mG}^{-1}(z))$ will be $0$-dimensional (resp.\ $n$-dimensional) and one does not obtain a standard integral geometry transform. For these reasons we will assume in this section that 
\[
N+n > \dim(Z) > N.
\]
We will also write 
\begin{align*}
\dim(Z) &= N + n', \\
n &= n' + n'',
\end{align*}
where $1 \leq n', n'' \leq n-1$.

Since $\pi_{\mG}$ is a submersion, for any $z \in \mG$ the fiber $\pi_{\mG}^{-1}(z)$ is an embedded $n'$-dimensional submanifold of $Z$, and correspondingly the set 
\begin{eqnarray}
\label{def: Gz}
G_z := \pi_{\mx}(\pi_{\mG}^{-1}(z))
\end{eqnarray}
is an embedded $n'$-dimensional submanifold of $\mx$. We equip $G_z$ with an orientation form as follows (see \cite{Quinto_measures} for more on the choice of measures).

\begin{Lemma} \label{lemma_orientation}
Given orientation forms $\omega_Z$ on $Z$ and $\omega_{\mG}$ on $\mG$, there is a natural orientation form $\omega_{G_z}$ on each $G_z$ induced by the submersion $\pi_{\mG}$.
\end{Lemma}
\begin{proof}
The submersion $\pi_{\mG}: Z \to \mG$ induces on each fiber $\pi_{\mG}^{-1}(z)$ the orientation form 
\[
\omega_{\pi_{\mG}^{-1}(z)}(v_1,\ldots,v_{n'}) = \frac{\omega_Z(v_1, \ldots, v_{n'}, w_1, \ldots, w_{N})}{\omega_{\mG}(d\pi_{\mG}(w_1), \ldots, d\pi_{\mG}(w_{N}))}
\]
where $\{ v_1, \ldots, v_{n'} \}$ spans $T_{(z,x)} \pi_{\mG}^{-1}(z)$ and $\{ v_1, \ldots, v_{n'}, w_1, \ldots, w_N \}$ spans $T_{(z,x)}Z$. 
Here we used that $\pi_{\mathcal G}$ is a submersion and $d \pi_{\mG}(v_j) = 0$ to ensure that $\{d\pi_{\mG}(w_1), \ldots, d\pi_{\mG}(w_{N}) \}$ is a basis of $T_z\mathcal G$. Furthermore, a standard change of basis formula for orientation forms ensures that the above definition does not depend on the choice of $\{w_1,\dots, w_N\}$ as long as $\{ v_1, \ldots, v_{n'}, w_1, \ldots, w_N \}$ spans $T_{(z,x)}Z$. It is then enough to note that $\pi_{\mx}: \pi_{\mG}^{-1}(z) \to G_z$ is a diffeomorphism and to define $\omega_{G_z}(d\pi_{\mx}(v_1), \ldots, d\pi_{\mx}(v_{n'})) = \omega_{\pi_{\mG}^{-1}(z)}(v_1,\ldots,v_{n'})$.
\end{proof}

Define the weighted double fibration transform $R: C^{\infty}_c(\mx) \to C^{\infty}(\mG)$ by 
\begin{eqnarray}
\label{def: R}
R f(z) = \int_{G_z} \kappa(z,x)f(x) \,d\omega_{G_z}(x), \qquad z \in \mG,
\end{eqnarray}
for some smooth and nowhere vanishing weight $\kappa(z,x)$.
The following result, which implies that $R$ indeed maps into $C^{\infty}(\mG)$, characterizes $R$ as an FIO associated with the conormal bundle $N^*Z$, acting on functions with values in the half density bundle $\Omega^{1/2}$.

\begin{Theorem} \label{thm_r_fio}
Assume that $\pi_{\mG}: Z \to \mG$ is a submersion and $\dim(Z) = N + n'$. Then $R$ is a Fourier integral operator of order $\frac{n}{4} - \frac{n'}{2} - \frac{N}{4}$  whose canonical relation is 
\begin{equation*}
C = (N^* Z \setminus 0)' =  \{ (z,A(z,x)\eta,x,-\eta) \mid (z,x) \in Z, \ \eta \in N_x^* G_z, \ \eta \neq 0 \}
\end{equation*}
where $A(z,x)$ is is a linear map $N_x^* G_z \to T_z^* \mG$ depending smoothly on $(z,x) \in Z$. The operator $R$ has a nowhere vanishing homogeneous principal symbol 
\[a_Z\in S^{\frac{n-n'}{2}}(N^*Z, \Omega^{1/2}).\]
\end{Theorem}

This result follows from \cite[Sections VI.3 and VI.6]{GuilleminSternberg} or \cite{Guillemin1985}. We will also give a proof below since some of the computations will be used in the proof of Theorem \ref{global elliptic regularity intro}. As a consequence of Theorem \ref{thm_r_fio}, $R$ is well defined as a map $\mathcal{E}'(\mx) \to \mathcal{D}'(\mG)$. We will later work under the assumption that $N \geq n$ and the left and right projections ($\pi_L$ and $\pi_R$) acting on $C$ have full rank. In this case, the rank of $d\pi_L$ will be $n+N$ while the rank of $d\pi_R$ will be $2n$. Then by Theorem 4.3.2 of \cite{Hormanderacta}, $R: H^{s}_c(X) \to H^{s+n'/2}_{\mathrm{loc}}(\mathcal G)$.

For many of our arguments, it will be convenient to write $Z$ locally as the graph $\{ x'' = \phi(z,x') \}$ for some coordinates $x = (x',x'')$ in $\mx$ with $x' \in \mR^{n'}$ and $x'' \in \mR^{n''}$.

\begin{Lemma}
\label{local description of Z}
Let $Z\subset \mR^N\times \mR^{n}$, ${\rm dim}(Z)>N$, be an embedded smooth submanifold such that 
\[
\text{$\pi_N : Z\to \mR^N$ is a submersion.}
\]
If $(\hat z, \hat x) \in Z$, we can find open subsets $ U\subset \subset \mR^{n'+n''}$ containing $\hat x$ and $V \subset\subset \mR^N$ containing $\hat z$ with $ U = U' \times U''$ and $U'\subset\subset\mR^{n'}$, $ U'' \subset\subset \mR^{n''}$ open, and a smooth function $\phi : V \times U' \to U''$ such that
\[
Z\cap (V\times U )= \{ (z,x) \mid x'' = \phi(z, x'), z\in V, x'\in U'\}.
\]
Conversely, if $Z = \{ x'' = \phi(z,x') \}$ for some smooth $\phi$, then $\pi_N$ is a submersion.
\end{Lemma}
\begin{proof}
Let $j$ be the embedding $Z \to \mR^N \times \mR^n$. Then $dj(w) = (d\pi_N(w), d\pi_n(w))$ for $w \in TZ$, so we may write $dj$ in local coordinates as the $(N+n) \times \dim(Z)$ matrix 
\[
dj = \left( \begin{array}{c} d\pi_N \\ d\pi_n \end{array} \right).
\]
Now $dj$ has $\dim(Z)$ linearly independent columns since $j$ is an immersion. This means that $dj$ has $\dim(Z)$ linearly independent rows. Furthermore, $d\pi_N$ has $N$ linearly independent rows since $\pi_N$ is a submersion. It follows that we can choose $n' = \dim(Z)-N$ components $x'$ of $x$, with $x = (x', x'') \in \mR^{n'}\times \mR^{n''}$, such that the differential of the projection $\pi_{N + n'}: Z\ni (z,x) \mapsto (z, x') \in \mR^N \times \mR^{n'}$ is a bijection at $( \hat z,  \hat x)$. Let $\pi_{n''}: Z\to \mR^{n''}$ be the projection to the other components.

By the inverse function theorem there is a small neighborhood $V\times U'\subset \mR^N\times \mR^{n'}$  of $(\hat z, \hat x')$ such that $\pi_{N+n'}^{-1}:  V \times  U'\to Z$ is a smooth diffeomorphism from $V \times  U'$ to a neighborhood of $(\hat z, \hat x)$ in $Z$ (here we use that $Z$ is an embedded submanifold). So the map $\phi := \pi_{n''}\circ \pi_{N+n'}^{-1} :  V\times  U' \to \mR^{n''}$ is the desired map. The converse follows since $\pi_N(z,x',\phi(z,x')) = z$.
\end{proof}

We now compute $N^* Z$, considered as a subset of $T^* \mG \times T^* \mx$, and describe the linear map $A(z,x)$ appearing in Theorem \ref{thm_r_fio} more precisely.

\begin{Lemma} \label{lemma_nstarz}
Assume that $\pi_{\mG}: Z \to \mG$ is a submersion. One has 
\[
N^* Z = \{ (z, \zeta, x, \eta) \mid (z,x) \in Z, \ \eta \in N_x^* G_z, \ \zeta = A(z,x) \eta \}
\]
where $A(z,x)$ is a linear map $N_x^* G_z \to T_z^* \mG$ depending smoothly on $(z,x) \in Z$. If $Z$ is locally given by $\{ x'' = \phi(z,x') \}$, then in these coordinates  
\begin{align*}
T_x G_z &= \{ (\eta', \phi_{x'}(z,x') \eta') \mid \eta \in \mR^{n'} \}, \\
N_x^* G_z &= \{ (-\phi_{x'}(z,x')^T \eta'', \eta'') \mid \eta'' \in \mR^{n''} \}, \\
N^*_{(z,x)} Z &= \{ (z, -\phi_z(z,x')^T \eta'', (x', \phi(z,x')), (-\phi_{x'}(z,x')^T \eta'', \eta'')) \mid \eta'' \in \mR^{n''} \},
\end{align*}
and 
\[
A(z,x)(-\phi_{x'}^T \eta'', \eta'') = -\phi_z^T \eta''.
\]
\end{Lemma}
\begin{proof}
By Lemma \ref{local description of Z}, we can write $Z$ locally as the set $\{ x'' = \phi(z,x') \}$. Any curve $x_s$ in $G_z$ satisfies $x_s'' = \phi(z,x_s')$ and therefore 
\[
\dot{x}_0'' = \phi_{x'}(z,x_0') \dot{x}_0'.
\]
This gives the formula for $T_x G_z$. Then any vector of the form $(-\phi_{x'}(z,x')^T \eta'', \eta'')$ must be in $N_x^* G_z$, and a dimension count shows that all vectors in $N_x^* G_z$ are of this form.

For any curve $(z_s,x_s)$ in $Z$ we have 
\[
\dot{x}_0'' = \phi_z(z_0,x'_0) \dot{z}_0 + \phi_{x'}(z_0,x'_0) \dot{x}_0'.
\]
It follows that $T_{(z_0,x_0)} Z = \{ (v, u', \phi_z(z_0,x_0') v + \phi_{x'}(z_0,x_0') u') \mid v \in \mR^N, \ u' \in \mR^{n'} \}$. Then $(\zeta, \eta) \in N^*_{(z_0,x_0)} Z$ if and only if 
\[
\zeta(v) + \eta'(u') + \eta''(\phi_z(z_0,x_0') v + \phi_{x'}(z_0,x_0') u') = 0
\]
for all $v$ and $u'$. Choosing $v=0$ we have $\eta'(u') + \eta''(\phi_{x'}(z_0,x_0') u') = 0$ for all $u'$, which is equivalent to $\eta \in N_{x_0}^* G_{z_0}$. Then varying $v$ gives $\zeta = -\phi_z(z_0,x_0')^T \eta''$. This proves the formula for $N_{(z,x)}^* Z$.

Finally, the argument above shows that any $(z, \zeta, x, \eta) \in N^* Z$ must satisfy $x \in G_z$, $\eta \in N_x^* G_z$, and $\zeta = A(z,x) \eta$ where $A(z,x)$ is defined on $N_x^* G_z$ by the required formula $A(z,x)(-\phi_{x'}^T \eta'', \eta'') = -\phi_z^T \eta''$.  In order to state precisely the smooth dependence of $A$ on $(z,x)$, consider the smooth vector bundles $E$ and $F$ over $Z$ with fibers at $(z,x)$ given by $E_{(z,x)} = N_x^* G_z$ and $F_{(z,x)} = T_z^* \mG$. Note that $F$ is a pullback bundle, and $E$ is a smooth vector bundle with local trivializations given by $((z,x',\phi(z,x')), (-\phi_z^T \eta'', \eta'')) \mapsto ((z,x'), \eta'')$. Then $A$ is indeed a smooth section of the bundle $\mathrm{Hom}(E,F)$.
\end{proof}

We proceed to the proof of Theorem \ref{thm_r_fio}. It follows from \cite{Guillemin1985} that if $Z$ is a double fibration, then $R$ is an FIO whose Schwartz kernel is essentially the delta function $\kappa \delta_Z$. To state this properly we need to consider $R$ acting on half densities. If we fix nonvanishing half densities on $\mG$ and $\mx$, we obtain identifications between functions and half densities and we can identify $R$ with an operator, still denoted by $R$, that acts on half density valued functions.

\begin{proof}[Proof of Theorem \ref{thm_r_fio}]
To simplify notation we assume without loss of generality that $\kappa(z,x) =1$.
We define $\delta_Z$ precisely via the formula 
\[
\langle \delta_Z, \Psi  \rangle = \int_{\mG} \int_{G_z} \Psi(z, x) \,d\omega_{G_z}(x) \,d\omega_{\mG}(z), \qquad \Psi \in C^{\infty}_c(\mG \times \mx).
\]
Then $\delta_Z$ is a continuous linear functional on $C^{\infty}_c(\mG \times \mx)$, i.e.\ a distribution density on $\mG \times \mx$. It follows from the definitions that 
\[
\langle \delta_Z, \psi(z) f(x)  \rangle = \int_{\mG} (R f(z)) \psi(z) \,d\omega_{\mG}(z), \qquad f \in C^{\infty}_c(\mx), \ \psi \in C^{\infty}_c(\mG).
\]
This shows that $\delta_Z$ is the (density valued) Schwartz kernel of $R$.

To see that $\delta_Z$ behaves indeed like a delta function of $Z$, let $(z_0, x_0) \in Z$ and use Lemma \ref{local description of Z} to find smooth positively oriented local coordinates $(z,x',x'')$ in some neighborhood $V \times U$ of $(z_0,x_0)$ so that $V \subset \mG$ and $U \subset \mx$ are open, $Z \cap (V \times U) = \{ (z,x',x'') \in V \times U \mid x'' = \phi(z',x') \}$, and $G_z \cap U = \{ (x',x'') \in U \mid x'' = \phi(z,x') \}$ for $z \in V$. In these coordinates $d\omega_{G_z}(x) = a_0(z,x') \,dx'$ and $d\omega_{\mG}(z) = b(z) \,dz$ where $a_0$ and $b$ are smooth positive functions. Then if $\Psi$ has small support near $(z_0,x_0)$ we have 
\[
\langle \delta_Z, \Psi  \rangle = \int_{\mR^N} \int_{\mR^n} \delta_{\{x''=\phi(z,x')\}}(z,x',x'') \Psi(z,x',x'') a_0(z,x') b(z) \,dx' \,dz.
\]
Thus in these local coordinates $\delta_Z$ becomes the conormal distribution $a_0(z,x') b(z) \delta_{\{x''=\phi(z,x')\}}$ in $\mR^{N+n}$ whose symbol is positive, smooth and homogeneous of degree zero. We may also write $\delta_Z$, or equivalently the Schwartz kernel $R(z,x)$ of $R$, as the oscillatory integral 
\begin{equation} \label{rzx_formula}
R(z,x) = \int_{\mR^{n''}} e^{i(\phi(z,x')-x'') \cdot \eta} a(z,x) \,d\eta
\end{equation}
for some positive nonvanishing smooth function $a$.

After multiplying by nonvanishing smooth half densities we can interpret $\delta_Z$ as a half density valued distribution conormal to $Z$. Since $N^* Z$ has fiber dimension $n - n'$, we see that $\delta_Z$ has principal symbol in $S^{\frac{n-n'}{2}}(N^* Z, \Omega^{1/2})$ according to the convention in \cite[Definition 18.2.10]{Hormander}. Then by \cite[Theorem 18.2.11]{Hormander} $\delta_Z$ is a conormal distribution of order $m$ in $\mG \times \mx$, where 
\[
m + \frac{N+n}{4} = \frac{n-n'}{2}.
\]
It follows that $R$ is a Fourier integral operator whose Schwartz kernel is a conormal to the Lagrangian manifold $N^* Z \setminus 0$, and thus $R$ has canonical relation $C = (N^* Z \setminus 0)'$ (see \cite[Section 25.2]{Hormander}). The formula for $C$ follows from Lemma \ref{lemma_nstarz}.
\end{proof}

Up to now we have considered submanifolds $Z \subset \mG \times \mx$ such that $\pi_{\mG}$ is a submersion.  From now on we will also assume that $\pi_{\mx}$ is a submersion, i.e.\ that $Z$ is a double fibration. This is locally characterized as follows.

\begin{Lemma} \label{lemma_pim_submersion}
Let $Z \subset \mG \times \mx$ be locally given by $\{ x'' = \phi(z,x') \}$. Then $\pi_{\mx}: Z \to \mi$ is a submersion iff the matrix $\phi_z(z,x')$ is surjective iff the map $A(z,x)$ in Lemma \ref{lemma_nstarz} is injective.
\end{Lemma}
\begin{proof}
We identify $(z,x',\phi(z,x')) \in Z$ with $(z,x')$. Since $\pi_{\mx}(z,x') = (x', \phi(z,x'))$, we have $d\pi_{\mx}(\dot{z}, \dot{x}') = (\dot{x}', \phi_z \dot{z} + \phi_{x'} \dot{x}')$. Thus $d\pi_{\mx}$ is surjective iff $\phi_z$ is surjective. This is equivalent with $A(z,x)$ being injective by Lemma \ref{lemma_nstarz}.
\end{proof}

Since $\pi_{\mx}$ is a submersion, the fibers $\pi_{\mx}^{-1}(x)$ are smooth manifolds in $Z$. Correspondingly the sets 
\[
H_x := \pi_{\mG}(\pi_{\mx}^{-1}(x))
\]
are embedded $(N-n'')$-dimensional submanifolds in $\mG$.

For a double fibration, both projections $\pi_{\mG}$ and $\pi_{\mx}$ are submersions and thus the roles of the $z$ and $x$ variables are somewhat symmetric. This leads to the following analogue of the results given above.

\begin{Lemma} \label{lem_df_b}
Let $Z$ be a double fibration with $N+n > \dim(Z) > n$. Then 
\[
N^* Z = \{ (z,\zeta,x,\eta) \mid (z,x) \in Z, \ \zeta \in N_z^* H_x, \ \eta = B(z,x)\zeta \}
\]
where $B(z,x)$ is a linear map $N_z^* H_x \to T_x^* X$ depending smoothly on $(z,x) \in Z$. If $(z_0, x_0) \in Z$, there are local coordinates $z = (z',z'')$ near $z_0$ and a smooth $\mR^{n''}$-valued function $b(x,z')$ near $(x_0, z_0')$ such that near $(z_0,x_0)$, $b_x$ is surjective and one has 
\begin{align*}
Z &= \{ (z,x) \mid z'' = b(x,z') \}, \\
N_{(z,x)}^* Z &= \{ ((z', b(x,z')), (-b_{z'}(x,z')^T \zeta'', \zeta''), x, -b_x(x,z')^T \zeta'') \mid \zeta'' \in \mR^{n''} \}.
\end{align*}
In these coordinates one has $B(z,x) \zeta = -b_x(x,z')^T \zeta''$. Conversely, if $b(x,z')$ is a smooth function such that $b_x$ is surjective, then $Z =  \{ z'' = b(x,z') \}$ is a double fibration.
\end{Lemma}
\begin{proof}
This follows by interchanging the roles of $z$ and $x$ above. Since $\pi_{\mx}$ is a submersion, we can use Lemma \ref{local description of Z} to find coordinates $z = (z',z'')$ with $z'' \in \mR^{n''}$ such that locally $Z = \{ z'' = b(x,z') \}$ for some smooth function $b$. Since $Z$ is a double fibration, $b_x$ is surjective by Lemma \ref{lemma_pim_submersion}. Conversely, if $Z = \{ z'' = b(x,z') \}$ where $b_x$ is surjective, then $Z$ is a double fibration by Lemmas \ref{local description of Z} and \ref{lemma_pim_submersion}.
\end{proof}

If $Z$ is a double fibration, it follows from the above results that for any $(z,x) \in Z$ the map $A(z,x)$ is a bijection $N_x^* G_z \to N_z^* H_x$ with inverse map given by $B(z,x)$. We can thus write $N^*Z$ in a more symmetric form as 
\[
N^* Z = \{ (z,\zeta,x,\eta) \mid (z,x) \in Z, \ \zeta = A(z,x) \eta, \ \eta = B(z,x)\zeta \},
\]
and we can parametrize points in $N^*Z$ over $(z,x)$ either by $\eta \in N_x^* G_z$ or by $\zeta \in N_z^* H_x$.

\section{The Bolker condition}
\label{injectivity of piL}

In this section we give several equivalent characterizations for the Bolker condition \cite{Guillemin1985} in the setting of double fibration transforms. This condition states that the left projection $\pi_L: C \to T^* \mG \setminus 0$ in the microlocal diagram \eqref{df_diagrams} is an injective immersion. From this point on we will also assume that $N \geq n$. One reason is that the Bolker condition can only hold when $N \geq n$ (see Lemma \ref{lemma_bolker_local_char}). Another reason is that the double fibration transform maps function in $\mx$ to functions in $\mG$, and the problem of inverting this transform will be formally underdetermined unless $N \geq n$.

\subsection{Injectivity of $\pi_L$}

As discussed in Section \ref{subsec_ray}, the injectivity of $\pi_L$ is in some way related to variations and conjugate points. We begin by discussing these notions in the double fibration setting. However, we will use the terminology \emph{$Z$-conjugate points} instead of conjugate points since the two notions are not always equivalent. Recall the linear map $A(z,x):  N_x^* G_z \to T_z^* \mG$ from Lemma \ref{lemma_nstarz}, and denote the adjoint map by $A(z,x)^*: T_z \mG \to (N_x^* G_z)^*$. By Lemma \ref{lemma_pim_submersion}, $A(z,x)$ is injective and $A(z,x)^*$ is surjective. The following notions generalize those in Definition \ref{def_hx_vts_ray}.

\begin{Definition} \label{def_variation_general}
Let $z_s$ be a curve through $z_0$ in $\mG$ with $\dot{z}_0:=\partial_s z_s|_{s=0} = w$. The family $(G_{z_s})$ is called a \emph{variation} of $G_{z_0}$, and we define the associated \emph{variation field}  
\[
J_w: G_{z_0} \to (N^* G_{z_0})^*, \ \ J_w(x) = A(z,x)^* w.
\]
Given $z \in \mG$ and $x, y \in G_z$ let $V_{z}(x, y)$ be the space of all possible directions of variation at $x$ among variations of $G_{z}$ that keep $y$ fixed, i.e.\ 
\[
V_{z}(x,y) = \{ J_w(x) \mid w \in T_{z} \mG, \ J_w(y) = 0 \}.
\]
We say that $x, y \in G_z$ are \emph{$Z$-conjugate} along $G_z$ if $V_{z}(x,y)$ is a strict subspace of $(N_x^* G_{z})^*$.
\end{Definition}

To justify the definition of $J_w(x)$, let $(G_{z_s})$ be a variation of $G_{z_0}$ with $\partial_s z_s\mid_{s= 0} = w$, and fix some $x_0 \in G_{z_0}$. If $x_s$ is any smooth curve through $x_0$ with $x_s \in G_{z_s}$, then $(z_s, x_s) \in Z$ and therefore any element in $N_{(z_0,x_0)}^*Z$ annihilates $(\partial_s z_s, \partial_s x_s)\mid_{s=0}$. Since covectors in $N_{(z_0,x_0)}^* Z$ have the form $(A(z_0,x_0)\eta, \eta)$ with $\eta \in N_{x_0}^* G_{z_0}$, we have 
\[
\eta(\dot{x}_0) = -A(z_0,x_0)\eta(\dot{z}_0) = -J_{w}(x_0)(\eta), \qquad \eta \in N_{x_0}^* G_{z_0}.
\]
Thus $-J_w(x_0)$ describes the normal component of the variation of $x_0$ within the manifolds $G_{z_s}$. We could fix a Riemannian metric on $\mx$ and identify $(N_x^* G_z)^*$ with the orthocomplement of $T_x G_z$ in this metric, and then $J_w$ would become an actual orthogonal vector field on $G_{z_0}$.

The following lemma gives several characterizations of $Z$-conjugate points. By (c) the notion of $Z$-conjugate points only makes sense when $N \geq 2n''$ (that is, if $N < 2n''$ then all pairs of points on any $G_z$ are $Z$-conjugate).

\begin{Lemma} \label{lemma_ncp_general}
Let $x, y \in G_z$ with $x \neq y$. The following conditions are equivalent.
\begin{enumerate}
\item[(a)]
$x$ and $y$ are not $Z$-conjugate along $G_z$.
\item[(b)] 
The space $V_{z}(x,y)$ has dimension $n''$.
\item[(c)] 
The space $\{ w \in T_z \mG \mid J_w(x) = J_w(y) = 0 \}$ has dimension $N-2n''$.
\item[(d)] 
$\ker(A(z,x)^*) + \ker(A(z,y)^*) = T_z \mG$.
\end{enumerate}
\end{Lemma}
\begin{proof}
The equivalence of (a) and (b) is clear since $\dim((N_x^* G_z)^*) = n''$. For the equivalence of (b) and (c), note that $\ker(A(z,y)^*)$ has dimension $N-n''$ since $A(z,y)^*$ is surjective, and apply the rank-nullity theorem to $A(z,x)^*: \ker(A(z,y)^*) \to (N_x^* G_z)^*$ whose range is $V_z(x,y)$ to get 
\[
\dim(\{ w \in T_z \mG \mid J_w(x) = J_w(y) = 0 \}) + \dim(V_z(x,y)) = N-n''.
\]
Finally, (c) means $\dim(\ker(A(z,x)^*) \cap \ker(A(z,y)^*)) = N-2n''$, which is equivalent to (d) by the general formula $\dim(E+F) + \dim(E \cap F) = \dim(E) + \dim(F)$.
\end{proof}

The following result gives a characterization for the global part of the Bolker condition.

\begin{Lemma}\label{l:no_conjugate_points}
Let $(z,\zeta,x,\eta) \in C$. Then $\pi_L^{-1}(z,\zeta) = \{(z,\zeta,x,\eta)\}$ iff 
\[
\text{for any $y \in G_z \setminus \{x\}$, the space $V_z(x,y)$ is not annihilated by $\eta$.}
\]
Moreover, $\pi_L: C \to T^* \mG \setminus 0$ is injective iff there are no pairs of $Z$-conjugate points on any $G_z$.
\end{Lemma}

\begin{proof}
By Lemma \ref{lemma_pim_submersion}, the map $A(z,x)$ is injective. Now if $\pi_L(z,\zeta,x,\eta)=\pi_L(z,\zeta,x,\eta_1)$ we have $\eta=\eta_1$, since $\zeta=A(z,x)\eta=A(z,x)\eta_1$. Therefore, $\pi_L$ is injective if and only if, for each pair of distinct points $(z,x),(z,y)\in Z$, we have a trivial intersection $\mathrm{im}(A(z,x))\cap\mathrm{im}(A(z,y))=\{0\}$. This latter condition is equivalent to Lemma \ref{lemma_ncp_general} part (d).

More specifically, $\pi_L^{-1}(z,\zeta) = \{(z,\zeta,x,\eta)\}$ if and only if \[A(z,x)\eta\not\in\mathrm{im}(A(z,x_1)),
\qquad \forall x_1 \in G_z\setminus\{x\}.\]
This latter condition is equivalent to the fact that there exists $w\in T_z\mG=(T_z^*\mG)^*$ such that $w(A(z,x)\eta)\neq0$ and $w(A(z,x_1)\eta_1)=0$ for all $\eta_1\in N^*_{x_1}G_z$, namely $A(z,x_1)^*w=0$.
\end{proof}

\subsection{Injectivity of $d\pi_L$}

The following result gives equivalent conditions for $d\pi_L$ to be injective. The equivalence of (1) and (2) is a general property of Lagrangian manifolds (see Lemma 4.3 of \cite{deHoopStolk}). Condition (3) gives a geometric interpretation: $d\pi_L|_{(z_0, \zeta_0, x_0, \eta_0)}$ is injective if and only if there are so many manifolds $G_z$ (with $z$ near $z_0$) through $x_0$ such that one can obtain any direction in $T_{x_0}^* \mx$ by varying $\eta_0$ within vectors conormal to the manifolds $G_z$. Condition (4) gives another geometric interpretation in terms of tangents of $G_z$. Finally, conditions (5) and (6) describe injectivity of $d\pi_L$ in terms of the local representations $Z = \{ x'' = \phi(z,x') \} = \{ z'' = b(x,z') \}$. Condition (5) will be used frequently in the computations in Section \ref{sec_wf}.

\begin{Lemma} \label{lemma_bolker_local_char}
Let $Z \subset \mG \times \mx$ be a double fibration. Given any $(z_0, \zeta_0, x_0, \eta_0) \in C$, the following conditions are equivalent.
\begin{enumerate}
\item[(1)] 
$d\pi_L|_{(z_0, \zeta_0, x_0, \eta_0)}$ is injective.
\item[(2)] 
$d\pi_R|_{(z_0, \zeta_0, x_0, \eta_0)}$ is surjective.
\item[(3)] 
For any $\xi_0 \in T_{x_0}^* \mx$ there are smooth curves $z_s$ through $z_0$ in $H_{x_0}$ and $\eta_s$ through $\eta_0$ in $T^*_{x_0} \mx$ such that $\eta_s \in N_{x_0}^* G_{z_s}$ and $\dot{\eta}_0 = \xi_0$.
\item[(4)] 
If $F: H_{x_0} \to (T_{x_0} \mx)^{n'}$ is any smooth map such that $\{ F_1(z), \ldots, F_{n'}(z) \}$ spans $T_{x_0} G_{z}$ for $z$ near $z_0$, then the map $(\eta_0(dF_1(\,\cdot\,)), \ldots, \eta_0(dF_{n'}(\,\cdot\,)))$ takes $T_{z_0} H_{x_0}$ onto $\mR^{n'}$.
\item[(5)] 
If $Z$ is given by $\{ x'' = \phi(z,x') \}$ near $(z_0,x_0)$, then 
\[
\text{the $N \times n$ matrix $\left( \phi_z(z_0,x'_0)^T, \ \  \p_{x'} (\phi_z(z_0,x')^T \eta_0'')|_{x'=x_0'} \right)$ is injective.}
\]
\item[(6)] 
If $Z$ is given by $\{ z'' = b(x,z') \}$ near $(z_0,x_0)$, then 
\[
\text{the $n \times N$ matrix $\left( b_x(x_0,z'_0)^T, \ \  \p_{z'} (b_x(x_0,z')^T \zeta_0'')|_{z'=z_0'} \right)$ is surjective.}
\]

\end{enumerate}
\end{Lemma}
\begin{proof}
By Lemma \ref{local description of Z}, we can express $Z$ as the set $\{ x'' = \phi(z,x') \}$ near $(z_0,x_0)$. By the description of $C$ in Lemma \ref{lemma_nstarz}, any curve $\Gamma_s$ in $C$ with $\Gamma_0=(z_0, \zeta_0, x_0, \eta_0)$ can be written as 
\[
\Gamma_s = (z_s, \phi_z(z_s, x_s')^T \eta_s'', (x_s', \phi(z_s,x_s')), (-\phi_{x'}(z_s,x_s')^T \eta_s'', \eta_s''))
\]
where $z_s$, $x_s'$ and $\eta_s''$ are curves in $\mR^N$, $\mR^{n'}$ and $\mR^{n''}$, respectively. We have chosen signs so that $\eta_0 = (-\phi_{x'}(z_0,x_0')^T \eta_0'', \eta_0'')$.

We first show (1) $\Longleftrightarrow$ (5). Now 
\[
d\pi_L(\dot{\Gamma}_0) = (\dot{z}_0, \p_s(\phi_z(z_s, x_s')^T \eta_0'')|_{s=0} + \phi_z(z_0,x_0')^T \dot{\eta}_0'').
\]
It follows that $d\pi_L(\dot{\Gamma}_0) = 0$ if and only if 
\[
\dot{z}_0 = 0, \qquad \phi_z(z_0,x_0')^T \dot{\eta}_0'' + \p_{x'}(\phi_z(z_0,x')^T \eta_0'')|_{x'=x_0'} \dot{x}_0' = 0.
\]
Thus $d\pi_L|_{(z_0, \zeta_0, x_0, \eta_0)}$ is injective if and only if (5) holds.

Next we show (2) $\Longleftrightarrow$ (5). We have 
\[
d\pi_R(\dot{\Gamma}_0) = ((\dot{x}_0', \p_s(\phi(z_s,x_s'))|_{s=0}), (-\p_s(\phi_{x'}(z_s,x_s')^T \eta_s'')|_{s=0}, \dot{\eta}_0'')).
\]
Thus in matrix form 
\[
d\pi_R(\dot{\Gamma}_0) = \left( \begin{array}{ccc} I & 0 & 0 \\ * & \phi_z(z_0,x_0') & 0 \\ * & -\p_z(\phi_{x'}(z,x_0')^T \eta_0'')|_{z=z_0} & * \\ 0 & 0 & I \end{array} \right) \left( \begin{array}{ccc} \dot{x}_0' \\ \dot{z}_0 \\ \dot{\eta}_0'' \end{array} \right).
\]
Then $d\pi_R|_{(z_0, \zeta_0, x_0, \eta_0)}$ is surjective if and only if  
\begin{equation} \label{transpose_fullrank}
\left( \begin{array}{c} \phi_z(z_0,x_0') \\ -\p_z(\phi_{x'}(z,x_0')^T \eta_0'')|_{z=z_0} \end{array} \right) \text{ has full rank}.
\end{equation}
After taking the transpose and noting that $(\p_z(\phi_{x'}^T \eta''))^T = \p_{x'}(\phi_z^T \eta'')$, the last condition is equivalent with the matrix $(\phi_z^T, \,-\p_{x'}(\phi_z^T \eta_0''))$ being injective, which is equivalent with (5).

The next step is to show (3) $\Longleftrightarrow$ (5). Assume that $z_s$ is a curve through $z_0$ in $H_{x_0}$ and $\eta_s$ is a curve through $\eta_0$ in $T^*_{x_0} \mx$ such that $\eta_s \in N^*_{x_0} G_{z_s}$. Since $z_s \in H_{x_0}$ we have $\phi(z_s,x_0') = x_0''$, which implies $\phi_z(z_0,x_0') \dot{z}_0 = 0$. We also have 
\[
\eta_s = (-\phi_{x'}(z_s,x_0')^T \eta_s'', \eta_s'')
\]
where $\eta_s''$ is a curve in $\mR^{n''}$. We have 
\[
\dot{\eta}_0 = (-\p_z(\phi_{x'}(z,x_0')^T \eta_0'')|_{z=z_0} \dot{z}_0 - \phi_{x'}(z_0,x_0')^T \dot{\eta}_0'', \dot{\eta}_0'').
\]
It follows that the condition in (3) holds if and only if 
\[
\p_z(\phi_{x'}(z,x_0')^T \eta_0'')|_{z=z_0}: \mathrm{ker}(\phi_z(z_0,x_0')) \to \mR^{n'} \text{ is surjective}.
\]
Since $Z$ is a double fibration, $\dim(\mathrm{ker}(\phi_z)) = N-n'' \geq n'$ using Lemma \ref{lemma_pim_submersion}. Thus (3) holds if and only if $\p_z(\phi_{x'}^T \eta_0'')|_{\mathrm{ker}(\phi_z)}$ has rank $n'$, which is equivalent with its kernel $\mathrm{ker}(\phi_z) \cap \mathrm{ker}(\p_z(\phi_{x'}^T \eta_0''))$ having dimension $N-n$. This is again equivalent with \eqref{transpose_fullrank}, which was seen above to be equivalent with (5).

Next we show (3) $\Longleftrightarrow$ (4). Assume that (3) holds, fix $\xi_0 \in T_{x_0}^* \mx$ and let $z_s$ and $\eta_s$ be as in (3). If $F$ is as in (4), then $\eta_s(F_j(z_s)) = 0$ for $1 \leq j \leq n'$. Taking derivatives at $s=0$, we get 
\begin{equation} \label{eq_xizero_dfj}
\xi_0(F_j(z_0)) + \eta_0(dF_j(\dot{z}_0)) = 0, \qquad 1 \leq j \leq n'.
\end{equation}
Since $\xi_0$ was arbitrary, this implies the conclusion in (4). Conversely, assume (4), let $\{ e_1, \dots, e_n \}$ be a basis of $T_{x_0}^* \mx$ with $e_j = F_j(z_0)$ for $1 \leq j \leq n'$, and let $\xi_0 \in T_{x_0}^* \mx$. We use (4) to find a curve $z_s$ through $z_0$ in $H_{x_0}$ such that \eqref{eq_xizero_dfj} holds. By Lemma \ref{lemma_nstarz}, after writing $Z$ locally as $\{ x'' = \phi(z,x') \}$, any curve $\eta_s$ through $\eta_0$ in $T_{x_0}^* \mx$ with $\eta_s \in N_{x_0}^* G_{z_s}$ has the form 
\[
\eta_s = (-\phi_{x'}(z_s,x_0') \eta_s'', \eta_s'')
\]
where $\eta_s''$ is a curve through $\eta_0''$ in $\mR^{n''}$. By choosing $\dot{\eta}_0''$ in a suitable way, we can ensure that $\eta_s$ satisfies 
\[
\dot{\eta}_0(e_j) = \xi_0(e_j), \qquad n'+1 \leq j \leq n.
\]
Since $\eta_s(F_j(z_s)) = 0$, taking the derivative at $s=0$ and using \eqref{eq_xizero_dfj} gives 
\[
\dot{\eta}_0(e_j) = \xi_0(e_j), \qquad 1 \leq j \leq n'.
\]
Thus $\dot{\eta}_0 = \xi_0$, which shows (3).

Finally, we show the equivalence (6) $\Longleftrightarrow$ (2). When $Z$ is given by $\{ z'' = b(x,z') \}$ near $(z_0,x_0)$, the projection $d\pi_R$ is given in matrix form by
\[
d\pi_R = \left( \begin{array}{ccc}  \p_{z'} (b_x(x_0,z')^T \zeta_0'')|_{z'=z_0'} & b_x(x_0,z'_0)^T & *\\ 0 & 0 & I\end{array} \right).
\]
Therefore, $d\pi_R$ is surjective if and only if the matrix $\big( b_x(x_0,z'_0)^T, \ \p_{z'} (b_x(x_0,z')^T \zeta_0'')|_{z'=z_0'} \big)$ has full rank.\end{proof}

If $d\pi_L|_{(z,\zeta,x,\eta)}$ is injective, then $\pi_L$ is injective in some neighborhood of $(z,\zeta,x,\eta)$ in $C$. We conclude this section with a lemma stating that $\pi_L$ is even injective if one restricts to a sufficiently small neighborhood of $(z,\zeta)$ and a neighborhood of $x$. This lemma will not be used explicitly later, but it justifies the claim made in Step 2 of Section \ref{sec_methods}. 

\begin{Lemma} \label{lem_bolker_open}
Suppose that $d\pi_L|_{(z,\zeta,x,\eta)}$ is injective. There are neighborhoods $\mathcal{V}$ of $(z,\zeta)$ in $T^* \mG$ and $U$ of $x$ in $\mx$ such that $\pi_L$ is injective in $\tilde{C}$, where 
\[
\tilde{C} = \{ (\tilde z,\tilde \zeta,\tilde x,\tilde \eta) \in C \mid (\tilde z, \tilde \zeta) \in \mathcal{V}, \ \tilde x \in U \}.
\]
\end{Lemma}
\begin{proof}
We argue by contradiction and assume that there are sequences $(z_j, \zeta_j) \to (z,\zeta)$, $x_{j,k} \to x$ as $j \to \infty$, and $\eta_{j,k} \in T_{x_{j,k}}^{*} \mx$ for $k=1,2$ such that $(z_j,\zeta_j, x_{j,k}, \eta_{j,k}) \in C$ but $(x_{j,1}, \eta_{j,1}) \neq (x_{j,2}, \eta_{j,2})$. By Lemma \ref{lem_df_b} we have 
\[
\eta_{j,k} = -B(z_j, x_{j,k}) \zeta_j.
\]
Since $B$ is smooth and $\eta = -B(z,x)\zeta$, we have $\eta_{j,1} \to \eta$ and $\eta_{j,2} \to \eta$. Then $(x_{j,1}, \eta_{j,1}) \neq (x_{j,2}, \eta_{j,2})$ are in a very small neighborhood of $(x,\eta)$ for $j$ large. This contradicts the fact that $\pi_L$ is injective near $(z,\zeta,x,\eta)$ in $C$.
\end{proof}

\section{The case of ray transforms} \label{sec_ray}

In this section we consider double fibration transforms as in Section \ref{sec_basic_doublefibration} in the special case where the manifolds $G_z$ are one-dimensional. We will show that each $G_z$ is an integral curve of some vector field $Y$ on $Z$. This provides an alternative point of view to double fibration ray transforms. We also discuss the Bolker condition for ray transforms and its relation to conjugate points..

\subsection{Double fibration vector fields}
\label{sec_vect_field}

We show that in a double fibration where the fibers $G_z$ are one-dimensional, the fibers are related to integral curves of a vector field:

\begin{Lemma} \label{lemma_dfb_vf}
Let $M$ be a manifold with smooth boundary, and let $\mG$ be an orientable manifold without boundary. Suppose that $Z$ is an embedded orientable submanifold of $\mG \times \mi$ with $\dim(Z) = \dim(\mG) + 1$, let $\pi_{\mG}: Z \to \mG$ be a surjective submersion with connected fibers, and let $\pi_M: Z \to \mi$ be a proper surjective submersion.

Then $Z$ is a smooth fiber bundle over $\mi$, and there is a vector field $Y$ on $Z$ with $d\pi_M(Y)$ nonvanishing such that $\pi_{\mG}^{-1}(z)$ is an integral curve of $Y$ for any $z \in \mG$. The vector field $Y$ is unique up to multiplication by a nonvanishing function on $Z$. Each integral curve is either periodic, or its projection on $M$ is injective. If we fix orientation forms on $Z$ and $\mG$ and let $\omega_{G_z}$ be the induced orientation form on $G_z$, then $Y$ can be chosen so that 
\begin{equation} \label{gzf_ray}
\int_{G_z} f \,\omega_{G_z} = \int_{-\tau_-(z)}^{\tau_+(z)} f(\pi_M(\gamma_z(t))) \,dt
\end{equation}
for any $f$ such that this is well defined. Here $\gamma_z: (-\tau_-(z), \tau_+(z)) \to Z$ is the maximally extended integral curve of $Y$ through $z$ (if $\gamma_z$ is periodic then its domain is chosen to be one of its minimal periods). The parametrization of curves is nonredundant in the following sense: given any $(z_0,x_0) \in Z$ there is a neighborhood $V$ of $z_0$ in $\mG$ that can be identified with a submanifold $\tilde{V}$ through $(z_0,x_0)$ in $Z$ such that the integral curves of $Y$ are transverse to $\tilde{V}$. 
\end{Lemma}
\begin{proof}
Since $\pi_M: Z \to \mi$ is a proper surjective submersion, Ehresmann's lemma ensures that $\pi_M$ is a locally trivial fibration and hence $Z$ is a fiber bundle over $\mi$. Let $\omega_{\pi_{\mG}^{-1}(z)}$ be the orientation form on $\pi_{\mG}^{-1}(z)$ as in Lemma \ref{lemma_orientation}.

We can now define $Y(z,x)$ to be the unique positively oriented vector in the one-dimensional space $T_{(z,x)} \pi_{\mG}^{-1}(z)$ that satisfies $\omega_{\pi_{\mG}^{-1}(z)}(Y) = 1$. Then $Y$ is a smooth vector field on $Z$ with $d\pi_{\mG}(Y) = 0$, and for any $(z,x_0) \in Z$ the maximally extended integral curve of $Y$ through $(z,x_0)$ is a component of $\pi_{\mG}^{-1}(z)$, hence equal to $\pi_{\mG}^{-1}(z)$ since the fibers are connected. The vector field $Y$ is unique up to multiplication by a nonvanishing function since any such $Y$ must lie in the one-dimensional space $T_{(z,x)} \pi_{\mG}^{-1}(z)$. We have $d\pi_M(Y) \neq 0$ since any $v \in TZ$ with $d\pi_{\mG}(v) = d\pi_M(v) = 0$ must satisfy $v=0$. If $\gamma_z(t)$ is an integral curve of $Y$ and if $\pi_M(\gamma_z(t_1)) = \pi_M(\gamma_z(t_2))$, then from $\pi_{\mG}(\gamma_z(t_1)) = \pi_{\mG}(\gamma_z(t_2))$ we obtain $\gamma_z(t_1) = \gamma_z(t_2)$. Thus each integral curve is either periodic or its projection to $M$ is injective. The formula \eqref{gzf_ray} follows since $\omega_{\pi_{\mG}^{-1}(z)}(Y) = 1$ and 
\[
\omega_{G_z}(d\pi_M(v)) = \omega_{\pi_{\mG}^{-1}(z)}(v), \qquad v \in T_{(z,x)} \pi_{\mG}^{-1}(z).
\]

Finally, given $(z_0,x_0) \in Z$ we use Lemma \ref{local description of Z} to write $Z$ locally in a neighborhood $V \times U'$ of $(z_0,x_0')$ as $\{ x'' = \phi(z,x') \}$. Then $\tilde{V} = \{ (z,x_0',\phi(z,x_0')) \mid z \in V \}$ is a submanifold of $Z$ that can be identified with $V$. The tangents of $\tilde{V}$ are of the form $(w,0',\phi_z(z,x_0') w)$, whereas $Y$ takes the form $(0,\lambda, \phi_{x'}(z,x_0') \lambda)$.
\end{proof}

\begin{Remark}
See \cite{Helgason2011} for classical examples of double fibration ray transforms integrating over periodic curves. 
\end{Remark}

We will next prove a converse to Lemma \ref{lemma_dfb_vf}, in the setting already discussed in Section \ref{subsec_ray}. Assume that $(M,g)$ is an oriented manifold with smooth boundary, and $\pi: \Xi \to M$ is a smooth fiber bundle over $M$ whose fibers $\Xi_x$ are manifolds without boundary (then $\p \Xi = \pi^{-1}(\p M)$). Let $Y$ be a vector field in $\Xi$ with flow $\Phi_t$, and consider its ``horizontal projection'' 
\[
Y^h = d\pi \circ Y.
\]
Given any $z \in \Xi$ we consider the maximally extended integral curve $\gamma_z: [-\tau_-(z), \tau_+(z)] \to \Xi$, $\gamma_{z}(t) = \Phi_t(z)$, so that $\dot{\gamma}_{z}(t) = Y(\gamma_{z}(t))$ and $\gamma_{z}(0) = z$. Note that the functions $\tau_{\pm}$ are not necessarily continuous and in general $\tau_{\pm}(z) \in [0,\infty]$, though we will only consider $z$ for which $\tau_{\pm}(z) < \infty$. For any curve $\gamma_z(t)$, there is a corresponding base space curve $x_{z}(t) = \pi(\gamma_{z}(t))$ whose tangent vector is $\dot{x}_{z}(t) = Y^h(\Phi_t(z)))$.

Let $\mG$ be a submanifold of $\Xi$ such that $Y$ is never tangent to $\mG$ (e.g.\ $\mG$ could be an open subset of $\p_+ \Xi$). We consider the set  
\[
Z = \{ (z, x_z(t)) \mid z \in \mG, \ t \in (-\tau_-(z),\tau_+(z)) \} \subset \mG \times M.
\]
The following converse to Lemma \ref{lemma_dfb_vf} shows that under certain conditions $Z$ is a double fibration. Each part below exactly corresponds to the related part in Definition \ref{def_dfrt}.

\begin{Lemma}
\label{lem: double fibration condition}
Let $D = \{ (z, t) \mid z \in \mG, \ t \in (-\tau_-(z),\tau_+(z)) \}$ and consider the map
\[
F: D \to \mG\times M, \ \ F(z,t) := (z, x_z(t)).
\]

\begin{enumerate}
\item[(i)] 
$F$ maps into $\mG \times \mi$ iff $x_z(t) \in \mi$ for $z \in \mG$ and $t \in (-\tau_-(z), \tau_+(z))$.
\item[(ii)] 
$F$ is injective iff the curves $x_z(t)$ do not self-intersect for $z \in \mG$.
\item[(iii)]
$F$ is an immersion iff $\dot{x}_z(t) \neq 0$ for $(z,t) \in D$.
\item[(iv)]
If the conditions in {\rm (i)--(iii)} hold and additionally $\tau_{\pm}(z) < \infty$ for all $z \in \mG$, then $F$ is an embedding of $D$ into $\mG \times \mi$.
\item[(v)]
If the conditions in {\rm (i)--(iv)} hold and additionally $V_z(t) \oplus \mR Y^h = T_{x_z(t)} M$ for all $z \in \mG$ and all $t$, then $\pi_M: Z \to \mi$ is a submersion and  $Z = F(D) \subset \mG \times \mi$ is a double fibration. In this case,
\begin{eqnarray}
\label{eq: N*Z of flows}
N^*Z = \{(z, (\pi \circ \Phi_t)^*\eta, x_z(t), -\eta) \mid z \in \mG, \ t\in (-\tau_-(z),\tau_+(z)), \ \eta(\dot x_z(t)) = 0 \}.
\end{eqnarray}
If $\mG$ is an open subset of $\p \Xi$, or more generally if $\dim(\mG) = \dim(\Xi)-1$, the condition $V_z(t) \oplus \mR Y^h = T_{x_z(t)} M$ holds automatically.
\end{enumerate}
\end{Lemma}
\begin{proof}
(i) is clear. For (ii), one has $F(z_1,t_1) = F(z_2,t_2)$ iff $z_1=z_2=z$ and $x_z(t_1) = x_z(t_2)$. This shows (ii).

(iii) Let $(z_s, t_s)$ be a curve in $D$. Since $x_z(t) = \pi(\Phi_t(z))$, we have 
\begin{eqnarray}
\label{eq: dF}
dF(\dot{z}_0, \dot{t}_0) = (\dot{z}_0, \dot{x}_{z_0}(t_0) \dot{t}_0 + d\pi(d\Phi_{t_0}(\dot{z}_0))).
\end{eqnarray}
We see that $dF|_{(z_0,t_0)}$ is injective iff $\dot{x}_{z_0}(t_0) \neq 0$.

(iv) Suppose that $\tau_+(z) < \infty$ for all $z \in \mG$, and let $(z_j, t_j)$ be a sequence in $D$ with $(z_j, x_{z_j}(t_j)) \to (z,x)$ in $\mG \times \mi$. Since $\tau_+$ is upper semicontinuous by Lemma \ref{lem: usc}, we have $\limsup \tau_+(z_j) \leq \tau_+(z) < \infty$. Thus $t_j \leq \tau_+(z)$, and similarly $t_j \geq -\tau_-(z)$. In particular $(t_j)$ is a bounded sequence, and hence after replacing $(t_j)$ by a subsequence we have  $t_j \to t_0 \in [-\tau_-(z),\tau_+(z)]$. This implies that $x = x_z(t_0)$, and $t_0 \in (-\tau_-(z),\tau_+(z))$ since $x \in \mi$. It follows that $F(D)$ is closed in $\mG \times \mi$. Moreover, if the conditions in (i)--(iii) also hold, then $F$ is an embedding.

(v) Since $F$ is an embedding, any smooth curve in $Z$ is of the form $\rho_s = F(z_s, t_s) = (z_s, x_{z_s}(t_s))$ where $z_s$ and $t_s$ are smooth curves in $\mG$ and $\mR$, respectively. Then 
\[
d\pi_{\mG}(\dot{\rho}_0) = \dot{z}_0,
\]
which shows that $d\pi_{\mG}$ is surjective. Similarly, 
\begin{eqnarray}
\label{transverse of dpiM}
d\pi_M(\dot{\rho}_0) = \p_s(\pi(\Phi_{t_s}(z_s))|_{s=0} = d\pi(Y(\Phi_{t_0}(z_0)) \dot{t}_0 + d \Phi_{t_0}(\dot{z}_0)).
\end{eqnarray}
According to Definition \ref{def_variations}, $d\pi(d \Phi_{t}(\dot{z}_0)) = J_{\dot{z}_0}(t)$. We also note for later purposes that 
\begin{equation} \label{y_phit_eq}
Y(\Phi_{t}(z)) = \p_t \Phi_t(z) = \p_s(\Phi_t(\Phi_s(z)))|_{s=0} = d\Phi_{t}(Y(z)).
\end{equation}
This shows that 
\[
d\pi_M(\dot{\rho}_0) = J_{\dot{z}_0}(t_0) + \dot{t}_0 Y^h(\Phi_{t_0}(z_0)).
\]
The assumption $V_z(t) \oplus \mR Y^h = T_{x_z(t)} M$ for all $z, t$ ensures that $\pi_M$ is a submersion. Thus $Z$ is a double fibration. The formula \eqref{eq: N*Z of flows} for $N^* Z$ follows since vectors in $T_{(z,x)}Z$ take on the form \eqref{eq: dF} (see also Lemma \ref{lemma_nstarz}). Finally, if $\dim(\mG) = \dim(\Xi)-1$, then $T_z\mG + \mR Y = T_z \Xi$ by the assumption that $Y$ is never tangent to $\mG$. Since $\Phi_t$ is a flow, $d\Phi_t$ is an isomorphism. By \eqref{transverse of dpiM} and \eqref{y_phit_eq} we have $d\pi_M(\dot{\rho}_0) = d\pi(d\Phi_{t_0}(Y(z_0)\dot{t}_0 + \dot{z}_0))$ and thus $V_z(t) \oplus \mR Y^h = T_{x_z(t)} M$.
\end{proof}

\begin{Lemma}
\label{lem: usc}
The map $z\mapsto \tau_+(z)$ is upper semicontinuous.
\end{Lemma}
\begin{proof}
Let $z_j \to z$ and suppose by contradiction that $\limsup \tau_+(z_j) > \tau_+(z)$. Without loss of generality this means that there exists an $\epsilon_0>0$ such that for all  $\epsilon\in(0,\epsilon_0)$ we have that $\tau_+(z_j) > \tau_+(z) +\epsilon$ for all $j\in \mathbb N$. This means that there exists a bounded sequence $t_j\in(\tau_+(z) +\epsilon, \tau_+(z_j))$ with $t_j\to \tau_+(z)+\epsilon$ such that $\Phi_{t_j}(z_j) \in \Xi$. By the joint continuity of $\Phi_{\cdot}(\cdot)$ we get that $\Phi_{\tau_+(z) +\epsilon}(z)\in \Xi$ since $\Xi$ is closed. This is true for all $\epsilon \in(0,\epsilon_0)$, which contradicts the definition of $\tau_+(z)$.
\end{proof}

\subsection{Ray transforms satisfying the Bolker condition}

We will next look at the Bolker condition in the special case of ray transforms. Let $z_s$ be a smooth curve in $\mG$ with $z_0 = z$ and $\partial_s z_s|_{s=0} = w$. In this subsection we slightly change notation and denote by $\tilde{J}_w(x)$  the variation field along $G_{z}$ as in Definition \ref{def_variation_general}. We can relate $\tilde{J}_w$ with the vector field $J_w: (-\tau_-(z), \tau_+(z)) \to T M$ along $x_{z}$ from Definition \ref{def_variations}, given by 
\[
J_w(t) = \p_s x_{z_s}(t)|_{s=0} = d\pi(d \Phi_t(w)).
\]
For all $\eta \in N_{x_z(t)}^* G_z$, since $\eta(J_w(t)) = \eta((\pi \circ \Phi_t|_{\mG})_* w) = (\pi \circ \Phi_t|_{\mG})^* \eta(w) = (A(z,x_z(t))\eta)(w)$, we have 
\[
\eta(J_w(t)) = \tilde{J}_w(x_z(t))(\eta).
\]
Thus $\tilde{J}_w(x_z(t)) = 0$ iff $J_w(t)$ is tangential, i.e.\ $J_w(t) \parallel \dot{x}_z(t)$. This means that $\tilde{J}_w(x_z(t))$ can be identified with $J_w(t)$, when the latter is considered as an element of $T_{x_z(t)} M / T_{x_z(t)} G_z$. 

Recall from Definition \ref{def_hx_vts_ray} the space 
\[
V_{z}(t, s) = \{ J_w(t) \mid w \in T_{z} \mG, \ J_w(s) \parallel \dot{x}_z(s) \},
\]
If $t \neq s$, we say that $x_z(t)$ and $x_z(s)$ are \emph{$Z$-conjugate} along $x_z$ if $V_{z}(t, s) + \mR Y^h$ is a strict subspace of $T_{x_z(t)} M$. Now $V_z(t,s) + \mR Y^h = T_{x_z(t)} M$ iff $V_z(x_z(t), x_z(s))$ is all of $(N_{x_z(t)}^* G_z)^*$, which implies that this notion of $Z$-conjugate points is equivalent to the one in Definition \ref{def_variation_general}. From Lemma \ref{lemma_ncp_general}, we see that the following are equivalent:
\begin{enumerate}
\item[(a)]
$x_z(t)$ and $x_z(s)$ are not $Z$-conjugate along $x_z$.
\item[(b)] 
$V_{z}(t,s) + \mR Y^h = T_{x_z(t)} M$.
\item[(c)] 
The space $\{ w \in T_z \mG \mid J_w(t), J_w(s) \text{ are tangential} \}$ has dimension $N-(2n-2)$.
\end{enumerate}

\begin{Remark} \label{rmk_cp_geodesics}
It follows that when $N < 2n-2$, all pairs of points on any $x_z$ are $Z$-conjugate. Moreover, when $N=2n-2$ the notion of two points being $Z$-conjugate is equivalent to the existence of a nontrivial variation field that is tangential at the two points.

In particular, if $\Xi = SM$ and $Y = X_g$ is the geodesic vector field, then we have the geodesic X-ray transform. Variation fields are standard Jacobi fields $J_w(t) = d\pi(d\Phi_t(w))$, but $w$ is now restricted to lie in $T_z \mG$ instead of $T_z(TM)$. Note that the tangential Jacobi field $\dot{x}_z(t)$ formally corresponds to taking $w=Y$, and $t \dot{x}_z(t)$ corresponds to the radial direction in $T_z(TM)$. In the usual case where $\dim(\mG) = 2n-2$ and $\mG$ is never tangent to $Y$, $Z$-conjugate points for the geodesic X-ray transform in the sense defined above correspond precisely to conjugate points for geodesics in the standard sense. Indeed, (c) above fails iff there is $w \in T_z \mG \setminus 0$ such that $J_w(t)$ and $J_w(s)$ are tangential to $x_z$, and after subtracting a linear combination of $\dot{x}_z(r)$ and $r \dot{x}_z(r)$ from $J_w(r)$ this is equivalent with having a nontrivial Jacobi field such that $J(t) = J(s) = 0$. 
\end{Remark}

The following result shows that the Bolker condition can indeed be characterized as in Definition~\ref{def_bolker_ray}.

\begin{Lemma} \label{lemma_bolker_ray}
Let $R$ be a double fibration ray transform and let $(z,\zeta,x,\eta) \in C$ with $x = x_z(t)$.
\begin{enumerate}
\item[(a)]
$\pi_L^{-1}(z,\zeta) = \{ (z,\zeta,x,\eta) \}$ iff $V_z(t,s)$ is not annihilated by $\eta$ for any $s \neq t$.
\item[(b)] 
$\pi_L$ is injective iff there are no pairs of $Z$-conjugate points on any $x_z$.
\item[(c)] 
$d\pi_L|_{(z,\zeta,x,\eta)}$ is injective iff $d(Y^h(x,\xi(\,\cdot\,)))(T_{z} H_{x})$ is not annihilated by $\eta$, where $(x,\xi(\tilde{z}))$ is the unique point over $x$ on the integral curve through $\tilde{z} \in H_{x}$.
\item[(d)] 
Let $\dim(\mG) = \dim(\Xi)-1$. Then $d\pi_L|_{(z,\zeta,x,\eta)}$ is injective iff $d(Y^h(x,\,\cdot\,))(T_{(x,\xi(z))} \Xi_x)$ is not annihilated by $\eta$.
\end{enumerate}
\end{Lemma}
\begin{proof}
Part (a) follows from Lemma \ref{l:no_conjugate_points} and the identification of $J_w(t)$, modulo tangential vectors, with $\tilde{J}_w(x_z(t))$. Part (b) follows from part (a).

Part (c) is a consequence of Lemma \ref{lemma_bolker_local_char} part (4) with the choice $F: H_x \to T_x M$, $F(\tilde{z}) = Y^h(\Phi_{t(\tilde{z})}(\tilde{z}))$ where $t(\tilde{z})$ is the unique time such that $x_{\tilde{z}}(t(\tilde{z})) = x$. The map $\tilde{z} \mapsto t(\tilde{z})$ is smooth by the implicit function theorem, and since $\Phi_{t(\tilde{z})}(\tilde{z}) = (x, \xi(\tilde{z}))$ it follows that also $\xi: H_x \to \Xi_x$ is smooth.

For part (d), note that when $\dim(\mG) = \dim(\Xi)-1$, one has $\dim(H_x) = \dim(\Xi_x)$ and thus $\xi$ maps between spaces of the same dimension. If $z_s$ is a curve in $H_x$ with $z_0=z$, we use \eqref{y_phit_eq} to obtain 
\begin{align*}
d\xi(\dot{z}_0) &= \p_s(\Phi_{t(z_s)}(z_s))|_{s=0} = d\Phi_{t(z_0)}(\dot{z}_0) + Y(\Phi_{t(z_0)}(z_0)) \p_s(t(z_s))|_{s=0} \\
 &= d\Phi_{t(z_0)}(\dot{z}_0 + Y(z_0) \p_s(t(z_s))|_{s=0}).
\end{align*}
Now $d\Phi_t$ is an isomorphism and $Y$ is never tangent to $H_x$. Thus $d\xi$ is injective and hence invertible. We obtain 
\[
d(Y^h(x,\xi(\,\cdot\,)))(T_{z} H_{x}) = d(Y^h(x,\,\cdot\,))(T_{(x,\xi(z))} \Xi_x).
\]
The result follows from (c).
\end{proof}

\subsection{Example 1: Geodesic X-ray transform} \label{sec_ex_xray}

Let us consider the geodesic X-ray transform (see \cite{PSU_book} for more details). Let $(M,g)$ be an oriented manifold with smooth boundary, let $SM = \{ (x,v) \in TM \mid \abs{v}_g = 1 \}$ be the unit sphere bundle, and let $\p_+ SM = \{ (x,v) \in SM \mid x \in \p M, \ g(v,\nu) < 0 \}$ where $\nu$ is the unit outer normal to $\p M$. We take $\Xi = SM$ and $Y = X_g$, where $X_g$ is the geodesic vector field. Let $\mG$ be a submanifold of $\p_+ SM$, so $Y$ is never tangent to $\mG$. Let also $\kappa \in C^{\infty}(\mG \times \mi)$ be nowhere vanishing. Then the curves $x_z(t)$ are unit speed geodesics, and the weighted geodesic X-ray transform is given by 
\begin{equation} \label{grt_def}
R f(z) = \int_0^{\tau_+(z)} \kappa(z, x_z(t)) f(x_z(t)) \,dt, \qquad z \in \mG.
\end{equation}

Let us assume the following conditions for all $z \in \mG$:
\begin{enumerate}
\item[(i)]
(No tangential intersections) $x_z(t) \in \mi$ for $t \in (0,\tau_+(z))$;
\item[(ii)]
(No self-intersections) $t \mapsto x_z(t)$ is injective;
\item[(iv)]
(Nontrapping) $\tau_+(z) < \infty$; 
\item[(v)] 
(Enough variations) $V_z(t) + \mR \dot{x}_z(t) = T_{x_z(t)} M$ for $t \in (0,\tau_+(z))$.
\end{enumerate}
Under these conditions, Lemma \ref{lem: double fibration condition} ensures that $R$ is a ray transform coming from a double fibration. The no singular points condition (iii) always holds since $\abs{\dot{x}_z(t)} = 1$. Condition (i) holds e.g.\ if $\p M$ is strictly convex. Conditions (ii) and (iv) hold e.g.\ if $(M,g)$ is a simple manifold, but they may fail for trapping manifolds such as the catenoid. As noted above, (v) is automatically satisfied if $\mG$ is an open subset of $\p_+ SM$, i.e.\ $\dim(\mG) = 2n-2$. Condition (v) holds also in many other cases, e.g.\ if $M \subset \mR^n$, $g$ is the Euclidean metric, and $\mG$ consists of all lines whose direction vector is orthogonal to $e_n$.

Let us next study the Bolker condition at $(z,\zeta,x,\eta) \in C$ where $x = x_z(t)$. In the usual case where $\dim(\mG) = 2n-2$, the map $d\pi_L$ is always injective by Lemma \ref{lemma_bolker_ray} part (d) since $Y^h(x,\,\cdot\,)$ is the identity map. Moreover, if the geodesics through $z \in \mG$ have no conjugate points in the usual sense, then $\pi_L$ is injective by Lemma \ref{lemma_bolker_ray} part (b) and Remark \ref{rmk_cp_geodesics}. Even if there are conjugate points or $\dim(\mG) < 2n-2$, the Bolker condition might still hold at some points $(z,\zeta,x,\eta) \in C$, but then one has to verify the conditions in Lemma \ref{lemma_bolker_ray} parts (a) and (c) instead.

We collect some of the above results in the following proposition, formulated in terms of a fixed geodesic $x_{z_0}$ that is never tangent to $\p M$ (hence also geodesics $x_z$ for $z$ near $z_0$ are never tangent). This result is already contained e.g.\ in \cite{StefanovUhlmann}.

\begin{Proposition} \label{geodesic bolker}
Let $(M,g)$ be a manifold with smooth boundary, let $z_0 \in \p_+ SM$ satisfy $\tau_+(z_0) < \infty$, and assume that the geodesic $x_{z_0}: [0,\tau_+(z_0)] \to M$ does not self-intersect, meets $\p M$ transversally at the endpoints, and otherwise stays in $\mi$. Let also $\kappa \in C^{\infty}(\mG \times \mi)$ be nowhere vanishing. If $\mG$ is a sufficiently small neighborhood of $z_0$ in $\p_+ SM$, then the geodesic ray transform given by \eqref{grt_def} is a double fibration ray transform. If there are no conjugate points along $x_{z_0}$, then the Bolker condition is satisfied at every $(z_0,\zeta, x,\eta) \in C$ where $x = x_{z_0}(t)$ for some $t$.
\end{Proposition}
\begin{proof}
The only thing to check is that if $x_{z_0}$ does not self-intersect, then $x_z$ does not self-intersect for $z$ close to $z_0$. We argue by contradiction and suppose that $z_j \to z_0$ but $x_{z_j}(t_j) = x_{z_j}(s_j)$ for some $t_j, s_j \in [0,\tau_+(z_j)]$ with $t_j \neq s_j$. Since $x_{z_0}(t)$ meets $\p M$ transversally at the endpoints, $\tau_+$ is smooth near $z_0$, and hence by compactness, after passing to a subsequence, we have $t_j \to t_0$ and $s_j \to s_0$. Since $x_{z_0}$ is injective, we must have $t_0=s_0$. In local coordinates near $x_{z_0}(t_0)$ we write 
\[
x_z(t) - x_z(s) = \dot{x}_z(s)(t-s) + O((t-s)^2)
\]
where the implied constant depends on second derivatives of $(z,t) \mapsto x_z(t)$. Since these second derivatives are bounded, for $j$ large one has 
\[
0 = \abs{x_{z_j}(t_j)-x_{z_j}(s_j)} \geq \frac{1}{2} \abs{\dot{x}_{z_j}(s_j)} \abs{t_j-s_j}.
\]
This is a contradiction since $t_j \neq s_j$ and since $\abs{\dot{x}_{z_j}(s_j)}$ has a uniform lower bound for $j$ large.
\end{proof}

\subsection{Example 2: Null bicharacteristic ray transform} \label{subsect: null bich transform}

Let $M$ be an $n$-manifold with smooth boundary, let $\pi: T^* M \setminus 0 \to M$ be the natural projection, and let $p \in C^{\infty}(T^* M \setminus 0)$. Let $H_p$ be the Hamilton vector field, and note that the fiberwise derivative $\nabla_{\xi} p$ is related to the Hamilton vector field $H_p$ by 
\[
\nabla_{\xi} p = d\pi(H_p).
\]
Then $\nabla_{\xi} p(x,\xi) \in T_x M$, and in local coordinates $\nabla_{\xi} p(x,\xi) = \p_{\xi_j} p(x,\xi) \p_{x_j}$. The Hessian of $p(x,\xi)$ with respect to $\xi$ is defined as the derivative  
\[
\nabla_{\xi}^2 p(x,\xi) = d(\nabla_{\xi} p(x,\,\cdot\,))|_{\xi}.
\]
It follows that $\nabla_{\xi}^2 p(x,\xi): T_x^* M \to T_x M$. In local coordinates 
\[
\nabla_{\xi}^2 p(x,\xi) \theta = \p_{\xi_j \xi_k} p(x,\xi) \theta_k \p_{x_j}.
\]

Define $\Xi = p^{-1}(0) = \{ (x,\xi) \in T^* M \setminus 0 \mid p(x,\xi) = 0 \}$. We make the standing assumption that 
\[
\nabla_{\xi} p \neq 0 \text{ everywhere on $\Xi$.}
\]
Then $\Xi$ is a smooth fiber bundle over $M$, and each fiber $\Xi_x = \{ \xi \in T_x^* M \setminus 0 \mid p(x,\xi) = 0 \}$ is a smooth $(n-1)$-dimensional submanifold of $T_x^* M$. For any $\xi \in \Xi_x$, we have the identification 
\begin{equation} \label{t_xi_xix}
T_{\xi} \Xi_x = \{ \eta \in T_x^* M \mid \eta(\nabla_{\xi} p(x,\xi)) = 0 \} = (\nabla_{\xi} p(x,\xi))^{\perp}.
\end{equation}

Let $Y = H_p|_{\Xi}$ be the Hamilton vector field, which is tangent to $\Xi$. As before, let $\Phi_t$ be the flow of $Y$ with $\gamma_z(t) = \Phi_t(z)$ and $x_z(t) = \pi(\gamma_z(t))$. Let $\mG$ be an open subset of $\p_+ \Xi$. Let also $\kappa \in C^{\infty}(\mG \times \mi)$ be a nowhere vanishing function. This setting gives rise to the weighted null bicharacteristic ray transform 
\begin{equation} \label{nbrt_def_later}
Rf(z) = \int_{-\tau_-(z)}^{\tau_+(z)} \kappa(z, x_z(t)) f(x_z(t)) \,dt, \qquad z \in \mG,
\end{equation}
whenever this expression is well defined. Since $\dot{x}_z = \nabla_{\xi} p$ is nonvanishing on $\Xi$, the curves $x_z$ have no singular points. Moreover, since $\dim(\mG) = 2n-2$, the variation condition (v) in Lemma \ref{lem: double fibration condition} is satisfied. Then $R$ will be a double fibration ray transform if we assume the following:
\begin{enumerate}
\item[(i)]
(No tangential intersections) $x_z(t) \in \mi$ for $z \in \mG$ and $t \in (-\tau_-(z),\tau_+(z))$;
\item[(ii)] 
(No self-intersections) $t \mapsto x_z(t)$ is injective for all $z \in \mG$;
\item[(iv)]
(Nontrapping) $\tau_{\pm}(z) < \infty$ for all $z \in \mG$.
\end{enumerate}

Next we consider the Bolker condition. From now on we assume that $p$ is homogeneous of degree $m$ in $\xi$, which will make the statements below cleaner. The following lemma shows that for the null bicharacteristic ray transform, there is always one normal direction missing from the possible directions of variation fields (but on the other hand tangential directions are always there).

\begin{Lemma} \label{lemma_null_xit}
Let $p$ be homogeneous in $\xi$. If $z \in \mG$ and $\Phi_t(z) = (x(t), \xi(t))$, then $\xi(t) \perp V_z(t,0)$ and $\dot{x}_z(t) \in V_z(t,0)$ for $t \neq 0$.
\end{Lemma}
\begin{proof}
Let $w \in T_z \mG$ and $J_w(t) = d\pi(d\Phi_t(w))$, and suppose that $J_w(0) \parallel \dot{x}_z(0)$, i.e.\ $J_w(0) \parallel \nabla_{\xi} p$. Let $\lambda|_{(x,\xi)} = \xi_j \,dx^j$ be the canonical $1$-form on $T^* M$, and note that 
\[
f(t) := \xi(t)(J_w(t)) = \lambda(d\Phi_t(w)) = (\Phi_t^* \lambda)(w).
\]
One has $f(0) = 0$ since $\xi \perp \nabla_{\xi} p$ on $\Xi$ by the homogeneity relation $\xi_j \p_{\xi_j} p = mp$. We use Cartan's formula and the homogeneity relation again to obtain that 
\[
\p_t(\Phi_t^* \lambda)|_{t=0} = \mathcal{L}_{H_p} \lambda = (d i_{H_p} + i_{H_p} d)\lambda = d(mp) + d\lambda(H_p,\,\cdot\,) = (m-1)dp.
\]
Thus we have 
\[
f'(t) = \p_s((\Phi_{t+s}^* \lambda)(w))|_{s=0} = (m-1) dp(d\Phi_t(w)).
\]
Since $d\Phi_t(w)$ is tangential to $\Xi$, it follows that $f'(t) = 0$. We have shown that $f(t) = \xi(t)(J_w(t)) = 0$ for all $w \in T_z \mG$, which implies that $\xi(t) \perp V_z(t,0)$.

Let now $z = (x,\xi_0)$ and let $w = \dot{z}_0 \in T_z \mG$ be a radial vector where $z_s = (x, (1+s)\xi_0) \in \mG$ by homogeneity. Since the curve $x_{z_s}(t)$ is obtained from $x_z(t)$ by reparametrization, we see that $J_w(t) = \p_s(x_{z_s}(t))|_{s=0}$ is always parallel to $\dot{x}_z(t)$.
\end{proof}

Motivated by Lemma \ref{lemma_null_xit}, we say that two points $x_z(t)$ and $x_z(s)$ are \emph{not conjugate} along $x_z$ if $V_z(t,s) = \ker(\xi(t))$, i.e.\ all directions of variation in the kernel of $\xi(t)$ are possible. This is equivalent to surjectivity of the map $d\pi \circ d\Phi_t: \{ w \in T_z \mG \mid J_w(s) \parallel \dot{x}_z(s) \} \to \ker(\xi(t))$. Since this is a map between $(n-1)$-dimensional spaces and since one can add a radial vector to $w$, we see that 
\[
\text{$x_z(t)$ and $x_z(s)$ are conjugate iff there is $w \in T_z \mG$, $w \neq 0$, with $J_w(t) = J_w(s) = 0$.}
\]
See \cite{LOSU} for an analysis of conjugate points in the case of the light ray transform.

We can now show that if there are no conjugate points and a nondegeneracy condition holds, then the Bolker condition at $(z,\zeta,x,\eta) \in C$ with $\Phi_t(z) = (x,\xi)$ holds precisely when $\eta \nparallel \xi$.

\begin{Lemma}
Let $p$ be homogeneous in $\xi$, and assume that 
\begin{equation} \label{null_nondeg}
\nabla_{\xi}^2 p \text{ is nondegenerate on $\Xi$.}
\end{equation}
Let $(z,\zeta,x,\eta) \in C$ where $\Phi_t(z) = (x,\xi)$. Then $d\pi_L|_{(z,\zeta,x,\eta)}$ is injective iff $\eta \nparallel \xi$. If $\eta \nparallel \xi$ and there are no conjugate points on $x_z$, then $\pi_L^{-1}(z,\zeta) = \{(z,\zeta,x,\eta)\}$.
\end{Lemma}
\begin{proof}
Since $Y^h = \nabla_{\xi} p$, we have $d(Y^h(x,\,\cdot\,))|_{\xi}(\theta) = \nabla_{\xi}^2 p(x,\xi) \theta$. Using Lemma \ref{lemma_bolker_ray} part (d) and \eqref{t_xi_xix}, we obtain that 
\[
\text{$d\pi_L|_{(z,\zeta,x,\eta)}$ is injective iff $\nabla_{\xi}^2 p(x,\xi) \eta \nparallel \nabla_{\xi} p(x,\xi)$.}
\]
Since $p$ is homogeneous, we have $\p_{\xi_j \xi_k} p \xi_k = (m-1) \p_{\xi_j} p$, which gives $\nabla_{\xi}^2 p(x,\xi)\xi = (m-1) \nabla_{\xi} p$. Since $\nabla_{\xi}^2 p(x,\xi)$ is assumed to be nondegenerate, we must have $m \neq 1$ and $(\nabla_{\xi}^2 p)^{-1}\nabla_{\xi} p \parallel \xi$. Consequently $d\pi_L|_{(z,\zeta,x,\eta)}$ is injective iff $\eta \nparallel \xi$. If $x_z$ has no conjugate points and $\eta \nparallel \xi$, then $V_z(t,s) = \ker(\xi(t))$ for $t \neq s$ and the second statement follows from Lemma \ref{lemma_bolker_ray}.
\end{proof}

We also show that \eqref{null_nondeg} ensures the absence of conjugate points on short curves.

\begin{Lemma} \label{lem_ncp_short}
Assume that \eqref{null_nondeg} holds. For any $z \in \Xi$, there is $\eps > 0$ so that $x_z|_{[-\eps,\eps]}$ has no conjugate points.
\end{Lemma}
\begin{proof}
Let $w \in T_z \Xi \setminus 0$ satisfy $J_w(0) = 0$, i.e.\ $d\pi(w) = 0$. Then $w = \dot{z}_0$ where $z_s = (x,\xi_s) \in \Xi_x$, so we can identify $w$ with an element of $T_x^* M$. If we fix some Riemannian metric on $M$ and denote by $D_t$ the covariant derivative, we have 
\[
D_t J_w(0) = D_t \p_s x_{z_s}(t)|_{s=t=0} = D_s(\nabla_{\xi} p(\Phi_t(z_s)))|_{s=t=0} = D_s(\nabla_{\xi} p(x,\xi_s))|_{s=0} = \nabla_{\xi}^2 p(x,\xi) w.
\]
By \eqref{null_nondeg} we have $D_t J_w(0) \neq 0$. It follows that for some $\eps > 0$, $J_w(t) \neq 0$ when $t \in [-\eps,\eps] \setminus 0$. This shows that $x_z|_{[-\eps,\eps]}$ has no points conjugate to $x_z(0)$. We can repeat this argument starting at $x_z(\rho)$ instead of $x_z(0)$ using a related variation field $J_{w_{\rho}}(t;\rho)$ with $w_\rho \in T_{\Phi_\rho(z) }\Xi$ with $\abs{w_{\rho}}=1$ in some fixed Riemannian metric. This implies by continuity that for some $\eps > 0$, if $\rho, t\in [-\eps,\eps]$ and $t\neq \rho$, one cannot have $J_w(t ) = J_w(\rho) = 0$ for some $w\in T_z\Xi$.
\end{proof}

The following result collects some of the facts given above (again, the only thing to prove is that the curves $x_z$ do not self-intersect for $z$ close to $z_0$, and this proceeds as in Proposition \ref{geodesic bolker}).

\begin{Proposition} \label{prop_null}
Let $M$ be a manifold with smooth boundary, let $p \in C^{\infty}(T^* M \setminus 0)$, and let $\nabla_{\xi} p$ be nonvanishing on $\Xi = p^{-1}(0)$. Let also $\kappa \in C^{\infty}(\mG \times \mi)$ be nowhere vanishing. Suppose that $z_0 \in \p_+ \Xi$ satisfies $\tau_{+}(z_0) < \infty$ and that the curve $x_{z_0}$ meets $\p M$ transversally at the endpoints, otherwise stays in $\mi$, and does not self-intersect. If $\mG$ is a sufficiently small neighborhood of $z_0$ in $\p_+ \Xi$, then the null bicharacteristic ray transform given by \eqref{nbrt_def_later} is a double fibration ray transform. If $p$ is homogeneous in $\xi$, \eqref{null_nondeg} holds and there are no conjugate points along $x_{z_0}$, then the Bolker condition is satisfied at every $(z_0,\zeta, x,\eta) \in C$ where $\Phi_t(z_0) = (x,\xi)$ for some $t$ and $\eta \nparallel \xi$.
\end{Proposition}

\section{Recovering the analytic wave front set} \label{sec_wf}

In this section we prove Theorem \ref{global elliptic regularity intro}, which we restate here for the reader's convenience.

\begin{Theorem}
\label{global elliptic regularity}
Let $R$ be an analytic double fibration transform as in Definition \ref{def_doublefibration_general}, and assume that the Bolker condition holds at $(\hat z, \hat \zeta,  \hat x, \hat \eta)\in C$. Then for any $f\in {\mathcal E}'(\mx)$, we have 
\[
(  \hat z, \hat \zeta) \notin \WF_a(Rf) \implies ( \hat x,  \hat \eta) \notin \WF_a(f).
\]
\end{Theorem}

\subsection{The model operator}

We introduce here a model operator relevant for Theorem \ref{global elliptic regularity}, related to the local representation $Z = \{ x'' = \phi(z,x') \}$ of a double fibration given in Lemma \ref{local description of Z}. Write $x = (x', x'') \in \mR^{n' + n''}$ and $z\in \mR^N$ with $N\geq n'+n''$. We recall that $n = n' + n'' = \dim(\mx)$, $N = \dim(\mG)$, and $N + n'=\dim(Z)$. We assume that the map
\begin{align*}
\phi:  \mR^N \times \mR^{n'} & \to \mR^{n''}
\end{align*}
is analytic in a neighborhood $V \times U'$ of a fixed point $(\hat{z}, \hat{x}')$, where $V \subset \mR^N$ is a neighborhood of $\hat{z}$ and $U' \subset \mR^{n'}$ is a neighborhood of $\hat{x}'$. Define 
\begin{eqnarray}
\label{algebraic def of Z}
Z:=\{(z,x', x'')\in V \times U' \times \mR^{n''}\mid x''=\phi(z,x')\}.
\end{eqnarray}
We choose a neighborhood $U''$ of $\hat{x}'' := \phi(\hat{z}, \hat{x}')$ in $\mR^{n''}$ such that if we set $U := U' \times U''$, then $Z \subset V\times U$.

We define the model operator $T$ as the FIO 
\[
Tf(z) = \int_U T(z,x) f(x) \,dx \in \mathcal D'(V)
\]
for $f\in {\mathcal E}'(U)$. The kernel $T(z,x) \in \mathcal D'(V \times U)$ is given by the oscillatory integral
\begin{eqnarray}
\label{def T}
T(z,x) := \int_{\mR^{n''}} e^{i(\phi(z,x') - x'') \cdot \eta''}a(z,x) \,d\eta''
\end{eqnarray}
where $a(z,x)$ is real valued and analytic on $V\times U$. That is, the Schwartz kernel of $T$ is a conormal distribution associated with the conormal bundle of $Z$, and $a(z,x)$ is a nonvanishing amplitude defined on $V\times U$ that is independent of $\eta''$.

Let $C:= (N^* Z \setminus 0)'\subset  T^*\mR^N \times T^*\mR^{n' + n''}$ be the canonical relation associated to the phase function 
\[
\Phi:\mR^{N+n'+n''}\times\mR^{n''}\to\mR, \ \ \Phi(z, x, \eta'') := (\phi(z,x')- x'') \cdot \eta''.
\]
The conormal bundle $N^* Z$ is given by 
\begin{align}
\label{local canonical relation}
N^* Z&= \Big\{ \big((z, -\partial_z \Phi(z,x)), (x, \p_x \Phi(z,x))\big) \ \Big|\  \p_{\eta''} \Phi(z,x) = 0, z \in V, x \in U'\times\mR^{n''} \Big\} \\
 &=\Big\{\big(z, \phi_z(z,x')^T \eta'',x', \phi(z,x'),  \phi_{x'}(z,x')^T\eta'', -\eta'' \big) \ \Big|\  \eta'' \in \mR^{n''}, z\in V, x'\in U'  \Big\}. \notag
\end{align}

We will assume a local version of the Bolker condition (i.e.\ injectivity of $d\pi_L$, see Lemma \ref{lemma_bolker_local_char}): we assume that at $(\hat z, \hat x', \hat\eta'') \in V\times U'\times \mR^{n''}$ the $N \times (n'+n'')$ matrix 
\begin{eqnarray}
\label{phiz has full rank}
(\phi_z(\hat z, \hat x')^T,\ \partial_{x'}( \phi_z(\hat z,  x')^T\hat \eta'')\mid_{x'= \hat x'})\ \ {\rm has\ linearly\ independent\ columns}.
\end{eqnarray}

For elliptic FIOs $T$ as above, we will prove the following analytic regularity statement for distributions supported in a sufficiently small neighborhood of $\hat x$.

\begin{Theorem}
\label{elliptic regularity}
Assume 
\eqref{phiz has full rank}. There exists a neighborhood $\hat U\subset \subset U$ containing $\hat x$ such that for all $f\in {\mathcal E}'(\hat U)$, 
\[
(\hat z,\hat \zeta) \notin \WF_a(T f) \implies  (\hat x, \hat \eta)\notin \WF_a (f).
\]
\end{Theorem}

\subsection{Proof of Theorem \ref{global elliptic regularity}}

We will prove that Theorem \ref{global elliptic regularity} is a consequence of Theorem \ref{elliptic regularity}, by showing that the operators considered in Theorem \ref{elliptic regularity} are local representations of operators in Theorem \ref{global elliptic regularity}.

First we give a simple characterization of a double fibration that will ensure that when one uses \cite[Theorem 8.5.5]{Hormander} below, there will no extra elements appearing in the wave front set. One direction would already follow from Lemmas \ref{lemma_nstarz} and \ref{lem_df_b}.

\begin{Lemma}
Let $Z$ be an embedded submanifold in $\mG \times \mx$. The projections $\pi_{\mG}: Z \to \mG$ and $\pi_{\mx}: Z \to \mx$ are submersions if and only if $N^* Z$ satisfies 
\begin{equation}
\label{no flat spots}
\{ (z,\zeta, x, \eta) \in N^* Z \setminus 0 \mid \zeta = 0 \} =  \{ (z,\zeta, x, \eta) \in N^* Z \setminus 0 \mid \eta = 0 \} = \emptyset.
\end{equation}
\end{Lemma}
\begin{proof}
The map $d\pi_{\mG}|_{(z,x)} : T_{(z,x)}Z \to T_z\mG$ is surjective if and only if for any $v \in T_z \mG$ there is some $(z,v,x,u) \in T Z$. The last condition implies that any $(z,\zeta, x, 0) \in N^* Z$ must satisfy $\zeta=0$. So 
\[ \{ (z,\zeta, x, \eta) \in N^* Z \setminus 0 \mid \zeta = 0 \} = \emptyset. \]
Conversely, suppose 
\[ \{ (z,\zeta, x, \eta) \in N^* Z \setminus 0 \mid \zeta = 0 \} = \emptyset \]
but $d\pi_{\mG}|_{(z,x)}$ is not surjective. Then its range has codimension $\geq 1$ and hence there is $\zeta \in T_z^* \mG \setminus 0$ with $\zeta(d\pi_{\mG}(T_{z,x} Z)) = \{0\}$. It follows that $(z,\zeta,x,0) \in N^* Z$ with $\zeta \neq 0$ which is a contradiction. The argument for $d\pi_{\mx}$ is analogous.
\end{proof}

\begin{proof}[Proof of Theorem \ref{global elliptic regularity}] We first write $f = \psi f + (1-\psi)f$, where $\psi \in C^{\infty}_c(\mx)$ satisfies $\psi=1$ near $\hat{x}$ and $\supp(\psi)$ is contained in an analytic coordinate chart of $\mx$. Since we have  \eqref{no flat spots} and since $\pi_L^{-1}(\hat z,\hat \zeta) = \{(\hat z,\hat \zeta,\hat x,\hat \eta)\}$ by the Bolker condition, we can apply \cite[Theorem 8.5.5]{Hormander} to conclude that $(\hat z, \hat \zeta) \notin \WF_a(R((1-\psi)f))$. So we have that $(\hat z, \hat \zeta)\notin \WF_a(R(\psi f))$.

Working near $\hat x$ and $\hat z$, we may identify $\mx$ with part of $\mR^n$ and $\mathcal G$ with part of $\mR^N$.
By Lemma \ref{local description of Z}, there are open sets $V\subset \mR^N$ containing $\hat z$ and $ U'\subset \mR^{n'}$ containing $\hat x'$ such that the submanifold $Z\cap (V\times U)\subset \mR^N \times \mR^{n'+n''}$ can be described in analytic coordinates by $x'' = \phi(z,x')$ for some analytic function $\phi : V\times  U' \to U'' \subset \mR^{n''}$. Here $U := U'\times U''$ for some small set $U''\subset \mR^{n''}$. It was proved in \eqref{rzx_formula} that in such local coordinates, one has $R(\psi f) = T(\psi f)$ where $T$ is the model operator as in \eqref{def T}.

The Bolker condition and Lemma \ref{lemma_bolker_local_char} ensure that \eqref{phiz has full rank} holds.
So Theorem \ref{elliptic regularity} then applies to show that there is an open set $\hat U \subset \subset U$ such that for all $F\in {\mathcal E}'(\hat U)$, 
\begin{eqnarray}
\label{elliptic reg for locally supported}
(\hat z, \hat \zeta)\notin \WF_a(TF)\implies  (\hat x, \hat \eta)\notin \WF_a(F).
\end{eqnarray}
We proved above that $(\hat z, \hat \zeta) \notin \WF_a(T(\psi f))$, where $\psi$ can be chosen to be supported in $\hat{U}$. It follows that  $(\hat x, \hat \eta)\notin \WF_a(\psi f)$, which proves that $(\hat x, \hat \eta)\notin \WF_a(f)$ as required.
\end{proof}

\subsection{Wave packet decomposition and representation of $Tf$}
To prove Theorem \ref{elliptic regularity}, instead of oscillatory integral methods, we will appeal to wave packet decompositions and FBI transforms. We will use  conventions as in \cite{Martinez} (with $h$ replaced by $\lambda^{-1}$ and $\xi$ replaced $-\xi$).

From this point on, we will often denote elements of $T^*V$ by $(v_1, v_2)$ instead of $(z,\zeta)$ and elements of $T^* U$ by $(u_1, u_2)$ instead of $(x,\eta)$. We also set $\hat v_1 := \hat z$, $\hat u_1 := \hat{x}$ etc.\ for consistency in notation.

Let $T$ be given by \eqref{def T}, let $m(y) = c_{n} e^{-\frac{1}{2} \abs{y}^2}$ and $M(z) = c_{N} e^{-\frac{1}{2} \abs{z}^2}$ be Gaussians centred at the origin on $\mR^{n}$ and $\mR^N$, respectively, with $n =n'+n''$ and $c_n = 2^{-n/2} \pi^{-3n/4}$. For $\lambda >0$, $v = (v_1,v_2) \in T^*V\cong V\times \mR^N$ and $u = (u_1, u_2) \in T^*U \cong U \times \mR^{n}$, we consider Gaussian wave packets defined by $M^\lambda_v(z) := \lambda^{3N/4} e^{i\lambda z\cdot v_2} M(\lambda^{1/2}(z-v_1))$ and $m^\lambda_u(y) := \lambda^{3n/4} e^{i\lambda y\cdot u_2} m(\lambda^{1/2}(y-u_1))$.

For $\hat U \subset \subset U$ and for any compactly supported distribution $f\in \mathcal E'(\hat U)$ we can define the FBI transform $(L^\lambda_mf)(u)$ formally by
\[
(L^\lambda_mf)(u) := \int f(y) \overline{m^\lambda_u(y)} \,dy = \langle f, m_u^{\lambda} \rangle_{L^2}.
\]
The FBI transform $(L^\lambda_Mg)(v)$ for compactly supported distributions $g\in \mathcal E'(V)$ can be defined analogously. The FBI transform satisfies the inversion formula 
\begin{eqnarray}
\label{inversion}
f(x) = \int (L_m^\lambda f)(u) m^\lambda_u(x) du = \int \langle f, m_u^{\lambda} \rangle m_u^{\lambda}(x) \,du
\end{eqnarray}
and the same of course holds for the transform $L_M^\lambda$.
Let $T$ be an FIO defined by \eqref{def T} and $f\in \mathcal E'(\hat U)$. We can use \eqref{inversion} to deduce the formula
\begin{eqnarray}
\label{wave packet relation}
(L^\lambda_MTf)(v) = \int f(x) \left[ \int \langle Tm_u^\lambda, M_v^\lambda \rangle \ol{m_u^\lambda(x)} \,du \right] \,dx := \int f(x) K_{\lambda}(x,v) \,dx
\end{eqnarray}
where
\begin{eqnarray}
\label{def: K}
K_{\lambda}(x,v) =  \int \langle Tm_u^\lambda, M_v^\lambda \rangle \ol{m_u^\lambda(x)} \,du.
\end{eqnarray}
In Section \ref{sec_klambda_integral} we will prove that the kernel $K_{\lambda}(x,v)$ can be expressed as follows.

\begin{Proposition}
\label{prop: Klambda integral}
There exists a relatively compact open set $\mathcal U \subset \mR^{N-n''}$ and an analytic diffeomorphism 
$$z(\cdot, \cdot) : \mathcal U \times  U \to Z$$
such that
\begin{eqnarray}
\label{Klambda integral 2}
K_{\lambda}(x,v) = c_{n,N} \lambda^{\frac{3N}{4}}  \int\limits_{ \zeta' \in \mathcal U}  e^{i \lambda \Psi(\zeta', x, v)}\tilde a(\zeta',x) \,d\zeta'
\end{eqnarray}
where $\Psi(\zeta', x,v) := - z(\zeta', x)\cdot v_2+i\frac{(z(\zeta',x)-v_1)^2}{2}$ and $\tilde a(\zeta', x)$ is analytic and nonvanishing.
\end{Proposition}

Before proceeding further we introduce a notation to streamline our presentation. If $\Omega\subset \mR^k$ is an open subset, we denote by $\Omega_{\mC}$ an open subset of $\mC^k$ with $\Omega_{\mC} \cap \mR^k = \Omega$. Here we use the natural embedding $\mR^k\subset \mC^k$. Similarly we have the extension of the set $Z$ defined in \eqref{algebraic def of Z}, given by 
\begin{eqnarray}
\label{algebraic def of Z complex}
Z_{\mC} := \{ (z,x) \in  V_{\mC} \times U_{\mC} \mid \phi (z,x') -x'' = 0\}
\end{eqnarray}
where $\phi$ is now the unique holomorphic extension of the original $\phi$, defined on a complex neighborhood of $V\times U'$.

We elaborate a bit more on some elementary properties of the holomorphic extension $\phi$ as they will be useful for us later. Writing $z\in \mC^N$ as $z = z^{\mR} + iz^{i\mR}$ with $z^\mR$ and $z^{i\mR}$ real vectors, we write $\phi$ as
\begin{eqnarray}
\label{phi real and im}
\phi(z,x') = \phi^{\mR} (z^\mR, z^{i\mR},x') + i\phi^{i\mR}(z^\mR, z^{i\mR},x').
\end{eqnarray}
Since $\phi$ takes on real values for real $z$, we have that 
\begin{eqnarray}
\label{phi im vanishes}
\phi^{i\mR}(z^{\mR},0,x') = 0, \qquad z^{\mR}\in \mR^{N},\  x'\in \mR^{n'}.
\end{eqnarray}
Note that the same holds for (real) derivatives in $z^\mR$ and $x$.

The phase appearing in \eqref{Klambda integral 2} will be the key component in our analysis of this kernel. In particular we are interested in its critical points as a map $\zeta' \mapsto \Psi(\zeta', x,v)$ and in the dependence of the critical point on $(x,v)\in U_\mC\times \mathcal V_\mC$ where $\mathcal V$ is given below. In order to state the required properties of the phase function, we introduce a map $\chi$ as follows.

\begin{Lemma} \label{lemma_ctilde}
Assume \eqref{phiz has full rank}. There is a neighborhood $\tilde C$ of $(\hat v_1, \hat v_2, \hat u_1, \hat u_2)$ in $C$, an analytic $(N+n)$-dimensional submanifold $\mathcal{V} = \pi_L(\tilde C)$ of $T^* \mR^N$ containing $(\hat v_1, \hat v_2)$, and an analytic surjective submersion  
\begin{eqnarray}
\label{def of chi}
\chi := \pi_R\circ\pi_L^{-1} : {\mathcal V} \to  \pi_{R}( \tilde{C})
\end{eqnarray}
where $\pi_{R}( \tilde{C})$ is an open set in $T^*U$.
\end{Lemma}
\begin{proof}
By \eqref{phiz has full rank} and Lemma \ref{lemma_bolker_local_char}, we know that $\pi_L$ is an immersion and $\pi_R$ is a submersion near $(\hat v_1, \hat v_2, \hat u_1, \hat u_2)$. Hence there is a neighborhood $\tilde C$ of $(\hat v_1, \hat v_2, \hat u_1, \hat u_2)$ in $C$ such that $\pi_L|_{\tilde C}: \tilde C \to T^* \mR^N$ is an embedding and $\mathcal{V} = \pi_L(\tilde C)$ is an analytic submanifold. Then $\chi$ is a well defined analytic map and it is a submersion if $\tilde C$ is chosen small enough. Hence also $\pi_R(\tilde C)$ is open.
\end{proof}

The graph of $\chi$ can be characterized in the following way (see \eqref{local canonical relation}, here we write $u_j= (u'_j, u''_j)$):
\begin{align}
\label{canonical relation characterization}
(u_1, u_2) = \chi(v_1,v_2) \quad &\Longleftrightarrow \quad (v_1, v_2, u_1, u_2) \in \tilde{C} \\
 &\Longleftrightarrow \quad v_2 = \phi_z(v_1,u_1')^T u_2'', \ u_1'' = \phi(v_1, u_1'), \ u_2' = - \phi_{x'}(v_1, u_1')^T u_2'' \notag
\end{align}

The proof of the following proposition will be given in Section \ref{sec_good_phase} and this will be the most technical part of the argument.

\begin{Proposition}
\label{good phase} 
There exist open subsets $\hat U \subset \subset U$ containing $\hat x$ and $\hat {\mathcal V}\subset\subset \mathcal V$ containing $\hat v =(\hat v_1, \hat v_2)$ such that for all $(x,v) \in \hat U_{\mC} \times \hat{ \mathcal V}_{\mC}$, the map $\zeta' \mapsto \Psi(\zeta', x,v)$ defined on $\mathcal U$ has a unique nondegenerate critical point $\zeta'_c(x,v)$ which depends holomorphically on $(x,v)$. The phase function $\psi(x,v):= \Psi(\zeta'_c(x,v), x,v)$ satisfies the following conditions:
\begin{enumerate}
\item[(1)]
If we write $v = (v_1, v_2)$, then $\psi(\pi(\chi(v)), v) +v_1\cdot v_2 = 0$ where $\chi$ is the map given by \eqref{def of chi}.
\item[(2)]
$\left(x, d_x\psi(x,v)\right)|_{x = \pi(\chi(v))} = \chi(v)$.
\item[(3)]
$\im(\psi(x,v)) \geq c|x-\pi(\chi(v))|^2$ if $x$ and $v$ are real valued.
\end{enumerate}
\end{Proposition}

\subsection{Stationary phase and proof of Theorem \ref{elliptic regularity}}

Let $\hat U\subset \mR^{n'+n''}$ be the open subset containing $\hat x$ appearing in the statement of Proposition \ref{good phase}. For all $f\in {\mathcal E}'(\hat U)$ we need to show that if $(\hat v_1, \hat v_2)\notin \WF_a(Tf)$ then $(\hat x, \hat u_2)\notin \WF_a(f)$.

By the assumption that $(\hat v_1,\hat v_2)$ is not in the analytic wave front set, by \eqref{wave packet relation} and by \cite[Definition 6.1]{sjostrand}, there is $\epsilon>0$ such that for any $v = (v_1, v_2)$ in some neighborhood of $(\hat v_1, \hat v_2)$ one has 
\begin{eqnarray}
\int f(x) K_{\lambda}(x,v_1,v_2) \,dx = O(e^{-\epsilon \lambda})
\end{eqnarray}
as $\lambda\to\infty$. %
The kernel $K_{\lambda}(x,v_1, v_2)$ is given by \eqref{Klambda integral 2}.
Proposition \ref{good phase} asserts that for complex $(x,v_1, v_2)$ near $(\hat x, \hat v_1, \hat v_2)$, the phase $\zeta' \mapsto \Psi(\zeta', x,v)$ defined on $\mathcal U$ has a unique nondegenerate critical point $\zeta_c(x,v_1,v_2)$. So by stationary phase (Theorem 2.8 and Remark 2.10 of \cite{sjostrand}) applied to the expression \eqref{Klambda integral 2} for $K_{\lambda}(x,v_1, v_2)$ we get that
\[
\int f(x) e^{i \lambda \psi(x,v_1, v_2)} A(x,v_1, v_2;\lambda) \,dx = O(e^{-\epsilon \lambda})
\]
for $(v_1, v_2)$ complex near $(\hat v_1,\hat v_2)$. By the stationary phase expansion formula (Theorem 2.8 and Remark 2.10 of \cite{sjostrand}), the new symbol $A(x,v_1,v_2;\lambda)$ is a nonvanishing classical analytic symbol in the sense of \cite[Section 1]{sjostrand}. We get then that 
\[
\int f(x) e^{i \lambda (\psi(x,v_1,v_2) + v_1\cdot v_2)} A(x,v_1,v_2;\lambda) \,dx  = O(e^{-\epsilon \lambda}).
\]
We now make a change of variable in the $(v_1,v_2)$ variable. By Lemma \ref{lemma_ctilde}, $\chi :{\mathcal V} \to \pi_R(\tilde C)\subset T^*U$ is an analytic surjective submersion onto the open set $\pi_R(\tilde C)$. By the local submersion theorem (which follows from the implicit function theorem and hence works in the analytic category), there exists an analytic map $\chi^+: \pi_R(\tilde C) \to \mathcal{V}$ satisfying 
\[
\chi \circ \chi^+ = \mathrm{Id}.
\]
 Therefore if $\hat U$ is chosen small enough, we can define $\tilde \psi: \hat U \times \pi_R(\tilde C)\to \mC$ by 
\[
\tilde \psi(x,u_1,u_2)  = (\psi(x,v_1,v_2) + v_1\cdot v_2)|_{(v_1, v_2) = \chi+(u_1,u_2)}.
\]
Note that the variable $x$ is untouched in this change of parameters. The above integral becomes
\begin{eqnarray}
\label{final integral}
\int f(x) e^{i \lambda \tilde \psi(x,u_1, u_2) } \tilde A(x,u_1, u_2;\lambda) \,dx = O(e^{-\epsilon \lambda})
\end{eqnarray}
for all $(u_1, u_2)$ near $(\hat x, \hat u_2)$.
By Proposition \ref{good phase} part (1) the phase function $\tilde\psi(x,u_1, u_2)$ vanishes when $x = u_1$, and by part (2) it satisfies 
\[
d_x (\tilde\psi(x, u_1, u_2) )|_{x = u_1} = u_2.
\]
Furthermore, by property (3) of Proposition \ref{good phase}
\[
\im(\tilde\psi(x,u_1, u_2)) \geq c|x- u_1|^2
\]
for real $(u_1, u_2)$. By \cite[Definition 6.1]{sjostrand} we have that $(\hat x, \hat u_2)\notin\WF_a(f)$.

\subsection{$K_{\lambda}(x,v)$ as an oscillatory integral and proof of Proposition \ref{prop: Klambda integral}} \label{sec_klambda_integral}
We observe that the integral kernel $K_{\lambda}(x,v)$ given by \eqref{def: K}
involves oscillatory wave packets which concentrate simultaneously in space and frequency. Writing out the integrals explicitly using \eqref{def T} for the operator $T$ and integrating out the $u_2$ and $y$ variables we have that for $(x,v) \in U \times \mathcal V$,
\begin{eqnarray}
\label{Klambda integral}
K_{\lambda}(x,v) = c_{n,N} \lambda^{\frac{3N}{4} + n''} \int\limits_{z\in V} \int\limits_{\eta\in\mR^{n''}}  e^{i\lambda(\phi(z,x') - x'')\cdot \eta}  e^{-i\lambda z\cdot v_2} e^{-\frac{\lambda|z-v_1|^2}{2}}a(z,x) \,d\eta \,dz.
\end{eqnarray}
The $\eta$ integral in \eqref{Klambda integral} results in $\lambda^{-n''} \delta_0(\phi(z,x') - x'') $ which is a distribution supported on the submanifold $Z\subset V\times U$.

The following lemma gives another coordinate representation for $Z$.

\begin{Lemma} 
\label{Z coordinates}
The set $Z$ has the following properties:
\begin{enumerate}
\item[(1)]
The set $Z$ is an analytic submanifold of $V\times U$ of dimension $N + n'$.
\item[(2)]
We can find $\Omega \subset  V \times  U$ containing $(\hat v_1, \hat x)$, an open set $\tilde \Omega \subset \mR^N\times \mR^{n'+n''}$, and an analytic diffeomorphism of the form $(\zeta, x)\mapsto (z(\zeta, x) ,x)$ from $\tilde \Omega$ to $\Omega$ such that 
$$Z\cap \Omega = \{(z(\zeta, x), x)\mid \zeta = (\zeta'', \zeta') = (0, \zeta'),\zeta'\in\mathcal U, x\in U\}$$
for some neighborhood $\mathcal U$ of the origin in $\mR^{N-n''}$. Also, the matrix $z_{\zeta}(\zeta,x)$ is invertible. 
\item[(3)]
The coordinate system constructed above extends holomorphically to a coordinate system (possibly after shrinking both $\Omega$ and $\tilde\Omega$) from $\tilde\Omega_{\mC}$ to $\Omega_{\mC}$ such that $Z_{\mC}\cap\Omega_{\mC}$ is again given by $\zeta'' = 0$.
\end{enumerate}
\end{Lemma}
\begin{proof} 

The statement (1) is a consequence of (2) so we proceed with (2). By \eqref{phiz has full rank}, we must have that the $ n''\times N$ matrix $\partial_z\phi$ has full rank. Choose $\Omega\subset V\times U$ small containing $(\hat v_1, \hat x)$ and without loss of generality we may assume that $\Omega = V\times U$. Since $N>n''$, we may assume that at $(\hat v_1, \hat x) \in Z\cap \Omega$, the first $n''$ vectors
$$\{\partial_{z_1}\phi, \dots, \partial_{z_{n''}}\phi\}$$
are linearly independent. For any $\zeta'' = (\zeta_1,\dots, \zeta_{n''})$ close to the origin and $x\in U$, use implicit function theorem to find the unique solution $z'' := (z_1, \dots, z_{n''})$  to the equation 
$$\phi(z,x')  - x''= \zeta''$$ 
so that 
$$(z_1, \dots, z_{n''}, \underbrace{ z_{n''+1}, \dots, z_{N}}_{z'}) \in V.$$
The dependence of $(z_1,\dots, z_{n''})$ on the variables
$$(\zeta'', z', x)\in \mR^{n''} \times \mR^{N - n''} \times \mR^{n' + n''}$$
is analytic.

Setting $\zeta' = z'$ and $\zeta = (\zeta'', \zeta')$ gives us an analytic map $z(\zeta,x)$ solving $\phi(z(\zeta,x),x') - x''= \zeta''$. By direct computation we see that the map $(\zeta, x)\mapsto (z(\zeta,x),x)$ has nondegenerate Jacobian and is therefore locally an analytic diffeomorphism. Consequently, the matrix $z_{\zeta}(\zeta,x)$ is invertible.

To see (3), extend $(\zeta,x)\mapsto (z(\zeta,x),x)$ holomorphically to a complex neighborhood $\tilde \Omega_{\mC}$ containing $\tilde \Omega$. By the fact that the real Jacobian of this holomorphic map is nonvanishing, the complex Jacobian is also nonvanishing. After possibly shrinking $\tilde\Omega_{\mC}$ and $\Omega_{\mC}$ we can conclude that $\zeta(\cdot,\cdot)$ is therefore a diffeomorphism between $\tilde\Omega_{\mC}$ and $\Omega_{\mC}$.

By construction of $z(\zeta'',\zeta', x)$, $\phi(z(\zeta'', \zeta'),x') -x'' = \zeta''$ for $(\zeta, x) \in \mathcal U \times U$. This identity therefore holds by unique continuation for the holomorphic extension to the complex domain. So $Z_{\mC}\cap \Omega_{\mC}$ is again given by $\{\zeta'' = 0\}$.\end{proof}
In what follows we assume without loss of generality that $Z = \Omega\cap Z$. 

We make a remark on the notation which will follow. We are often interested in coordinate systems for $Z$ or $Z_{\mC}$ which are given by $(z(0,\zeta', x), x)$. Therefore, to simplify notation we omit the $0$ in the first argument and write $z(\zeta',x)$ in place of $z(0,\zeta',x)$ in this case.

\begin{proof}[Proof of Proposition \ref{prop: Klambda integral}]
Integrating in $\eta$, \eqref{Klambda integral} becomes
$$K_{\lambda}(x,v) = c_{n,N} \lambda^{\frac{3N}{4}}  \int\limits_{Z^x}    e^{-i\lambda z\cdot v_2} e^{-\frac{\lambda|z-v_1|^2}{2}}a(z,x) \,dz$$
where for each $x\in U$,
$$Z^x := \{z\in V\mid (z,x)\in Z\} =  \{z(\zeta', x)\mid \zeta'\in\mathcal U\}$$
by Lemma \ref{Z coordinates}. After a change of coordinate into the $\zeta'$ variable this integral becomes
$$K_{\lambda}(x,v) = c_{n,N} \lambda^{\frac{3N}{4}} \int\limits_{\zeta'\in \mathcal U} e^{i\lambda\Psi(\zeta', x,v)}\tilde a (\zeta', x) \,d\zeta'$$
for some nonvanishing real analytic function $\tilde a(\zeta',x)$ defined on $\mathcal U \times U$. 
\end{proof}

\subsection{Critical points of $\Psi$ and proof of Proposition \ref{good phase}} \label{sec_good_phase}
We now examine the critical points of the phase function 
$$\Psi(\zeta'; x,v) := - z(\zeta', x)\cdot v_2+ i\frac{(z(\zeta',x)-v_1)^2}{2}$$
appearing in \eqref{Klambda integral 2}. This function has a holomorphic extension when the variables are allowed to be complex which we still denote by $\Psi$. For what follows we treat $\zeta'$ as a variable for the function $\Psi$ and $x$ and $v= (v_1, v_2)$ as parameters. 

Let $\hat x \in U$ and $\hat v = (\hat v_1, \hat v_2) \in \mathcal V$ satisfy $\hat x = \pi(\chi(\hat v))$ which automatically guarantees that $(\hat v_1, \hat x)\in Z$ by \eqref{canonical relation characterization}. Choose $\Omega$ as in statement (2) of Lemma \ref{Z coordinates}. Then Lemma \ref{Z coordinates} gives a $\hat \zeta'\in \mathcal U$ such that $z(\hat \zeta', \hat x) = \hat v_1$. Statement (3) of Lemma \ref{Z coordinates} provides a complex neighborhood $\Omega_{\mC}$  of $\mC^{N} \times \mC^{n'+n''}$ containing $(\hat v_1, \hat x)$ which is diffeomorphic to a complex neighborhood $\tilde \Omega_{\mC}$ containing $((0,\hat \zeta'), \hat x)$. The diffeomorphism is given by 
\[
(\zeta, x)\in \tilde\Omega_{\mC} \mapsto (z(\zeta,x), x)\in \Omega_{\mC}
\]
near $((0,\hat{\zeta}'), \hat x)$. The submanifold $Z_{\mC}\cap \Omega_{\mC}$ is given by $\zeta = (0, \zeta')$ where $\zeta'\in  \mathcal U_{\mC}$ and $\mathcal U$ is given as in Lemma \ref{Z coordinates}.

We first look at the possibility of critical points when $x = \pi(\chi(v))$:
\begin{Lemma}
\label{real crit points}
For each real $(x,v = (v_1,v_2)) \in U\times \mathcal V$, writing $x(v) = \pi(\chi(v))$, the following holds:
\begin{enumerate}
\item[(1)]
$\Psi(\cdot; x(v), v)$ has a real critical point $ \zeta_c'(x(v),v)\in \mathcal U$ satisfying $z(\zeta_c'(x(v),v), x(v)) = v_1$.
\item[(2)]
The matrix of complex derivatives $\left( \frac{\partial^2 \Psi}{\partial \zeta'^2}\right)$ at $\zeta_c'(x(v),v)$ is nondegenerate.
\item[(3)]
If $\ \mathcal U$ is chosen small enough, the critical point $\zeta_c'(x(v),v)$ is the unique critical point of $\zeta'\mapsto \Psi(\cdot ; x(v), v)$ in ${\mathcal U}$.
\end{enumerate}
\end{Lemma}
\begin{proof}
Let $(x, u_2) = \chi(v_1, v_2)$. To see (1) we use \eqref{canonical relation characterization} to get
\begin{eqnarray}
\label{v1v2x relation}
v_2 = \phi_z(v_1,x')^T u_2'',\ \ \phi(v_1, x') = x''.
\end{eqnarray}
We see from the second equality that $(v_1, x) \in Z$ and by Lemma \ref{Z coordinates} we can find $\zeta_c'\in \mathcal U$ real such that $z(\zeta_c', x) = v_1$. We claim that $\zeta_c'$ is a critical point of $\Psi(\cdot; x,v)$. 

Indeed using the first equality in \eqref{v1v2x relation} for $v_2$ we get:
$$- \partial_{\zeta'} (z(\zeta',x) \cdot v_2) \mid_{\zeta' = \zeta_c'} = -\partial_{\zeta'} z^T v_2=  -(\partial_{\zeta'} z)^T (\phi_z(v_1, x'))^T u_2'' = (\partial_{\zeta'} (\phi(z(\zeta',x),x')) )^Tu_2''\mid_{\zeta'= \zeta_c'} = 0.$$
The second last equality comes by the chain rule. The last equality comes from the fact that $\phi(z(\zeta',x),x') = x''$ is independent of $\zeta'$. The fact that $\partial_{\zeta'} (z(\zeta',x) - v_1)^2\mid_{\zeta' = \zeta_c'} = 0$ is trivial using $z(\zeta_c', x) = v_1$.

To see (2), we observe that since $\Psi$ extends holomorphically in the $\zeta$ variable to a complex neighborhood of $\zeta_c'$, the complex Hessian coincides with the real Hessian by the Cauchy-Riemann equations. Hence it is enough to consider the real Hessian (with respect to $\zeta'$) of $\Psi(\zeta'; x,v) = - z(\zeta', x)\cdot v_2+ i\frac{(z(\zeta',x)-v_1)^2}{2}$ at $\zeta_c'$. Direct computation using $z(\zeta_c', x) = v_1$ yields that this Hessian is of the form $B + iA^TA$ where $A = z_{\zeta'}(\zeta_c',x)$ is an injective real matrix (since $z_{\zeta}$ is invertible by Lemma \ref{Z coordinates} (2)) while $B = -\p_{\zeta'}^2(z \cdot v_2)|_{\zeta'=\zeta_c'}$ is a symmetric real matrix.  

It is easy to see that such a matrix acting on complex vectors has trivial kernel. Indeed, let $(B+ i A^T A) (a + ib) = 0$ for real vectors $a,b\in \mR^{N-n''}$. We then have that $B a = A^T A b$ and $Bb = -A^TAa$. Taking the inner product of the first equation with $b$ and the second equation with $a$ we get 
\[
b\cdot Ba = |Ab|^2,\ -a\cdot Bb =  |Aa|^2.
\]
Since $B$ is symmetric the two equations combine to yield $-|Aa|^2 = |Ab|^2$, so $Ab =A a  =0$. Now injectivity of $A$ gives that $a = b = 0$. So (2) is verified.

The statement (3) follows directly from (2) and the holomorphic implicit function theorem.
\end{proof}

The holomorphic implicit function theorem together with the non-degeneracy of the Hessian of $\zeta'\mapsto \Psi(\zeta' ; x(v), v)$ at $\zeta'= \zeta_c'(x(v),v)$ provides us complex critical points even when $x$ is not necessarily equal to $\pi(\chi(v))$:
\begin{Corollary}
\label{complex critical points}
Given any $U\times\mathcal V$ small enough containing $(\hat x , \hat v)$ with $\hat x = \pi(\chi(\hat v))$ and $\mathcal U$ small enough containing $\zeta'_c(\hat x, \hat v)$, for all $(x,v) \in  U_\mC \times {\mathcal V}_{\mC}$ the map $\zeta' \mapsto \Psi(\zeta'; x,v)$ has a unique complex critical point $\zeta'_c(x,v)\in \mathcal U_\mC$ which depends holomorphically on $(x,v)$. Moreover, if $(x,v) \in U \times \mathcal{V}$ and $x = \pi(\chi(v))$, then $\zeta_c'(x,v)$ is real.
\end{Corollary}

%
%
%
%
%
%
%
%
%
%
%
At this point it is convenient to introduce the following auxiliary phase function:
\begin{eqnarray}
\label{def of Phi}
\Phi(z, \eta; x,v) := (\phi(z,x') - x'')\cdot \eta - z\cdot v_2  + \frac{i(z-v_1)^2}{2}
\end{eqnarray}
which has a holomorphic extension that we still denote by $\Phi$. Clearly, $\Phi(z(\zeta',x), \eta; x,v) = \Psi(\zeta';x,v)$.

We treat $(z, \eta)\in V_{\mC} \times \mC^{n''}$ as variables and $(x,v)\in  U \times \mathcal V$ as parameters where $(x,v)\in U \times \mathcal V$ (in most cases we will consider real $(x,v)$ though $(z,\eta)$ can still be complex). 
By the characterization of $\chi$ in \eqref{canonical relation characterization}, when $\chi(v) = (x, u_2)$, we see using direct calculation that the point $(v_1, u_2'') \in \mR^N \times \mR^{n''}$ is a real critical point of $(z,\eta)\mapsto \Phi(z, \eta ; x, v)$. 
For arbitrary $(x,v)$, we will try to find complex critical points $(z_c,\eta_c)$ for $\Phi(\cdot; x,v)$. If they exist, critical points $(z_c, \eta_c)$ are given by the equation:
\begin{eqnarray}
\label{critical point eq for Phi}
\phi(z_c,x') - x'' = 0, \qquad \phi_z(z_c, x')^T \eta_c - v_2 =- i(z_c-v_1).
\end{eqnarray}

The critical points of $\Psi$ are related to that of $\Phi$ in the following way:
\begin{Lemma}
\label{crit of Phi is crit of Psi}
Suppose $U\times \mathcal V$ is sufficiently small containing $(\hat x, \hat v)$ and let $(x,v)\in U_\mC\times \mathcal V_\mC$. Suppose $(z_c, \eta_c)$ is a complex critical point of $\Phi(\cdot; x,v)$ so that $(z_c,x)\in V_\mC \times U_\mC$, then we can find $\zeta'_0 \in \mathcal U_{\mC}$ such that $z(\zeta'_0, x) = z_c$ and $\zeta'_0$ is a critical point of $\Psi(\cdot; x,v)$.
\end{Lemma}
\begin{proof}
The first equation of \eqref{critical point eq for Phi} indicates that $(z_c, x)\in Z_{\mC}$ (see \eqref{algebraic def of Z complex} for definition of $Z_{\mC}$). So we may find $\zeta_0'\in\mathcal U_\mC$ such that $z_c = z(\zeta_0', x)$. We now need to show that $\zeta_0'$ is a critical point of $\Psi(\cdot; x,v)$. 

Observe that $\phi(z(\zeta',x) , x') - x'' = 0$ for all $\zeta'\in\mathcal{U}_{\mC}$ so differentiating in $\zeta'$ we get
\[
(\partial_{\zeta'} z(\zeta_0',x))^T \phi_z(z(\zeta_0', x), x')^T = 0 .
\]
Multiply the second equation of \eqref{critical point eq for Phi} by $(\partial_{\zeta'} z(\zeta_0',x))^T$ and use the above identity we have that 
\[
\partial_{\zeta'} z(\zeta_0',x)^Tv_2 = i \partial_{\zeta'} z(\zeta_0',x)^T(z(\zeta_0',x)- v_1).
\]
This is precisely what it means for $\zeta_0'$ to solve $\Psi_{\zeta'}(\zeta'_0; x,v) = 0$.
\end{proof}
\begin{Lemma}
\label{main crit pt of Phi}
Let $(x,v)$ be real and close to $(\hat{x}, \hat{v})$ with $x = \pi(\chi(v))$, and write $\chi(v) = (x, u_2)$. The point $(z, \eta) := (v_1, u_2'')$ is a critical point of $\Phi$ in $V \times\mR^{n''}$. Furthermore, at this point the complex Hessian of $\Phi$ is nondegenerate with positive semidefinite imaginary part.
\end{Lemma}
\begin{proof}
By the characterization of $\chi$ in \eqref{canonical relation characterization}, when $\chi(v) = (x,  u_2)$, the point $(z, \eta) = (v_1, u_2'')$ solves \eqref{critical point eq for Phi} and hence it is a critical point of $\Phi$.

We use again that the complex Hessian is equal to real Hessian by holomorphicity, so it is enough to consider the real Hessian of $\Phi$ at $(v_1, u_2'')$. By \eqref{phiz has full rank} $\phi_z(z, x')$ has full rank when $(x,v)$ is close to $(\hat{x}, \hat{v})$. Direct calculation gives the following block form for the Hessian of $\Phi$:
\begin{eqnarray}\nonumber
\begin{pmatrix}
\partial_z^2 \left(\phi(z, x')\cdot u_2''\right)+iI & \phi_z(z, x')^T\\
\phi_z(z, x')& 0
\end{pmatrix} 
= 
\underbrace{\begin{pmatrix}
\partial_z^2 \left(\phi(z, x')\cdot u_2''\right) & \phi_z(z, x')^T\\
\phi_z(z, x')& 0
\end{pmatrix}}_{A}  + i
\underbrace{\begin{pmatrix}
I_{N\times N} &0\\
0& 0
\end{pmatrix} }_{B}\\
\end{eqnarray}
Note that both $A$ and $B$ are real valued symmetric matrices. To show that $A+iB$ is nondegenerate, we suppose $(A+iB) (a+ib) = 0$ for some $a,b\in \mR^{N + n''}$. Then $Aa = Bb$ and $Ab = -Ba$, which gives that $Bb \cdot b = -Ba \cdot a$ since $A$ is symmetric. Using the fact that $B = B^T B$ yields $Ba = Bb = 0$. This means that $a = (0, a'')$ and $b = (0,b'')$ where $a'', b''\in \mR^{n''}$. We then have that 
\[
\phi_z(z, x')^T (a''+ib'') = 0.
\]
Using the fact that $\phi_z(z, x')$ is real and has full rank shows $a''= b''=0$. So we have shown that the complex Hessian of $\Phi$ at $(z, \eta)$ is nondegenerate.
\end{proof}

The next result gives a correspondence between the complex critical points of $\Phi$ and $\Psi$.

\begin{Lemma}
\label{crit of Phi in small nbhd}
For all $(x,v) \in U_{\mC} \times \mathcal V_{\mC}$, let $\zeta'_c(x,v)$ be the critical point of $\Psi$ deduced in Corollary \ref{complex critical points}. Then for all $(x,v)$ near $(\hat x, \hat v)$, there exists a unique critical point $({\mathcal Z}_c(x,v), \eta_c(x,v)) \in V_\mC \times \mC^{n''}$ of $\Phi$ and it satisfies
\begin{eqnarray}
\label{mathcal Z is zeta}
{\mathcal Z}_c(x,v) = z(\zeta'_c(x,v), x).
\end{eqnarray}
In particular, if $x$ and $v$ are real and $x = \pi(\chi(v))$, then writing $(x, u_2) = \chi(v)$, we have that $({\mathcal Z}_c(x, v), \eta_c( x, v) )= (v_1, u_2'')$ and thus the critical point is real.
\end{Lemma}
\begin{proof}
By Lemma \ref{main crit pt of Phi}, when $\chi(\hat v) = (\hat x, \hat u_2)$, the point $(\hat v_1, \hat u_2'')$ is a real critical point of $\Phi$. Furthermore, at this point the complex Hessian of $\Phi$ is nondegenerate. The holomorphic implicit function theorem allows us to find critical points $({\mathcal Z}_c(x,v), \eta_c(x,v)) \in V_\mC \times \mC^{n''}$ of $\Phi(\cdot; x,v)$ as $(x,v)$ vary in $U_\mC\times V_\mC$ with ${\mathcal Z}_c(\hat x, \hat v) = \hat v_1$.  The dependence of ${\mathcal Z}_c(x,v)$ on $(x,v)$ is holomorphic.

For each $({\mathcal Z}_c(x,v), \eta_c(x,v))$ use Lemma \ref{crit of Phi is crit of Psi} to find $\zeta'_0(x,v)$ so that $\zeta'_0(x,v)$ is a critical point of $\Psi(\cdot; x,v)$ and $z(\zeta'_0(x,v), x) = {\mathcal Z}_c(x,v)$. By the uniqueness statement of Corollary \ref{complex critical points}, $\zeta'_0(x,v) = \zeta'_c(x,v)$. This means that ${\mathcal Z}_c(x,v) = z(\zeta'_c(x,v),x)$.

To see that $({\mathcal Z}_c(x,v), \eta_c(x,v))$ is the unique critical point in $V_\mC\times \mC^{n''}$, suppose $(\tilde {\mathcal Z}_c(x,v), \tilde \eta_c(x,v)) \in V_\mC \times \mC^{n''}$ is another critical point of $\Phi(\cdot; x,v)$. Then by Lemma \ref{crit of Phi is crit of Psi}, there is $\tilde \zeta'_0(x,v)\in \mathcal U_{\mC}$ which is a critical point of $\Psi(\cdot; x,v)$ such that $\tilde {\mathcal Z}_c(x,v) = z(\tilde \zeta'_0(x,v), x)$. Uniqueness of critical points stated in Corollary \ref{complex critical points} of $\Psi(\cdot;x,v)$ then forces $\tilde \zeta'_0(x,v) = \zeta'_c(x,v)$. This means $\tilde {\mathcal Z}_c(x,v)  = {\mathcal Z}_c(x,v) $. The second equation of \eqref{critical point eq for Phi} and injectivity of the matrix $\phi_z( {\mathcal Z}_c(x,v), x')^T$ then forces $\tilde \eta_c(x,v) = \eta_c(x,v)$.

Finally, when $(x,v)$ is real and $x = \pi(\chi(v))$, the last statement in the lemma follows from Lemma \ref{main crit pt of Phi} and the uniqueness part above.
\end{proof}
The identity in Lemma \ref{crit of Phi in small nbhd} gives us a convenient way to prove the following quantitative estimate:
\begin{Lemma}
\label{coercive estimate}
 There is a constant $c_0>0$ such that for all $(x,v) \in U \times \mathcal V$, we have the estimate 
\begin{eqnarray}
\label{zc estimate}
|z(\zeta'_c(x,v),x) - v_1| \geq c_0 |x - \pi(\chi(v))|.
\end{eqnarray}
\end{Lemma}

\begin{proof}
By \eqref{mathcal Z is zeta}, it is equivalent to show that 
\begin{eqnarray}
\label{zc estimate'}
|{\mathcal Z}_c(x,v) - v_1| \geq c|x - \pi(\chi(v))|.
\end{eqnarray}
To this end, for all $v\in \mathcal V$, let $(u_1,u_2) = \chi(v)$. By the characterization of $\chi$ in \eqref{canonical relation characterization} we have that
\begin{eqnarray}
\label{uv eq}
\phi(v_1, u_1') = u_1'',\quad \phi_z(v_1, u_1')^T u_2'' -v_2= 0,\quad u_2' = -\phi_{x'}(v_1,u_1')^Tu_2''.
\end{eqnarray}
Meanwhile, by \eqref{critical point eq for Phi} ${\mathcal Z}_c(x,v)$ satisfies
\begin{eqnarray}
\label{zc eq}
\phi({\mathcal Z}_c(x,v), x') = x'',\quad \phi_z({\mathcal Z}_c(x,v), x')^T\eta_c(x,v) - v_2 = i{({\mathcal Z}_c(x,v) - v_1)} .
\end{eqnarray}
Subtracting the first equation of \eqref{uv eq} and the first equation of \eqref{zc eq}, Lipschitz continuity of $\phi$ gives that 
\[
|x'' - u_1''| \leq c(|v_1 - {\mathcal Z}_c(x,v)| + |x' - u_1'|).
\]
So to show \eqref{zc estimate'}, it in fact suffices to show
\begin{eqnarray}
\label{x' estimate}
 |x' - u_1'|\leq c|v_1 - {\mathcal Z}_c(x,v)|.
 \end{eqnarray}

To this end, we look at the second equation of \eqref{zc eq} and subtract from it the second equation of \eqref{uv eq}: 
\[
\phi_z({\mathcal Z}_c(x,v),x')^T \eta_c(x,v) - \phi_z(v_1, u_1')^T u_2'' = O({\mathcal Z}_c(x,v) - v_1).
\]
Using standard continuity estimates allows us to replace the $\phi_z({\mathcal Z}_c(x,v),x')^T\eta_c(x,v)$ in the first term by $\phi_z(v_1, x')^T\eta_c(x,v)$ to get, for $(x,v)$ in $U_\mC \times \mathcal V_\mC$ (chosen small enough),
\[
\phi_z(v_1,x')^T \eta_c(x,v) - \phi_z(v_1, u_1')^T u_2'' = O({\mathcal Z}_c(x,v) - v_1).
\]
Using the above expression, the estimate \eqref{x' estimate} would be established if we could show that for all $(x',\eta)$ near $(u_1', u_2'')$ we have that 
\begin{eqnarray}
\label{inverse function estimate}
|x'- u_1'| + |\eta - u_2''| \leq C|\phi_z(v_1,x')^T \eta - \phi_z(v_1, u_1')^T u_2''|.
\end{eqnarray}

To establish \eqref{inverse function estimate} we observe that by \eqref{phiz has full rank} the Jacobian matrix 
\[
A: =\begin{pmatrix} \partial_{x'}\left( \phi_z(v_1, x')^T\eta\right)\mid_{x' = u_1'},& \phi_z(v_1, u_1')^T\end{pmatrix}
\]
is an $N\times (n'+ n'')$ matrix of full rank when $(v_1, u_1')\in V\times U'$ and $\eta$ is near $\hat \eta$. As $N\geq n'+ n''$, the matrix $A^T A$ is invertible.

We now expand the function $(x', \eta)\mapsto \phi_z(v_1,x')^T\eta$ around the point $(x', \eta) = (u_1', u_2'')$ to get 
\[
\phi_z(v_1,x')^T\eta - \phi_z(v_1,u_1')^Tu_2''= A\begin{pmatrix}
           x' - u_1'\\
           \eta - u_2''
         \end{pmatrix} + Q\left(\begin{pmatrix}
           x' - u_1'\\
           \eta - u_2''
         \end{pmatrix}\right)
\]
for some quadratic form $Q$. The estimate \eqref{inverse function estimate} comes now directly from the invertibility of $A^TA$. 
\end{proof}

Fix $(\hat x, \hat v)$ and $U\times \mathcal V$ as in Corollary \ref{complex critical points}. 
We now define the holomorphic phase on $U_\mC\times\mathcal V_\mC$ by
\begin{eqnarray}
\label{psi}
\psi (x,v) &:=& \Psi(\zeta'_c(x,v); x,v) \\\nonumber &=& -z(\zeta'_c(x,v),x) \cdot v_2 + i\frac{( z(\zeta'_c(x,v),x) - v_1)^2}{2}.
\end{eqnarray}
We showed in Lemma \ref{real crit points} that when $x = \pi(\chi(v))$, the critical point $\zeta'_c(x,v)$ is real valued and consequently $z(\zeta'_c(x,v),x) = v_1$ which implies $\psi(x,v)$ is real valued. It turns out that when $x - \pi(\chi(v))$ is small, the imaginary part of $\psi(x,v)$ is bounded below by the imaginary part of $z(\zeta'_c(x,v),x)$:
\begin{Lemma}
\label{bound the imaginary part}
Writing $z(\zeta_c'(x,v),x) = z^{\mR}(\zeta_c'(x,v),x) + i z^{i\mR}(\zeta_c'(x,v),x)$ where $z^{\mR}$ and $z^{i\mR}$ are real, 
we have for real $x\in U$ and $v\in {\mathcal V}$ the estimate 
\begin{eqnarray}
\label{zI bounded by psi}
\im (\psi(x,v)) \geq C| z^{i\mR}(\zeta'_c(x,v), x)|^2
\end{eqnarray}
when $|x -\pi(\chi(v))|$ is sufficiently small.
\end{Lemma}
\begin{proof}
Setting 
$${\mathcal Z}_c(x,v) ={\mathcal Z}_c^{\mR}(x,v)+ i {\mathcal Z}_c^{i\mR}(x,v),$$
for real valued ${\mathcal Z}_c^\mR(x,v)$ and ${\mathcal Z}_c^{i\mR}(x,v)$, it suffices by \eqref{mathcal Z is zeta} to show that
\begin{eqnarray}
\label{equivalent imaginary bound}
\im(\psi(x,v)) \geq C|{\mathcal Z}_c^{i\mR}(x,v)|^2.
\end{eqnarray}

To this end we follow the presentation of Lemma 7.7.8 in \cite[vol.\ I]{Hormander}. We simplify notation by $(z,\eta) = w$ and $(x,v) = y$, where both $w$ and $y$ are real. We have that
\[
\Phi(z,\eta; x,v) = \Phi(w;y) = \psi(y) + \sum_{|\alpha| =2}\frac{H_\alpha(y)}{\alpha!} (w- w(y))^\alpha + O((w-w(y))^3).
\]
Here $w(y)$ denotes the critical point $w(y) = ({\mathcal Z}_c(x,v), \eta_c(x,v))$ (recall that the variable $y$ denotes $(x,v)$). We remark that while $w$ is real, $w(y)$ can be complex and that 
\begin{eqnarray}
\label{im(w(y)) vanishes}
\lim\limits_{y\to (\pi(\chi(v)), v)}\im(w(y)) = 0
\end{eqnarray}
for any $v\in \mathcal V$ by Lemma \ref{crit of Phi in small nbhd}.
Note the identity
\begin{eqnarray}
\label{second term is hessian}
H_\alpha(y) = \partial_{w}^\alpha\Phi(w(y);y),\qquad |\alpha | = 2.
\end{eqnarray}

For real valued $w \in \mR^{N} \times \mR^{n''}$, the definition of $\Phi$ gives that $\im(\Phi(w;y))\geq0$, so
\[
0\leq \im(\psi(y)) + \im\left(\sum_{|\alpha| =2}\frac{H_\alpha(y)}{\alpha!} (w- w(y))^\alpha + O((w-w(y))^3)\right).
\]
Now choose $w = \re{(w(y))} - t{|\im(w(y))|}$ with $t\in \mR^{N+n''}$ and $|t|\leq 1$. We get
\[
|\im(w(y))|^2\left(- \im\left(\sum_{|\alpha| =2}\frac{H_\alpha(y)}{\alpha!} \left(t +i\frac{\im(w(y))}{|\im(w(y))|}\right)^\alpha + O(|\im(w(y))|)\right) \right)\leq \im(\psi(y)).
\]
Since the inequality holds for all $|t|\leq1$ we can take the supremum to obtain that 
\[
|\im(w(y))|^2\sup_{|t|\leq1}\left(- \im\left(\sum_{|\alpha| =2}\frac{H_\alpha(y)}{\alpha!} \left(t +i\frac{\im(w(y))}{|\im(w(y))|}\right)^\alpha + O(|\im(w(y))|)\right) \right)\leq \im(\psi(y)).
\]
Thus, by \eqref{im(w(y)) vanishes}, proving \eqref{equivalent imaginary bound} amounts to showing that 
\begin{equation} \label{sup_bound}
\sup_{|t|\leq1}\left(- \im\left(\sum_{|\alpha| =2}\frac{H_\alpha(y)}{\alpha!} \left(t +i\omega\right)^\alpha \right)\right) \geq c>0
\end{equation}
uniformly for $y\in U \times \mathcal V$ and $\omega \in \mR^{N}\times \mR^{n''}$ with $|\omega| = 1$.

Using continuity and compactness, it suffices to prove \eqref{sup_bound} when $y = \hat y = (\hat x, \hat v)$. Using \eqref{second term is hessian}, we see that this can be written as
$$ \sup_{|t|\leq1}  - \im \left( (t+i\omega)\cdot H(\hat y)(t+i\omega)\right)\geq c>0$$
where $H(\hat y)$ is the Hessian of the map $w = (z,\eta) \mapsto \Phi(z,\eta; x,v)$ evaluated at 
$$ w(\hat y) = ({\mathcal Z}_c(\hat x,\hat v), \eta_c(\hat x,\hat v))) = (\hat v_1, \hat u_2'')$$
 (recall that the variable $y$ is shorthand for the variables $(x,v)$ and $w(y)$ denotes $({\mathcal Z}_c(x,v), \eta_c(x,v))$). Lemma \ref{main crit pt of Phi} states that $H(\hat y)$ is a symmetric nondegenerate matrix whose imaginary part is positive semi-definite. Lemma 7.7.9 in \cite[Vol.\ I]{Hormander} then gives the desired uniform lower bound.
\end{proof}

\begin{Corollary}
\label{estimating the square}
There are constants $c_1, c_2 >0$ such that for all $(x,v) \in U \times \mathcal V$, the following estimate holds:
\[
\re \left[ (z(\zeta_c'(x,v),x) - v_1)^2 \right]\geq c_1 |x - \pi(\chi(v))|^2 - c_2 \im (\psi(x,v)).
\]
\end{Corollary}
\begin{proof}
Since $v_1$ is real, this is a direct consequence of Lemma \ref{coercive estimate} and Lemma \ref{bound the imaginary part}.
\end{proof}

\begin{Lemma}
\label{linear term} 
 For all $(x,v)\in U \times\mathcal V$ and for all $\epsilon>0$,
\[
- \im(z(\zeta'_c(x,v),x)\cdot v_2) \geq - C\epsilon |x-\pi(\chi(v))|^2 - C\epsilon^{-1} \im(\psi(x,v)) - C\epsilon^{-1} |x-\pi(\chi(v))|^3.
\]
The constant $C$ is uniform over $(x,v) \in U\times \mathcal V$ and independent of $\epsilon$.
\end{Lemma}
\begin{proof}

We again use \eqref{mathcal Z is zeta} to deduce that the estimate in the Lemma is equivalent to the same estimate for $-{\mathcal Z}_c(x,v)\cdot v_2$. To this end,  we write ${\mathcal Z}_c(x,v)= {\mathcal Z}_c^\mR(x,v) + i{\mathcal Z}_c^{i\mR}(x,v)$. Observe that for $v\in \mathcal V$, by Lemma \ref{real crit points} part (1)  and \eqref{mathcal Z is zeta}, 
\begin{eqnarray}
\label{mathcal Z U W}
{\mathcal Z}_c(\pi(\chi(v)),v) ={\mathcal Z}_c^\mR(\pi(\chi(v)),v) + i{\mathcal Z}_c^{i\mR}(\pi(\chi(v)),v) = v_1.
\end{eqnarray}

Now write $\phi$ in terms of its real and imaginary part as in \eqref{phi real and im} to obtain from \eqref{critical point eq for Phi}
$$\phi^{i\mR}({\mathcal Z}_c^\mR(x,v), {\mathcal Z}_c^{i\mR}(x,v), x') = 0.$$
Differentiating in $x$ we get 
\begin{multline}
\label{derivative once} 
0 = \left({\mathcal Z}_c^\mR\right)_x^T \phi^{i\mR}_{z^{\mR}}({\mathcal Z}_c^{\mR}(x,v), {\mathcal Z}_c^{i\mR}(x,v),x')^T \\
 + \left({\mathcal Z}_c^{i\mR}\right)_x^T \phi^{i\mR}_{z^{i\mR}}({\mathcal Z}_c^\mR(x,v), {\mathcal Z}_c^{i\mR}(x,v), x')^T + \phi^{i\mR}_{x}({\mathcal Z}_c^\mR(x,v), {\mathcal Z}_c^{i\mR}(x,v), x')^T.
\end{multline}
If we denote $\chi(v) = (u_1(v), u_2(v))$ and write $u_2(v) = (u_2(v)' , u_2(v)'') \in \mR^{n' + n''}$, then by \eqref{canonical relation characterization} we have that 
\begin{eqnarray}
\label{v2 equals to}
v_2 =  \phi_z(v_1, u_1(v)')^T u_2(v)''
\end{eqnarray}
We multiply \eqref{derivative once} by $u_2(v)''$ to get 
\begin{multline}
\label{first derivative with v2}
0 = \left({{\mathcal Z}_c^\mR}\right)_x^T \phi^{i\mR}_{z^\mR}({\mathcal Z}_c^\mR(x,v), {\mathcal Z}_c^{i\mR}(x,v),x')^T u_2(v)'' \\
 +\left( {\mathcal Z}_c^{i\mR}\right)_x^T \phi^{i\mR}_{z^{i\mR}}({\mathcal Z}_c^\mR(x,v), {\mathcal Z}_c^{i\mR}(x,v), x')^T u_2(v)'' + \phi^{i\mR}_{x}({\mathcal Z}_c^\mR(x,v), {\mathcal Z}_c^{i\mR}(x,v), x')^T u_2(v)''.
\end{multline}
In the above expression both $\left({\mathcal Z}_c^\mR\right)_x$ and $\left({\mathcal Z}_c^{i\mR}\right)_x$ still depend on $x$. Now we set $x = \pi(\chi(v)) = u_1(v)$ and use \eqref{mathcal Z U W}. Equation \eqref{first derivative with v2} becomes
\begin{eqnarray}
\label{derivative multiply u2''}
0 &=&\left( {\mathcal Z}_c^\mR\right)_x^T \phi^{i\mR}_{z^\mR}(v_1, 0,u_1(v)')^T u_2(v)'' + \left({\mathcal Z}_c^{i\mR}\right)_x^T \phi^{i\mR}_{z^{i\mR}}(v_1, 0, u_1(v)')^T u_2(v)''\\\nonumber&+& \phi^{i\mR}_x(v_1, 0, u_1(v)')^T u_2(v)''.
\end{eqnarray}
Now use \eqref{phi im vanishes} with $x = u_1(v)$ to obtain that 
\begin{eqnarray}
\label{Zcx vanish}
0=\left({\mathcal Z}_c^{i\mR}\right)_x^T \phi^{i\mR}_{z^{i\mR}}(v_1, 0, u_1(v)')^T u_2(v)''.
\end{eqnarray}
Since $\phi$ is holomorphic, the Cauchy-Riemann equations give $\phi^{i\mR}_{z^{i\mR}}(v_1, 0, u_1(v)') = \phi^{\mR}_{z^\mR}(v_1, 0, u_1(v)')$. Thus we get from \eqref{phi im vanishes} and \eqref{v2 equals to} that when $x = u_1(v)$,
\begin{eqnarray}
\label{CR consequence}
\phi^{i\mR}_{z^{i\mR}}(v_1, 0, u_1(v)')^T u_2(v)'' = \phi^{\mR}_{z^\mR}(v_1, 0, u_1(v)')^T u_2(v)'' =   \phi_z(v_1, u_1(v)')^T u_2(v)''  =  v_2.
\end{eqnarray}
Combine \eqref{Zcx vanish} and \eqref{CR consequence} to get 
\begin{eqnarray}
\label{first order vanishing}
0=
\left({\mathcal Z}_c^{i\mR}\right)_x^T \mid_{x = u_1(v)} v_2 = \partial_x\left({\mathcal Z}_c^{i\mR}\cdot v_2\right)\mid_{x = u_1(v)}.
\end{eqnarray}

We now differentiate \eqref{first derivative with v2} with respect to $x$ and set $x  = u_1(v)$ to get
\begin{eqnarray}
\label{second order small}
 0 =\left( \left({\mathcal Z}_c^{i\mR}\right)_x^T A + B\left( {\mathcal Z}_c^{i\mR}\right)_x + \partial_x^2 ({\mathcal Z}_c^{i\mR}(x,v)\cdot v_2)\right)|_{x = u_1(v)}
 \end{eqnarray}
for some matrices $A$ and $B$ depending on $v$. Indeed, differentiating the first term of \eqref{first derivative with v2} with respect to $x$ yields an expression of the form $B\left( {\mathcal Z}_c^{i\mR}\right)_x $ by using \eqref{phi im vanishes} multiple times. Differentiating the second term of \eqref{first derivative with v2} with respect to $x$ and setting $x = u_1(v)$ yields an expression of the form $\left(\left( {\mathcal Z}_c^{i\mR}\right)_x ^T A+ \partial_x^2 ({\mathcal Z}_c^{i\mR}(x,v)\cdot v_2)\right)|_{x = u_1(v)}$ due to \eqref{CR consequence} and Lemma \ref{real crit points} part (1). Finally, differentiate the third term of \eqref{first derivative with v2} with respect to $x$ and setting $x = u_1(v)$ yields an expression of the form $B\left( {\mathcal Z}_c^{i\mR}\right)_x $ by \eqref{phi im vanishes}.

We now Taylor expand using \eqref{first order vanishing}, \eqref{second order small}, and the fact that ${\mathcal Z}_c^{i\mR}(u_1(v), v)) = 0$ to get
\begin{align}
|\im({\mathcal Z}_c(x,v) \cdot v_2)| &= |{\mathcal Z}_c^{i\mR}(x,v) \cdot v_2| \notag \\
 &\leq |\langle (x- u_1(v)), \left(\left({\mathcal Z}_c^{i\mR}\right)_x^T A + B\left( {\mathcal Z}_c^{i\mR}\right)_x\right)(x - u_1(v))\rangle| + C|x- u_1(v)|^3 \notag \\
 &\leq \epsilon^{-1}\left|\left( {\mathcal Z}_c^{i\mR}\right)_x(x-u_1(v))\right|^2 +\epsilon |x-u_1(v)|^2 + C |x-u_1(v)|^3 \notag \\
 &\leq \epsilon^{-1}\left|{\mathcal Z}_c^{i\mR}(x,v)\right|^2 +\epsilon |x-u_1(v)|^2 + C\epsilon^{-1}|x-u_1(v)|^3. \label{chain of ineq}
\end{align}
In the last inequality we used ${\mathcal Z}_c^{i\mR}(x,v) = \left({\mathcal Z}_c^{i\mR}\right)_x\left(u_1(v),v\right)(x-u_1(v) )+ O((x-u_1(v))^2)$. Recall 
\[
{\mathcal Z}_c^{i\mR}(x,v) = \im({\mathcal Z}_c(x,v) )= \im (z(\zeta_c'(x,v),x))
\]
by \eqref{mathcal Z is zeta}. Finally, applying the estimate  \eqref{zI bounded by psi} to the first term in the last line of \eqref{chain of ineq} gives 
\[
|\im ({\mathcal Z}_c(x,v) \cdot v_2)| \leq C \epsilon^{-1}\im(\psi(x,v))+\epsilon |x-u_1(v)|^2 + C\epsilon^{-1}|x-u_1(v)|^3
\]
and the proof is complete.
\end{proof}

\begin{Corollary}
\label{final coercive estimate}
Let $\hat x = \pi(\chi(\hat v))$. We may choose $\hat U\subset\subset U$ containing $\hat x$ and $\hat {\mathcal V}\subset\subset\mathcal V$ containing $\hat v$ small enough such that 
the phase function $\psi(x,v)$ given by \eqref{psi} satisfies
\[
\im(\psi(x,v) )\geq c|x-\pi(\chi(v))|^2
\]
for all $(x,v) \in \hat U\times \hat {\mathcal V}$.
\end{Corollary}
\begin{proof}
By Corollary \ref{estimating the square} we have
\[
\frac{1}{2}\re\left[ (z(\zeta'_c(x,v),x) - v_1)^2 \right]\geq c|x - \pi(\chi(v))|^2 - C\im (\psi(x,v)).
\]
Now choose $\epsilon>0$ sufficiently small in the estimate of Lemma \ref{linear term} and add it to the above inequality we get 
\begin{align*}
\im(\psi(x,v)) &= -\im(z(\zeta'_c(x,v), x) \cdot v_2) + \frac{1}{2}\re \left[ (z(\zeta'_c(x,v),x) - v_1)^2\right] \\
 &\geq c|x - \pi(\chi(v))|^2 - C\epsilon^{-1}\im (\psi(x,v)) - C\epsilon^{-1} |x-\pi(\chi(v))|^3.
\end{align*}
For this choice of $\epsilon$, we now choose $U\times \mathcal V$ sufficiently small so that 
\[
C\epsilon^{-1}|x-\pi(\chi(v))| \leq \frac{c}{2}.
\]
for all $(x,v)\in U \times \mathcal V$. We then have the desired estimate.\end{proof}

By Lemma \ref{crit of Phi in small nbhd} the function $\psi$ defined in \eqref{psi} satisfies
\begin{eqnarray}
\label{psi relation}
\psi(x,v) = \Phi({\mathcal Z}_c(x,v), \eta_c(x,v);x,v) = -{\mathcal Z}_c(x,v)\cdot v_2 + \frac{i({\mathcal Z}_c(x,v) - v_1)^2}{2}.
\end{eqnarray}
We can now give the proof of the main technical result.

\begin{proof}[Proof of Proposition \ref{good phase}]
The existence and uniqueness of the critical point $\zeta'_c(x,v)$ is stated in Corollary \ref{complex critical points}.

Property (1) comes from the definition of $\psi$ for real valued $(x,v) \in U\times \mathcal V$ and extends by analyticity to $(x,v)\in  U_{\mC} \times { \mathcal V}_{\mC}$.

For property (2), we first check it for $v\in {\mathcal V}$ real valued and then use analyticity to extend the identity to complex valued $v\in  {\mathcal V}_{\mC}$. First, write $\chi(v) = (u_1(v), u_2(v))$. 
Use \eqref{critical point eq for Phi} to get
\[\partial_x\left(\phi({\mathcal Z}_c(x,v),x') - x''\right) = 0.
\]
Expanding the derivatives and setting $x = u_1(v)$, so that ${\mathcal Z}_c(u_1(v),v) = v_1$, we get
\[
\partial_x {\mathcal Z}_c(x,v)^T |_{x= u_1(v)}\phi_z^T(v_1,u_1(v)')  = \begin{pmatrix}- \phi_{x'}^T(v_1, u_1(v)')  \\ I_{n''\times n''}\end{pmatrix}.
\]
Now multiply both sides of the above equality by $-u_2(v)''$. The identity in \eqref{canonical relation characterization} then gives
\[
\partial_x \left({\mathcal Z}_c(x, v)\cdot v_2\right) |_{x = \pi(\chi(v))} =-u_2(v).
\]
Differentiate the expression \eqref{psi relation} and use the above identity to obtain $d_x\psi(\pi(\chi(v)),v) = u_2(v)$. This is property (2).
Property (3) is Corollary \ref{final coercive estimate} if we choose $\hat U$ and $\hat {\mathcal V}$ as in Corollary \ref{final coercive estimate}.
\end{proof}

{
We conclude this section with the following remark outlining how one can generalize the argument above to amplitudes $a(z,x,\eta)$ which also depend on $\eta$.

\begin{Remark}
\label{general amplitude}
Let $T$ be an FIO with Schwartz kernel given by 
\[
T(z,x) := \int_{\mR^{n''}} e^{i(\phi(z,x') - x'') \cdot \eta}a(z,x,\eta) \,d\eta
\]
where $a(z,x,\eta)$ is a classical analytic symbol. In this case, instead of first performing the $\eta$ integral as in the proof of Proposition \ref{prop: Klambda integral}, we consider the $\eta$ and $z$ integral jointly in the expression \eqref{Klambda integral}. We split this integral into two parts:
\begin{eqnarray}\label{K split}
\nonumber
K_{\lambda}(x,v) &=& c_{n,N} \lambda^{\frac{3N}{4} + n''} \int\limits_{z\in V} \int\limits_{|\eta|\leq R}  e^{i\lambda\Phi(z,\eta; x,v)}a(z,x,\eta) \,d\eta \,dz\\
&&+ c_{n,N} \lambda^{\frac{3N}{4} + n''} \int\limits_{z\in V} \int\limits_{|\eta|\geq R}  e^{i\lambda\Phi(z,\eta; x,v)}a(z,x,\eta) \,d\eta \,dz
\end{eqnarray}
where $\Phi(z,\eta;x,v) = (\phi(z,x') - x'')\cdot \eta - z\cdot v_2 + i\frac{|z-v_1|^2}{2}$. Here $R$ is chosen so that for all $(x,v)$ near $(\hat x, \hat v)$,
\begin{eqnarray}
\label{choice or R>0}
\text{there is no } (z,\eta) \in V \times \mathbb R^{n''} \setminus { B_R(0)}\ {\rm s.t.}\ \nabla_{z,\eta} \Phi(z,\eta;x,v) = 0.
\end{eqnarray}
One applies stationary phase (see Remark 2.9 and 2.10 of \cite{sjostrand}) to get that
\begin{eqnarray}
\label{main term in K}
c_{n,N} \lambda^{\frac{3N}{4} + n''} \int\limits_{z\in V} \int\limits_{|\eta|\leq R}  e^{i\lambda\Phi(z,\eta; x,v)}a(z,x,\eta) \,d\eta \,dz \sim \lambda^{\ell}e^{i\lambda\psi(x, v_1, v_2)}A(x,v_1, v_2;\lambda)
\end{eqnarray}
for some $\ell \in \mathbb N$.

Since no critical points of the phase reside in the non-compact part of \eqref{K split} we can write for any $k\in \mathbb N$
$$e^{i\lambda\Phi} = \lambda^{-k} \left(\frac{\nabla \bar\Phi\cdot \nabla}{|\nabla \Phi|^2}\right)^ke^{i\lambda \Phi} := \lambda^{-k} (\mathcal L_\Phi)^k e^{i\lambda\Phi}$$
Applying integration by parts to the non-compact part of \eqref{K split} gives that
\begin{eqnarray}
\label{3 terms of remainder}
&& \int\limits_{z\in V} \int\limits_{|\eta|\geq R}  e^{i\lambda\Phi(z,\eta; x,v)}a(z,x,\eta) \,d\eta \,dz \\\nonumber
&=& \lambda^{-k}\int\limits_{z\in V} \int\limits_{|\eta|\geq R}  e^{i\lambda\Phi(z,\eta; x,v)}(\mathcal L^*)^ka(z,x,\eta) \,d\eta \,dz\\\nonumber
&+& \sum\limits_{j=1}^{k}\lambda^{-j}\int\limits_{z\in \partial V} \int\limits_{|\eta|\geq R} e^{i\lambda\Phi(z,\eta; x,v)}\frac{\nabla_z \bar\Phi\cdot \nu_z}{|\nabla \Phi|^2} ({\mathcal L}^*)^{j-1}a(z,x,\eta) d\eta dz\\\nonumber
&+&\sum\limits_{j=1}^k \lambda^{-j}\int\limits_{z\in V} \int\limits_{|\eta|= R}  e^{i\lambda\Phi(z,\eta; x,v)} \frac{\nabla_\eta \bar\Phi\cdot \eta}{R|\nabla \Phi|^2} ({\mathcal L}^*)^{j-1}a(z,x,\eta) d\eta dz
\end{eqnarray}
We now need to verify that each of the terms in \eqref{3 terms of remainder} are exponentially decaying as $\lambda \to \infty$. The first term involving the interior integration can be shown to have exponential decay by using standard pseudoanalytic estimates on the symbol $a(z,x,\eta)$ and by choosing $k$ appropriately as a function of $\lambda$.

For the first boundary term, we observe that for $v_1$ bounded away from the boundary $\partial V$, the integral decays exponentially as $\lambda\to\infty$.

The second boundary term requires more care. When $x = \pi(\chi(v))$ we can see from relation \eqref{canonical relation characterization} and condition \eqref{choice or R>0},that $\iota_{V\times \partial B_R(0)}^* \Phi$ has no critical points. This means that we can integrate by parts in the last term of \eqref{3 terms of remainder} and estimate this term in the same way as we did for the first term of \eqref{3 terms of remainder}. 

We have argued that all three terms in \eqref{3 terms of remainder} decay exponentially fast. This shows that the second term of \eqref{K split} decays exponentially quickly. This combined with \eqref{main term in K} shows that
$$K_\lambda(x,v) \sim \lambda^{\ell}e^{i\lambda\psi(x, v_1, v_2)}A(x,v_1, v_2;\lambda)$$
modulo exponentially decaying terms.
\end{Remark}}

\section{Applications to uniqueness theorems in integral geometry} \label{sec_appl}

We will now discuss the proofs of Theorems \ref{thm_local_uniqueness}--\ref{thm_b_k} in the introduction. These follow more or less directly from the analytic regularity result in Theorem \ref{global elliptic regularity intro}, the analysis of the Bolker condition in Sections \ref{injectivity of piL} and \ref{sec_ray}, and the microlocal analytic continuation result from \cite[Theorem 8.5.6']{Hormander} which is rephrased as follows.

\begin{Theorem} \label{thm_mac}
Let $\mx$ be an analytic manifold, and let $\Sigma$ be a $C^2$ hypersurface through $x_0 \in \mx$ with conormal $\nu_0$ at $x_0$. If $f \in \mathcal{D}'(\mx)$ satisfies $f=0$ on one side of $\Sigma$ near $x_0$ and if $(x_0,\nu_0) \notin \WF_a(f)$ or $(x_0,-\nu_0) \notin \WF_a(f)$, then $f=0$ near $x_0$.
\end{Theorem}

Theorem \ref{thm_local_uniqueness} in the introduction is an immediate consequence of Theorem \ref{global elliptic regularity intro} and Theorem \ref{thm_mac}.

We proceed to the weighted geodesic X-ray transform $R$ as defined in Section \ref{sec_ex_xray}. In the analytic case, Theorem \ref{thm_mac} allows us to upgrade Proposition \ref{geodesic bolker} to the following uniqueness result that is already contained in \cite{StefanovUhlmann}. As mentioned in the introduction, it might be possible to remove the assumptions that the geodesic has no self-intersections or tangential intersections by using suitable extension procedures.

\begin{Proposition} \label{thm_xray_uniqueness}
Let $(M,g)$ be an analytic manifold with smooth boundary such that $g$ is analytic up to $\p M$. Let $z_0 \in \p_+ SM$ satisfy $\tau_+(z_0) < \infty$, and assume that the geodesic $x_{z_0}: [0,\tau_+(z_0)] \to M$ does not self-intersect, has no conjugate points, meets $\p M$ transversally at the endpoints, and otherwise stays in $\mi$. Let $\kappa$ be analytic and nowhere vanishing in $\mG \times \mi$, and assume that $\Sigma$ is a $C^2$ hypersurface such that $x_{z_0}$ is tangent to $\Sigma$ at $x_0 \in \mi$. If $Rf(z) = 0$ for $z$ near $z_0$ and if $f$ vanishes on one side of $\Sigma$ near $x_0$, then $f = 0$ near $x_0$.
\end{Proposition}

We can now prove that the geodesic X-ray transform, possibly with analytic weight, is injective on compact nontrapping strictly convex two-dimensional manifolds in the analytic case.

\begin{proof}[Proof of Theorem \ref{thm_xray_twodim}]
Suppose that $f \in C(M)$ and $Rf = 0$. Embed $M$ in a slightly larger analytic nontrapping manifold $M_1$ with strictly convex boundary, extend $g$ analytically to $M_1$, and extend $f$ by zero to $M_1$. Note that since the ray transform of $f$ vanishes and $\p M$ is strictly convex, we have $f|_{\partial M} = 0$ by looking at short geodesics, so the zero extension is continuous in $M_1$. Then we have $f \in \mathcal{E}'(\mi_1)$ and $\supp(f) \subset M$.

Since $\partial M$ is strictly convex, we can apply Proposition \ref{thm_xray_uniqueness} near tangential geodesics to $\partial M$ to conclude that $f$ must vanish near $\partial M$. We now use a layer stripping argument: by \cite{BeteluGulliverLittman}, the manifold $M \setminus \{p\}$ for some $p \in M^{\mathrm{int}}$ is foliated by strictly convex hypersurfaces $\Gamma_s$ for $s \in (0,1]$ so that $\Gamma_1 = \partial M$ and $\Gamma_s \to \{p\}$ as $s \to 0$. Define  
\[
I = \{ t \in (0,1] \,:\, f = 0 \text{ in } \cup_{s \in [t,1]}\Gamma_s \}.
\]
By the argument above we have $1 \in I$. Moreover, if $t \in I$, then applying Theorem \ref{thm_xray_uniqueness} for geodesics tangent to $\Gamma_t$ shows that some neighborhood of $t$ is in $I$. Thus $I$ is open and closed, so $I = (0,1]$ by connectedness. This shows that $f=0$ in $M \setminus \{p\}$ and that $f \equiv 0$ by continuity.
\end{proof}

Next we consider the setting in Section \ref{subsect: null bich transform} and let $R$ be the null bicharacteristic ray transform with nowhere vanishing weight. When all the objects involved are analytic, Theorem \ref{thm_intro_support_nbrt} follows immediately from Proposition \ref{prop_null}, Theorem \ref{global elliptic regularity intro} and Theorem \ref{thm_mac}. We now use this to prove the uniqueness result involving a strictly pseudoconvex foliation. See also \cite{PSUZ} for further analysis of foliation conditions.

\begin{proof}[Proof of Theorem \ref{thm_bichar_foliation}]
Define 
\[
I = \{ t \in (0,1] \mid f = 0 \text{ in an open set containing } \cup_{s \in [t,1]}\Gamma_s \}.
\]
By assumption (a) one has $1 \in I$. By definition $I$ is open. To show that $I$ is closed, we suppose that $t_j \in I$ and $t_j \to t \in (0,1]$. Then $f = 0$ in $\cup_{s \in (t,1)} \Gamma_s$ and in particular $f$ vanishes on one side of $\Gamma_t$. Let $x \in \Gamma_t$ and let $\nu = dF(x)$ be conormal to $\Gamma_t$ at $x$. By assumption (b), there is $\xi \in \Xi_{x}$ such that $\nu \perp Y^h(x,\xi)$ and $\nu \nparallel \xi$. Then $z = (x,\xi)$ is such that $x_{z}$ is tangent to $\Gamma_t$ at $x$. Since $p(x,\xi) = 0$ and $\{ p, F \}(x,\xi) = 0$, assumption (c) can be rephrased as 
\begin{equation} \label{f_convex}
\p_s^2[F(x_{z}(s))] \big|_{s=0} = H_p^2 F(x,\xi) = \{ p, \{ p, F \} \}(x,\xi) > 0.
\end{equation}
Let $\tilde{M} = \cup_{s \in (0,t+\delta]} \Gamma_s$. If $\delta > 0$ is sufficiently small, the segment of the curve $x_{z}$ that lies in $\tilde{M}$ is very short by \eqref{f_convex} and does not have conjugate points by Lemma \ref{lem_ncp_short}. It also meets $\p \tilde{M}$ transversally at the endpoints and does not self-intersect. We can now invoke Theorem \ref{thm_intro_support_nbrt} in $\tilde{M}$ to conclude that $f=0$ near $x$. Repeating this argument for all $x \in \Gamma_t$ shows that $t \in I$. Consequently $I$ is closed, and by connectedness $I=(0,1]$.
\end{proof}

Finally we consider the transform introduced in Section \ref{subsubsec_k}, repeating some of the statements given there. Let $\mx$ be a manifold without boundary. We will study a transform that integrates a function $f$ over submanifolds given by 
 \[
 G_{\theta,s} = \{ x \in \mx \mid b(x,\theta) = s \}.
 \]
Here $b: \mx \times V' \to \mR^k$ is a smooth function, $\theta \in V'$ and $s \in V''$, where $V'$ and $V''$ are open subsets of $\mR^m$ and $\mR^k$, respectively. Here $m \geq 1$ and $1 \leq k \leq n-1$. Comparing with Lemma \ref{lem_df_b}, we have relabeled variables so that $z = (z',z'') = (\theta,s)$ and $m=N-n''$, $k = n''$. We also assume that we are considering integrals over $G_{\theta,s}$ with $(\theta,s)$ close to a fixed $(\theta_0,s_0)$, so we can indeed work in local coordinates and take $\mG = V' \times V''$. Moreover, if we let $U \subset \mx$ be a neighborhood of $G_{\theta_0,s_0}$ such that $G_{\theta,s} \subset U$ for $(\theta,s) \in \mG$, it is enough to have everything defined for $x \in U$.
 
We assume that $b_x(x,\theta)$ is surjective for $x \in U$ and $\theta \in V'$. Then Lemma \ref{lem_df_b} ensures that $Z = \{ (\theta,s,x) \mid b(x,\theta)=s \}$ is a double fibration. If we fix orientation forms on $Z$ and $\mG$ (we can use the Lebesgue measure on $\mG \subset \mR^N$), we obtain orientation forms on each $G_{\theta,s}$ and consider the weighted double fibration transform 
\[
Rf(\theta,s) = \int_{G_{\theta,s}} \kappa(\theta,s,x) f(x) \,d\omega_{G_{\theta,s}}(x), \qquad (\theta,s) \in \mG,
\]
where $\kappa$ is smooth and nowhere vanishing.

We next study the Bolker condition. By Lemma \ref{lem_df_b}, the canonical relation of $R$ may be written as 
\begin{equation} \label{c_b_form}
C = \{ (z, \zeta, x, -B(z,x) \zeta) \mid b(x,\theta) = s, \ \zeta \in N_z^* H_x \setminus 0 \}
\end{equation}
where $z = (\theta,s) = (\theta, b(x,\theta))$, and $\zeta \in N_z^* H_x \subset \mR^N$ has the form $\zeta = (-b_{\theta}(x,\theta)^T \zeta'', \zeta'')$. From Lemma \ref{lemma_bolker_local_char} we see that $d\pi_L$ is injective if for any $x \in U$, $\theta \in V'$ and $\zeta'' \in \mR^{n''} \setminus 0$, the linear map  
\begin{equation*} 
\text{$\left( b_x(x,\theta)^T, \ \  \p_{\theta} (b_x(x,\theta)^T \zeta'') \right)$ is surjective.}
\end{equation*}
We also see from \eqref{c_b_form} that $\pi_L$ is injective if for any $\theta \in V'$ and $\zeta'' \in \mR^{n''} \setminus 0$, the map 
\begin{equation*} 
x \mapsto (b(x,\theta), b_{\theta}(x,\theta)^T \zeta'') \text{ is injective on $U$.}
\end{equation*}
Thus the Bolker condition indeed amounts to \eqref{cdk_bolker1} and \eqref{cdk_bolker2}.

Using these facts, the combination of Theorem \ref{global elliptic regularity intro} and Theorem \ref{thm_mac} yields Theorem \ref{thm_b_k} in the introduction.

\end{document}